\documentclass[final]{siamltex}
\usepackage{amsfonts}
\usepackage{graphics}
\usepackage{caption2}
\usepackage{amssymb}
\usepackage{psfrag}
\usepackage{amsmath}
\usepackage{graphicx}
\usepackage{multirow}
\usepackage{cite}
\usepackage{subfigure}
\usepackage[normalem]{ulem}
\usepackage{color}

\allowdisplaybreaks[4]
\newcommand{\bm}[1]{\mbox{\boldmath{$#1$}}}

\def\x{{\bf   x}}

\def\u{{\bf   u}}

\def\q{{\bf   q}}

\def\I{{\bf   I}}

\def\n{{\bf   n}}

\def\({\left(}
\def\[{\left[}
\def\){\right)}
\def\]{\right]}
\def\div{\nabla\cdot }
\def\grad{\nabla }

\newtheorem{lem}{Lemma}
\newtheorem{thm}{Theorem}

\numberwithin{equation}{section}
\numberwithin{thm}{section}
\numberwithin{lem}{section}
\begin{document}
\title{Thermodynamically consistent   simulation of  nonisothermal  diffuse-interface two-phase flow  with  Peng-Robinson equation of state\thanks{This work is  supported by   National Natural Science Foundation of China (No.11301163),  and KAUST research fund to the
Computational Transport Phenomena Laboratory at KAUST.}}

\author{Jisheng Kou\thanks{School of Mathematics and Statistics, Hubei Engineering  University, Xiaogan 432000, Hubei, China. } \and Shuyu Sun\thanks{Corresponding author. Computational Transport Phenomena Laboratory, Division of Physical Science and Engineering,
King Abdullah University of Science and
Technology, Thuwal 23955-6900, Kingdom of Saudi Arabia.   Email: {\tt shuyu.sun@kaust.edu.sa}.}}

 \maketitle

\begin{abstract}
In this paper, we consider a diffuse-interface gas-liquid two-phase flow model with  inhomogeneous  temperatures, in which we employ the Peng-Robinson equation of state and the temperature-dependent influence parameter instead of the van der Waals equation of state and the constant influence parameter used in the existing models.  As a result, our model can characterize accurately the physical behaviors of numerous realistic gas-liquid fluids, especially hydrocarbons.  Furthermore,  we prove a relation associating     the pressure gradient with the gradients of temperature and  chemical potential, and thereby derive   a new formulation of the  momentum balance equation, which shows that  gradients of the chemical potential and temperature   become the primary driving force of the fluid motion.  It is rigorously proved that the new  formulations of the model obey the first and second laws of thermodynamics. 
To design efficient  numerical methods,  we prove that Helmholtz free energy density   is a concave function with respect to the temperature under certain physical conditions.  Based on the proposed modeling  formulations and the convex-concave  splitting of Helmholtz free energy density,  we propose a novel thermodynamically stable  numerical scheme.  We rigorously prove that  the proposed method  satisfies the first and second laws of thermodynamics. Finally, numerical tests are carried out to verify the effectiveness of the proposed simulation method.
\end{abstract}

%%%%%%%%%%%%%%%5
\section{Introduction}
%%%%%%%%%%%%%%%%
 Modeling and simulation of gas-liquid two-phase   flow has a wide range of applications in industrial  and scientific problems.  In the oil reservoir,   hydrocarbon is usually split into gas and liquid phases due to the effect of  temperature, pressure and gravity.  In the thermal enhanced oil recovery \cite{chen2006multiphase}, heat is introduced intentionally  to reduce the oil viscosity  or vaporize part of the oil for the purpose of decreasing the mobility  such that oil  flows more freely through the reservoir.  The heated oil may also vaporize and then form better oil once it  condenses. In  the natural world, gas-liquid two-phase   flow is also one of the commonest  phenomena, such as boiling, evaporation, and  condensation  \cite{liu2015liquid}. 
 %Since  shale gas reservoir  has become an increasingly important source of natural gas,  pore-scale modeling has attracted more attentions in recent years.  
 In this paper, we mainly focus on the diffuse-interface model of  the gas-liquid two-phase   flow at the pore scale, but which has indeed general formulations and can be applied for other applications.

In order  to describe a gas-liquid interface,  van der Waals introduced a gradient term in the Helmholtz free energy density, see \cite{Onuki2007PRE} and the references therein.   Korteweg developed the so-called Korteweg stress formulation induced by composition gradients, see \cite{liu2015liquid,Onuki2007PRE} and the references therein. From then on, the diffuse-interface models for  two-phase   fluid flow  have been extensively developed in the literature,  \cite{CahnHilliard1958,Abels2012TwoPhaseModel,Bao2012FEM} for instance.

In the traditional theories of phase transitions,  the temperature  is usually assumed to be constant. However, there exist many situations in which phase transitions are strongly influenced by an inhomogeneous temperature field, such as boiling, evaporation,   condensation and thermal enhanced oil recovery. 
To deal with such problems,  a diffuse-interface model accounting for variable temperatures has been developed in \cite{Onuki2005PRL,Onuki2007PRE} based on thermodynamical relations, and recently, \cite{liu2015liquid} proposed a  continuum mechanics modeling framework for liquid-vapor  flows using the thermodynamical laws.   Such models have been applied and extended to investigate the fluid problems with inhomogeneous temperatures \cite{Bueno2016liquid,Qian2016heatflow} for instance. In such models,  the Helmholtz free energy density usually  consists of two contributions: one results from the bulk phase of a fluid, which can be formulated by  van der Waals  equation of state, and the other is the density gradient contribution on the two-phase interfaces.  Although the van der Waals  equation of state is popularly used in physics,  the Peng-Robinson equation of state \cite{Peng1976EOS} has more accuracy for numerous realistic gas-liquid fluids including N$_2$, CO$_2$,  and hydrocarbons; as a result, it has been extensively employed in petroleum and chemical industries.  In recent years, modeling and simulation of two-phase flow based on the Peng-Robinson equation of state   have become an attractive and challenging research topic   in the reservoir  and chemical engineering   
\cite{kousun2015CMA,qiaosun2014,kousun2015SISC,kousun2015CHE,kouandsun2016multiscale,kousun2016Flash,fan2017componentwise,mikyvska2015General,smejkal2017phase}. 
 In this paper, we will study a nonisothermal diffuse-interface model combining with the Peng-Robinson equation of state.   In practices,  the existing models   usually employ a constant influence parameter in the density gradient contribution of Helmholtz free energy density, even though   \cite{Onuki2007PRE} assumed that such  parameter can depend on the density.  However, this influence parameter   is generally  viewed to highly depend on the temperature \cite{miqueu2004modelling}. Here, we will adapt a realistic formulation of the influence parameter, which is a function of temperature being consistent with the Peng-Robinson equation of state.

The  models of \cite{Onuki2005PRL,Onuki2007PRE,liu2015liquid} use a  thermodynamic pressure, which is a function of the molar density and temperature.  However, the pressure has a complicate formulation, which causes  inconvenience in theoretical analysis and construction of numerical methods.  In this paper,  we will investigate  a relation between the  gradients of pressure, temperature and chemical potential, and from this, we can simplify  the modeling  equations, which  allow us  to conveniently prove the satisfaction of thermodynamical laws and to design  efficient numerical schemes.

 For numerical simulation of diffuse-interface models,  it
 is demanded that numerical schemes shall satisfy the laws of thermodynamics due to the physical processes obeying such laws.   More precisely speaking,  for the  motion of a fluid with inhomogeneous temperatures, the first law of thermodynamics   (i.e. the energy balance law) is a basic physical  principle, and once it is satisfied, we may obtain reliable and accurate results from   numerical simulations.  The second law of thermodynamics describes the entropy production of the realistic irreversible  processes. As shown in \cite{kou2017compositional}, for a system under a fixed temperature,  one  can derive a formulation of  entropy by the first law  of thermodynamics, and then  from the second law of thermodynamics, one can further derive the total (free) energy dissipation law, which is admitted in the phase-field model \cite{shen2015SIAM,shen2016JCP,Bao2012FEM}.  Thus,  a main challenge  in numerical simulation is to design efficient numerical schemes that still satisfy the  laws of thermodynamics.  However, there are too few such methods available in the literature due to the short development history and
 more complications of these problems.  A notable progress is that  a provably entropy-stable  numerical scheme was designed and analyzed in  \cite{liu2015liquid}, which   is based fundamentally on the concept of functional entropy variables.
 
 It is different from the numerical schemes developed in \cite{liu2015liquid} that our proposed numerical schemes will be designed using the convex-concave splitting of Helmholtz free energy density.  For phase-field models, there are a lot of efforts on the developments of energy-dissipated schemes in the literature, \cite{shen2015SIAM,shen2016JCP} for instance, in which the convex-concave splitting of free energy functions is  a key and efficient technique. For the Peng-Robinson equation of state,  \cite{qiaosun2014} analyzed the convex-concave splitting of the Helmholtz free energy density with respect to molar density, but its convex-concave property with respect to the temperature (which is a key point for the problem considered in this work) is not explored yet.  
 
Another challenge in the considered modeling equations  is the strongly nonlinear, tightly mutual coupling relationship between molar density, temperature and velocity.  The resulted discrete equations  in \cite{liu2015liquid}  is still strongly nonlinear and fully coupled.  The approach of constructing an auxiliary velocity can be used to reduce the tight coupling relation between the phase function and velocity in phase filed simulation \cite{shen2015SIAM}.   In this paper,  we extend  this approach  to the considered problem, and we define   an  auxiliary velocity, which depends on molar density and temperature. As a result,  the nonlinear coupling relation in the proposed numerical scheme is alleviated to a great extent.  We further propose a decoupled, linearized iterative method for solving the discrete equations, which satisfies the discrete first law of thermodynamics.

The key contributions of our work are listed as below:

(1)  The Peng-Robinson equation of state  is employed to replace the van der Waals equation of state in the existing nonisothermal diffuse-interface two-phase flow models.  Moreover, we use a realistic formulation for the influence parameter in the gradient contribution of Helmholtz free energy density, which is a function of temperature instead of taking a constant as in the existing models.  As a result, this modified model can characterize accurately the physical behaviors of numerous realistic gas-liquid fluids including N$_2$, CO$_2$ and hydrocarbons etc. 

(2) A relation associating     the pressure gradient with the gradients of temperature and  chemical potential is proved,  and from this, we propose   a new formulation of the  momentum balance equation, which demonstrates that  chemical potential and temperature gradients  become the primary driving force of the fluid motion. The energy balance equation is also simplified.  With the new  formulations, it is convenient to  prove that the model obeys the first and second laws of thermodynamics. 

(3) We analyze the convex-concave splitting of  Helmholtz free energy density; in particular, we prove that its bulk contribution   is a concave function with respect to the temperature, and show that its gradient contribution is concave with respect to the temperature under certain conditions.

(4)  Based on the proposed modeling formulations,     combining the convex-concave  splitting of Helmholtz free energy density,  we propose a novel thermodynamically consistent   numerical scheme, in which   an auxiliary velocity is introduced to treat the coupling relations between molar density,  velocity and temperature.  We prove that  the proposed method rigorously satisfies the first and second laws of thermodynamics. 

Here, we note that  thermodynamical consistency of a model or a numerical method means that such model or such method obeys the first and second laws of thermodynamics.  Thermodynamical consistency is also called as thermodynamical stability.

The body  of this paper is organized as   follows.  In Section 2, we will  introduce the  thermodynamic formulations derived from Peng-Robinson equation of state,  the temperature-dependent influence parameter, and modeling equations of two-phase diffuse-interface flow with a variable temperature field.
 In Section 3,  the modeling equations are simplified with the help of a relation between the gradients of pressure, temperature and chemical potential;  subsequently, it is  proved that the simplified model obeys the laws of thermodynamics.  In Section 4, we propose  a thermodynamically consistent   numerical method based on an auxiliary velocity and the  convex-concave  splitting of Helmholtz free energy density, and we also prove that the proposed scheme satisfies the discrete laws of thermodynamics.
  In   Section 5,  numerical tests are carried out  to verify  effectiveness of the proposed method.   Finally,   some concluding remarks are provided in Section 6.

%%%%%%%%%%%%%%%%%%%
\section{Mathematical model}
%%%%%%%%%%%%%%%%%%%%
In this section, we first describe the expressions of the thermodynamical variables and    temperature-dependent influence parameter, and subsequently we   formulate  the modeling equations of a diffuse-interface two-phase flow model with a variable temperature field.

\subsection{Formulations of thermodynamical variables}

We consider a pure substance fluid, and  let $n$ to denote the molar density of the substance.  We now present  the formulations of Helmholtz free energy density, entropy and internal energy, which are derived from Peng-Robinson equation of state \cite{Peng1976EOS,smejkal2017phase}. Let $T$ be the absolute  temperature. We denote  by $T_{c}$ and $P_{c}$  the   critical temperature and critical pressure, respectively, and let the reduced temperature be  $T_{r}=T/T_{c}$. 
Let $a$ and $b$ be the energy parameter and  the covolume, respectively, which are calculated as
\begin{eqnarray*}
   a(T)= 0.45724\frac{R^2T_{c}^2}{P_{c}}\[1+m(1-\sqrt{T_{r}})\]^2,~~~~b= 0.07780\frac{RT_{c}}{P_{c}},
\end{eqnarray*}
where $R$ is the ideal gas constant.
 The coefficient $m$ is calculated  by the following formulas
\begin{eqnarray*}
 m=0.37464 + 1.54226\omega-  0.26992\omega^2 ,~~\omega\leq0.49,
\end{eqnarray*}
\begin{eqnarray*}
 m=0.379642+1.485030\omega-0.164423\omega^2 +0.016666 \omega^3,~~\omega>0.49,
\end{eqnarray*}
where $\omega$ is the acentric factor.

 The   correlation coefficients $\alpha_i$  estimate the molar heat capacity of ideal gas at the constant pressure   as \cite{smejkal2017phase}
\begin{eqnarray}\label{eqHeatCapacity01}
    \psi_p(T)=\sum_{i=0}^3\alpha_iT^i.
\end{eqnarray}

The bulk Helmholtz free energy density, denoted by $f_{b}$,  is calculated as a sum of three contributions
%  \begin{eqnarray*}\label{eqHelmholtzEnergy_a0_01}
%    f_b^{\textnormal{ideal}}(n,T)&=& \vartheta_b(n,T)-Ts_b(n,T) \nonumber\\
%    &=&n\vartheta_0 +n\sum_{i=0}^3\alpha_i\frac{T^{i+1}-T_0^{i+1}}{i+1}-nR(T-T_0)\nonumber\\
%    &&
%    -nRT\ln\(1-bn\)-nRT\ln\(\frac{P_0}{nRT}\)-nT\int_{T_0}^T\frac{\psi_p(\xi)}{\xi}d\xi \nonumber\\
%    &&+\frac{a(T)n}{2\sqrt{2}b}\ln\(\frac{1+(1-\sqrt{2})b n}{1+(1+\sqrt{2})b n}\),
%\end{eqnarray*}
\begin{eqnarray*}\label{eqHelmholtzEnergy_a0_01}
    f_b(n,T)&=& f_b^{\textnormal{ideal}}(n,T) + f_b^{\textnormal{repulsion}}(n,T)+f_b^{\textnormal{attraction}}(n,T),
\end{eqnarray*}
where
\begin{eqnarray*}\label{eqHelmholtzEnergy_a0_01}
    f_b^{\textnormal{ideal}}(n,T)&=& n\vartheta_0 +n\sum_{i=0}^3\alpha_i\frac{T^{i+1}-T_0^{i+1}}{i+1}-nR(T-T_0)\nonumber\\
    &&-nRT\ln\(\frac{P_0}{nRT}\)-nT\int_{T_0}^T\frac{\psi_p(\xi)}{\xi}d\xi,
\end{eqnarray*}
\begin{eqnarray*}\label{eqHelmholtzEnergy_a0_02}
    f_b^{\textnormal{repulsion}}(n,T)=-nRT\ln\(1-bn\),
\end{eqnarray*}
\begin{eqnarray*}\label{eqHelmholtzEnergy_a0_03}
    f_b^{\textnormal{attraction}}(n,T)= \frac{a(T)n}{2\sqrt{2}b}\ln\(\frac{1+(1-\sqrt{2})b n}{1+(1+\sqrt{2})b n}\),
\end{eqnarray*}
where  $T_0=  298.15 $K, $ P_0 = 1 $bar, and $\vartheta_0=-2478.95687512 $ J$/$mol. We note that the ideal contribution $f_b^{\textnormal{ideal}}$ is indeed enriched by the heat capacity term.

The bulk internal energy, denoted by $\vartheta_b$,  is formulated as \cite{smejkal2017phase}
 \begin{eqnarray*}\label{eqinternalenergy}
    \vartheta_b(n,T)&=&  n\vartheta_0 +n\sum_{i=0}^3\alpha_i\frac{T^{i+1}-T_0^{i+1}}{i+1}-nR(T-T_0)\nonumber\\
    &&+\frac{n\(a(T)-Ta'(T)\)}{2\sqrt{2}b}\ln\(\frac{1+(1-\sqrt{2})b n}{1+(1+\sqrt{2})b n}\),
\end{eqnarray*}
where $a'(T)$ denotes the  derivative with respect to $T$.
We  denote by $s_b$ the bulk entropy and  express it as \cite{smejkal2017phase}
 \begin{eqnarray*}\label{eqinternalenergy}
    s_b(n,T)&=&  nR\ln\(1-bn\)+nR\ln\(\frac{P_0}{nRT}\)+n\int_{T_0}^T\frac{\psi_p(\xi)}{\xi}d\xi \nonumber\\
    &&-\frac{na'(T)}{2\sqrt{2}b}\ln\(\frac{1+(1-\sqrt{2})b n}{1+(1+\sqrt{2})b n}\).
\end{eqnarray*}

The  influence parameter  generally    relies  on the temperature but  independent of the   molar density.  We denote the  influence parameters by $c$, which is  given by \cite{miqueu2004modelling}
\begin{eqnarray}\label{eqDefinfluenceparameter}
 c(T)=a(T)b^{2/3}\[\beta_1(1-T_{r})+\beta_2\],
\end{eqnarray}
where  $a$ and $b$ are the energy parameter and  the covolume respectively and the coefficients $\beta_1$ and $\beta_2$ are calculated as
\begin{eqnarray*}
 \beta_1=-\frac{10^{-16}}{1.2326+1.3757\omega},~~~~\beta_2=\frac{10^{-16}}{0.9051+1.5410\omega}.
\end{eqnarray*}

We now express the density gradient contribution to Helmholtz free energy density and denote it by $f_\grad$:
\begin{eqnarray}\label{eqDefHEP01}
    f_\grad=\frac{1}{2}c\grad n\cdot\grad n.
\end{eqnarray}
 The general   Helmholtz free energy density (denoted by $f$)  is  a sum of   two contributions:
\begin{eqnarray}\label{eqDefHEP02}
    f=f_b+f_\grad.
\end{eqnarray}
  By thermodynamical relations \cite{Groot2015NET,firoozabadi1999thermodynamics}, the entropy density (denoted by $s$) and chemical potential (denoted by $\mu$) can be expressed as
 \begin{eqnarray}\label{eqHelmholtzDensityDeri}
s=-\(\frac{\delta  f(n,T)}{\delta T}\)_n,~~~~\mu=\(\frac{\delta  f(n,T)}{\delta n}\)_T,
 \end{eqnarray}
 where  $\frac{\delta f}{\delta T}$ and $\frac{\delta f}{\delta n}$ represent the variational derivatives. 
We further define 
\begin{eqnarray}\label{eqDefDfDT}
    \gamma &=& \(\frac{\delta f(n,T)}{\delta T}\)_n=\gamma_b+\gamma_\grad,
\end{eqnarray}
where $\gamma_b=\(\frac{\partial  f_b(n,T)}{\partial T}\)_{n}$ and $$\gamma_\grad = \(\frac{\delta f_\grad}{\delta T}\)_n=\frac{1}{2} c'(T)|\nabla n|^2.$$
 We denote the entropy contribution of the bulk fluid by $s_b$ and the entropy gradient contribution  by $s_\grad$. The thermodynamical relation yields 
$$s=s_b+s_\grad=-\gamma_b-\gamma_\grad.$$
 Let the bulk chemical potential be $\mu_b=\(\frac{\partial  f_b(n,T)}{\partial n}\)_{T}$. The general form of  chemical potential is expressed  as
\begin{eqnarray}\label{eqDefChPtl}
    \mu=\(\frac{\delta  f(n,T)}{\delta n}\)_{T}=\mu_b+\mu_\grad,
    \end{eqnarray}
where $\mu_\grad$ is the gradient contribution of chemical potential
\begin{eqnarray}\label{eqDefChPtlGrad}
    \mu_\grad=\(\frac{\delta  f_\grad(n,T)}{\delta n}\)_{T}=-\div c\grad{n}.
\end{eqnarray}

%%%%%%%%%%%%
\subsection{Model equations}
%%%%%%%%%%%%
 
We now describe the modeling equations based on the models of \cite{Onuki2005PRL,Onuki2007PRE,liu2015liquid},  but  the original van der Waals equation of state is replaced by the Peng-Robinson equation of state.  Moreover,  the temperature-dependent    influence parameter given in \eqref{eqDefinfluenceparameter} is adopt  instead of  constant parameters.

We denote   the mass density by $\rho$ as $\rho=nM_w$, where $M_w$ is the molar weight.  The fluid velocity is denoted by $\u$. The  law of mass conservation states
\begin{eqnarray}\label{eqMassConserve1C01}
\frac{\partial n}{\partial t}+\div(n\u)=0,
\end{eqnarray}
which is also reformulated  by a mass form
\begin{eqnarray}\label{eqMassConserve1C02}
\frac{\partial\rho}{\partial t}+\div(\rho\u)=0.
\end{eqnarray}
  The momentum balance  equation is expressed  as
\begin{eqnarray}\label{eqMomentumConserve01}
&& \frac{\partial(\rho\u)}{\partial t}+\div\(\rho\u\otimes\u\)=-\div\bm\sigma,
\end{eqnarray}
where  $\bm\sigma$ is the total stress.
Utilizing the mass conservation equation, we can reformulate  \eqref{eqMomentumConserve01} as
\begin{eqnarray}\label{eqMomentumConserve02}
&& \rho\(\frac{\partial\u}{\partial t}+\(\u\cdot\grad\)\u\)=-\div\bm\sigma .
\end{eqnarray}
For the realistic viscous  flow, the total stress can be split into two parts:  reversible part (denoted by $\bm\sigma_{\textnormal{rev}}$) and irreversible part  (denoted by $\bm\sigma_{\textnormal{irrev}}$):
\begin{eqnarray}\label{eqTotalStress}
   \bm\sigma=\bm\sigma_{\textnormal{rev}}+\bm\sigma_{\textnormal{irrev}}.
\end{eqnarray}
The reversible stress has the form
\begin{eqnarray}\label{eqMulticomponentTotalStressA}
 \bm\sigma_{\textnormal{rev}}=  p \I+ c\(\nabla n\otimes\nabla n\),
\end{eqnarray}
where $p$ is the pressure and $\I$ is the second-order identity tensor. The   pressure  with density gradient contribution  can be expressed as
\begin{eqnarray}\label{eqDefGeneralPres}
    p &=& n \mu- f\nonumber\\
   &=&n\(\mu_b - \div{c\grad{n}}\)-f_b(n)-\frac{1}{2} c\nabla n\cdot\nabla n\nonumber\\
   &=&p_b- n\div{c\grad{n}}-\frac{1}{2} c \nabla n\cdot\nabla n,
\end{eqnarray}
where $p_b$ is the bulk pressure as $$p_b=n\mu_b-f_b.$$

Let $\eta$ and and $\xi$ represent the shear viscosity and volumetric  viscosity respectively.   We assume  $\xi>\frac{2} {3}\eta$ as usual. Newtonian fluid theory suggests
\begin{eqnarray}\label{eqTotalStressB}
   \bm\sigma_{\textnormal{irrev}}=-\eta D(\u)-\(\lambda\div\u\) \I,
\end{eqnarray}
where $D(\u)=\nabla\u+\nabla\u^T$ and $\lambda=\xi-\frac{2}{3}\eta$. 

We denote by $\vartheta$ the internal energy density per unit volume, and  the total energy density includes the internal energy and kinetic energy as $e_T  =\vartheta+\frac{1}{2}\rho|\u|^2$.  The total energy balance equation is stated as
\begin{eqnarray}\label{eqTotalEnergyBalance}
  \frac{\partial e_T}{\partial t}+\div\(e_T\u+\bm\sigma\cdot\u\)=\div\(c(\grad n\otimes\grad n)\cdot\u-\(\div\(\u n\)\)c\grad{n}\)-\div\q,
\end{eqnarray}
where $\q$ is the heat transfer flux as
\begin{eqnarray}\label{eqthermalconductivityGrad02}
  \q= -\Theta \grad T.
 \end{eqnarray} 
 Here, $\Theta$ denotes the heat diffusion coefficient, which depends  generally  on the molar density and temperature. 
   
   %The above system given by \eqref{eqMassConserve1C01}, \eqref{eqMomentumConserve02} and {eqTotalEnergyBalance}    

We now derive the equation of internal energy density from \eqref{eqTotalEnergyBalance}. Using the momentum balance  equation, we obtain the transport of kinetic energy density as 
\begin{eqnarray}\label{eqtransportkineticenergydensity}
 &&\frac{1}{2}\frac{\partial \(\rho|\u|^2\)}{\partial t}+\frac{1}{2}\div\(\u \(\rho|\u|^2\)\)\nonumber\\
 &&~~= \rho \u\cdot\frac{\partial \u}{\partial t}+\frac{1}{2}\u\cdot\u\frac{\partial\rho}{\partial t} +\frac{1}{2} \big(\(\u\cdot\u\)\div\(\rho\u\)+ 2\rho\u\cdot\(\u\cdot\grad\u\)\big)\nonumber\\
   &&~~= \rho\u\cdot\(\frac{\partial \u}{\partial t}+\u\cdot\grad\u\)+\frac{1}{2} \u\cdot\u\(\frac{\partial\rho}{\partial t}+\div\(\rho\u\)\)\nonumber\\
   &&~~= -\u\cdot\div\bm\sigma\nonumber\\
&&~~=-\div\(\bm\sigma\cdot\u\)+\bm\sigma:\grad\u.
\end{eqnarray}
Substituting \eqref{eqtransportkineticenergydensity} into \eqref{eqTotalEnergyBalance} yields  the balance equation of  internal energy density
 \begin{eqnarray}\label{eqInternalEnergy01}
   \frac{\partial \vartheta}{\partial t}+\div(\vartheta\u)=\div\big(c(\grad n\otimes\grad n)\cdot\u-\(\div\(\u n\)\)c\grad{n}\big)-\div\q- \bm\sigma:\grad\u.
 \end{eqnarray}

We denote the bulk internal energy density  by $\vartheta_b$. Then the thermodynamical relation gives 
 \begin{eqnarray}\label{eqBulkInternalenergydensity}
\vartheta_b=f_b+s_bT.
 \end{eqnarray}
 Furthermore, we denote by $\vartheta_\grad$ the gradient contribution of internal energy density,  and  from the thermodynamical relation and formulations of $f_\grad$ and $s_\grad$, we obtain
 \begin{eqnarray}\label{eqGradInternalenergydensity}
\vartheta_\grad=f_\grad+s_\grad T=\frac{1}{2}\(c(T)-Tc'(T)\)\grad n\cdot\grad n.
 \end{eqnarray}

 \section{New formulations and thermodynamical consistency}
 
 As shown in the previous section,  the reversible stress, consisting of the pressure and surface tension terms, has a complicate form. It is  inconvenient for theoretical analysis and construction of numerical methods.   In this section,  we will  prove a relation between the gradients of pressure, temperature and chemical potential, which allows us to simplify  the   momentum balance  equation and energy balance equation.  By the simplified equations, it is convenient to prove that the model obeys the first and second laws of thermodynamics. 

\subsection{New formulations}
The following theorem provides a relation between the gradients of pressure, temperature and chemical potential.
 \begin{thm}\label{lemHelmholtzTransportEqnPureGrad}
The gradients of pressure, temperature and chemical potential have the following relation
\begin{eqnarray}\label{eqPresChptlGradRelation}
 n\grad \mu-  \gamma\grad T
 =\grad p+\div c(\grad n\otimes\grad n).
\end{eqnarray}

\end{thm}
\begin{proof}
First, we have the identity 
\begin{eqnarray}\label{eqPresChptlGradRelationProof01}
 n\grad\mu_b=\grad p_b+\gamma_b\grad T,
\end{eqnarray}
which is deduced from the formulation of $p_b$ as $$\grad p_b=\grad \(n\mu_b-f_b\)=n\grad\mu_b+\mu_b\grad n-\mu_b\grad n-\gamma_b\grad T=n\grad\mu_b-\gamma_b\grad T.$$
From   formulations of  the pressure and chemical potential, taking into account \eqref{eqPresChptlGradRelationProof01}, we  derive  the relation  \eqref{eqPresChptlGradRelation}  as
\begin{eqnarray}\label{eqPresChptlGradRelationProof02}
 n\grad \mu-\grad p 
 &=&n\grad\(\mu_b-\div\(c\grad{n}\)\)-\grad\(p_b- n\div\(c\grad{n}\)-\frac{1}{2} c \nabla n\cdot\nabla n\)\nonumber\\
 &=&n\grad\mu_b-\grad p_b-n\grad\(\div\(c\grad{n}\)\)+\grad\( n\div\(c\grad{n}\)\)+\frac{1}{2}\grad\( c \nabla n\cdot\nabla n\)\nonumber\\
 &=&\gamma_b\grad T+\(\div\(c\grad{n}\)\)\grad n+\frac{1}{2}\grad\( c \nabla n\cdot\nabla n\)\nonumber\\
 &=&\gamma_b\grad T+\(\div\(c\grad{n}\)\)\grad n+\frac{1}{2} |\nabla n|^2\grad{c}+\frac{1}{2}c\grad|\nabla n|^2\nonumber\\
 &=&\gamma_b\grad T+\gamma_\grad\grad{T}+\div c(\grad n\otimes\grad n)\nonumber\\
 &=& \gamma\grad T+\div c(\grad n\otimes\grad n),
\end{eqnarray}
where we have also used the identity   
\begin{eqnarray}\label{eqPresChptlGradRelationProof03}
 \(\div\(c\grad{n}\)\)\grad n+\frac{1}{2}c\grad|\grad n|^2=\div c(\grad n\otimes\grad n).
 \end{eqnarray}
This ends the proof.
\end{proof}

Applying \eqref{eqPresChptlGradRelation} to \eqref{eqMomentumConserve02}, we can obtain a new formulation of the momentum balance equation   
\begin{eqnarray}\label{eqMomentumConserve03}
 \rho\(\frac{\partial\u}{\partial t}+\(\u\cdot\grad\)\u\)= -n\grad  \mu+\gamma\grad T + \div\eta D\(\u\)+\grad\(\lambda\div\u\)  ,
\end{eqnarray}
which demonstrates that the gradients of chemical potential  and temperature are the primal driving force.
 
 We now turn to simplify  the energy balance equation. Applying \eqref{eqPresChptlGradRelation} to \eqref{eqInternalEnergy01}, we  derive
 \begin{eqnarray}\label{eqInternalEnergy02}
   \frac{\partial \vartheta}{\partial t}+\div(\vartheta\u)&=&-\div\big( \q-c(\grad n\otimes\grad n)\cdot\u+\(\div\(\u n\)\)c\grad{n}\big)\nonumber\\
 &&- p \div\u-\(c\grad n\otimes\grad n\):\grad\u - \bm\sigma_{\textnormal{irrev}}:\grad\u\nonumber\\
 &=&-\div\( \q+\u p+\(\div\(\u n\)\)c\grad{n}\)\nonumber\\
 &&+\u\cdot\grad{p}+\u\cdot\div\(c\grad n\otimes\grad n\) - \bm\sigma_{\textnormal{irrev}}:\grad\u\nonumber\\
 &=&-\div\( \q+\u p+\(\div\(\u n\)\)c\grad{n}\)\nonumber\\
 &&+\u\cdot\(n\grad \mu- \gamma\grad T\) - \bm\sigma_{\textnormal{irrev}}:\grad\u\nonumber\\
 &=&-\div\( \q+\(\div\(\u n\)\)c\grad{n}\)+\div\(\u f\)\nonumber\\
 &&- \mu\div(\u n)- \u\cdot\gamma\grad T - \bm\sigma_{\textnormal{irrev}}:\grad\u.
 \end{eqnarray}
 Moving the term $\div\(\u f\)$ into the left-hand side and taking into account $\vartheta=f+Ts$, we obtain  the balance equation of  internal energy density
  \begin{eqnarray}\label{eqInternalEnergy03}
   \frac{\partial \vartheta}{\partial t}+\div(sT\u)&=&-\div\( \q+\(\div\(\u n\)\)c\grad{n}\)\nonumber\\
 &&- \mu\div(\u n) - \u\cdot\gamma\grad T - \bm\sigma_{\textnormal{irrev}}:\grad\u.
 \end{eqnarray}

 We consider the fluids in a closed domain $\Omega$   with a fixed volume.  The natural boundary conditions can be formulated as
\begin{eqnarray}\label{eqEntropySecondb}
 \u =0,~~~~\grad{n}\cdot\bm\nu_{\partial\Omega} =0.
\end{eqnarray}
where  $\bm\nu_{\partial\Omega}$ denotes a normal unit outward vector  to the boundary $\partial \Omega$. For the temperature, we partition the domain boundary $\partial\Omega$   into two non-overlapping subdivisions as
 $\partial\Omega=\Gamma_n\cup\Gamma_d,$
  and impose the boundary conditions
\begin{eqnarray}\label{eqEntropySecondb}
\q\cdot\bm\nu_{\partial\Omega} =\q_{B}~~\textnormal{on}~~\Gamma_n,~~~~\textnormal{and} ~~~~T=T_B~~\textnormal{on}~~\Gamma_d, 
\end{eqnarray}
where $\q_{B}$ is the given heat transfer flux across the boundary and $T_B$ is the given temperature distribution  on the boundary. It is noted that either $\Gamma_n$ or $\Gamma_d$ may vanish or be redivided in a specific problem.  The initial conditions for molar density, temperature and velocity are also provided. 

In summary, the   system of simplified modeling equations is composed of mass balance equation \eqref{eqMassConserve1C01}, the momentum balance equation \eqref{eqMomentumConserve03}, and the balance equation of  internal energy density \eqref{eqInternalEnergy03}, as well as the initial and boundary conditions. As we will see in the next subsection,  it is convenient to verify that this system satisfies the  laws of thermodynamics. 

 \subsection{Thermodynamical consistency}
 We first prove that the simplified model satisfies the first law of thermodynamics. We  define the kinetic energy and internal energy over the domain as
 \begin{eqnarray}\label{eqEnergy01}
 \mathcal{H}&=&\frac{1}{2}\int_\Omega\rho|\u|^2d\x,~~~~~\mathcal{U}=\int_\Omega\vartheta d\x.
\end{eqnarray}
Furthermore, we define the total energy $\mathcal{E}$ over the domain 
$$\mathcal{E}=\int_\Omega e_Td\x=\mathcal{H}+\mathcal{U}.$$ 

\begin{thm}\label{lemHelmholtzTransportEqnPureGrad}
The system of equations \eqref{eqMassConserve1C01}, \eqref{eqMomentumConserve03} and  \eqref{eqInternalEnergy03} satisfies the first law of thermodynamics as
\begin{eqnarray}\label{eqTotalEnergy}
   \frac{\partial \mathcal{E}}{\partial t} &=& -\int_{\partial\Omega}\q_{\partial\Omega}\cdot\bm\nu_{\partial\Omega} d\bm{s},
 \end{eqnarray}
 where $\q_{\partial\Omega}$ denotes the  heat transfer flux between the system and  its environment and $\bm\nu_{\partial\Omega}$ denotes a normal unit outward vector  to the boundary $\partial \Omega$.
\end{thm}
\begin{proof}
From \eqref{eqtransportkineticenergydensity}, we get
 \begin{eqnarray}\label{eqEnergyProof01}
 \frac{\partial\mathcal{H}}{\partial t}&=&\int_\Omega \rho\u\cdot\(\frac{\partial \u}{\partial t}+\u\cdot\grad\u\)d\x.
\end{eqnarray}
Substituting   \eqref{eqMomentumConserve03} into \eqref{eqEnergyProof01} yields 
\begin{eqnarray}\label{eqEnergyProof02}
 \frac{\partial\mathcal{H}}{\partial t}&=&\int_\Omega  \u\cdot\(-n\grad  \mu+\gamma\grad T + \div\eta D\(\u\)+\grad\lambda\div\u\)d\x\nonumber\\
&=&\int_\Omega  \u\cdot\(\gamma\grad T-n\grad  \mu\)d\x+\int_\Omega  \bm\sigma_{\textnormal{irrev}}:\grad\u d\x.
\end{eqnarray}
Integrating \eqref{eqInternalEnergy03} over $\Omega$ yields
\begin{eqnarray}\label{eqEnergyProof03}
   \frac{\partial \mathcal{U}}{\partial t} &=& -\int_{\partial\Omega}\q_{\partial\Omega}\cdot\bm\nu_{\partial\Omega} d\bm{s}- \int_\Omega\(\mu\div(\u n) + \u\cdot\gamma\grad T + \bm\sigma_{\textnormal{irrev}}:\grad\u\)d\x.
 \end{eqnarray}
Taking a summation of  \eqref{eqEnergyProof02} and \eqref{eqEnergyProof03},  and then taking into account  $$\int_\Omega\(\mu\div(\u n)+n\u\cdot\grad  \mu\)d\x=\int_\Omega\div\(\mu\u n\) d\x=0,$$
 we obtain \eqref{eqTotalEnergy}.
  \end{proof}
  
 We next prove that  the simplified model system satisfies the second law of thermodynamics.  For the notations, we use $\(\cdot,\cdot\)$ and $\|\cdot\|$ to represent the $L^2\(\Omega\)$, $\(L^2\(\Omega\)\)^d$ or $\(L^2\(\Omega\)\)^{d\times d}$ inner product and norm respectively.  We  define the entropy over the domain and denote it by  $\mathcal{S} $. Since $s=\frac{1}{T}\(\vartheta-f\)$,  we obtain 
 \begin{eqnarray}\label{eqTotalEntropy}
   \frac{\partial \mathcal{S} }{\partial t}&=&\(\frac{\partial s}{\partial t},1\)=\(\frac{\partial\(\vartheta-f\)}{\partial t},\frac{1}{T}\)-\(\frac{\partial T}{\partial t},\frac{s}{T}\).
 \end{eqnarray}
 
In order to estimate the entropy, we first need to derive a variation equation of Helmholtz free energy density. 
 \begin{lem}\label{lemTransportHelmholtzTransportEqnPureGrad}
The Helmholtz free energy density satisfies the following variation  equation 
\begin{eqnarray}\label{eqHelmholtzTransportEqnPureGrad}
 \frac{\partial  f}{\partial t}&=&\gamma\frac{\partial T}{\partial t}-\mu\div(n\u) -\div\(\(\div\(\u n\)\)c\grad{n}\).
\end{eqnarray}
\end{lem}
\begin{proof}
First, using the mass conservation equation, we obtain the variation  of the bulk Helmholtz free energy $f_b$ as
\begin{eqnarray}\label{eqHelmholtzTransportProof01}
 \frac{\partial f_b}{\partial t}
 &=&\mu_b\frac{\partial n}{\partial t}+\gamma_b\frac{\partial T}{\partial t}\nonumber\\
 &=&-\mu_b\div(n\u)+\gamma_b\frac{\partial T}{\partial t}.
\end{eqnarray}
Furthermore, we derive the variation  of the gradient contribution of Helmholtz free energy $f_\grad$ as
\begin{eqnarray}\label{eqHelmholtzTransportProof02}
 \frac{\partial  f_\grad}{\partial t} &=&\frac{1}{2}\frac{\partial\( c \nabla n\cdot\nabla n\)}{\partial{t}}\nonumber\\
 &=&\frac{1}{2}|\nabla n|^2\frac{\partial c }{\partial{t}} +\frac{1}{2}c\frac{\partial\(  \nabla n\cdot\nabla n\)}{\partial{t}}\nonumber\\
 &=&\gamma_\grad\frac{\partial T}{\partial t}+c\nabla n\cdot\nabla\frac{\partial n}{\partial{t}}\nonumber\\
 &=&\gamma_\grad\frac{\partial T}{\partial t} -c\nabla n\cdot\nabla\(\div\(\u{n}\)\)\nonumber\\
 &=&\gamma_\grad\frac{\partial T}{\partial t}-\div\(\(\div\(\u n\)\)c\grad{n}\)-\mu_\grad\div\(\u n\).
\end{eqnarray}
Thus, \eqref{eqHelmholtzTransportEqnPureGrad} is obtained summing \eqref{eqHelmholtzTransportProof01} and  \eqref{eqHelmholtzTransportProof02}.
\end{proof}

\begin{thm}\label{lemHelmholtzTransportEqnPureGrad}
The system of equations \eqref{eqMassConserve1C01}, \eqref{eqMomentumConserve03} and  \eqref{eqInternalEnergy03} satisfies the second law of thermodynamics as 
\begin{eqnarray}\label{eqTotalEntropyMax}
   \frac{\partial \mathcal{S}}{\partial t}+\int_{\partial\Omega}\frac{\q_{\partial\Omega}\cdot\bm\nu_{\partial\Omega} }{T}d\bm{s}&=& \left\|\frac{\Theta^{1/2}}{T} \grad T\right\|^2 +\frac{1}{2}\left\|\(\frac{\eta}{T}\)^{1/2}D\(\u\)\right\|^2 \nonumber\\
   &&+\left\|\(\frac{\lambda}{T}\)^{1/2}\div\u\right\|^2,
 \end{eqnarray}
\end{thm}
 where $\q_{\partial\Omega}$ denotes the  heat transfer flux between the system and  its environment and $\bm\nu_{\partial\Omega}$ denotes a normal unit outward vector  to the boundary $\partial \Omega$.
\begin{proof}
% \begin{eqnarray}\label{eqInternalEnergy03}
%   \frac{\partial \vartheta}{\partial t}+\div(sT\u)&=&-\div\( \q+\(\div\(\u n\)\)c\grad{n}\)\nonumber\\
% &&- \mu\div(\u n) - \u\cdot\gamma\grad T - \bm\sigma_{\textnormal{irrev}}:\grad\u.
% \end{eqnarray}
%\begin{eqnarray}\label{eqHelmholtzTransportEqnPureGrad}
% \frac{\partial  f}{\partial t}&=&\gamma\frac{\partial T}{\partial t}-\mu\div(n\u) -\div\(\(\div\(\u n\)\)c\grad{n}\).
%\end{eqnarray}
We combine \eqref{eqInternalEnergy03} and \eqref{eqHelmholtzTransportEqnPureGrad} and obtain
\begin{eqnarray}\label{eqEntropyProductionProof01}
  \frac{\partial \(\vartheta-f\)}{\partial t}+\div(sT\u)&=&-\div \q  - \u\cdot\gamma\grad T - \bm\sigma_{\textnormal{irrev}}:\grad\u -\gamma\frac{\partial T}{\partial t}.
   \end{eqnarray}
Since $$\div(sT\u)=T\div(s\u)+s\u\cdot\grad{T}=T\div(s\u)-\gamma\u\cdot\grad T,$$
the equation \eqref{eqEntropyProductionProof01} can be reduced into
\begin{eqnarray}\label{eqEntropyProductionProof02}
  \frac{\partial \(\vartheta-f\)}{\partial t}&=&-T\div(s\u)-\div \q   - \bm\sigma_{\textnormal{irrev}}:\grad\u -\gamma\frac{\partial T}{\partial t}.
   \end{eqnarray}
Substituting \eqref{eqEntropyProductionProof02} into \eqref{eqTotalEntropy}, and taking into account $s=-\gamma$, we derive the entropy variation with time 
\begin{eqnarray}\label{eqEntropyProductionProof03}
   \frac{\partial \mathcal{S}}{\partial t}&=& -\(\div(s\u),1\)+\(\div \Theta \grad T,\frac{1}{T}\)-\(\bm\sigma_{\textnormal{irrev}}:\grad\u,\frac{1}{T}\)\nonumber\\
   &=& -\int_{\partial\Omega}\frac{\q_{\partial\Omega}\cdot\bm\nu_{\partial\Omega} }{T}d\bm{s}+\left\|\frac{\Theta^{1/2}}{T} \grad T\right\|^2-\(\bm\sigma_{\textnormal{irrev}}:\grad\u,\frac{1}{T}\).
 \end{eqnarray}
Applying the formulation of $\bm\sigma_{\textnormal{irrev}}$ to  \eqref{eqEntropyProductionProof03}, we obtain the equation \eqref{eqTotalEntropyMax}.
\end{proof}

\section{Thermodynamically consistent    numerical method}
In this section, we focus on designing  semi-implicit time marching schemes,    which are based on the above simplified formulations and obey the laws of thermodynamics. 
For this purpose,  it is a key ingredient to construct the convex-concave splitting of  Helmholtz free energy density with respect to molar density and temperature.
The other challenge  results from the tight nonlinear coupling relation among  molar density, temperature  and velocity. In order to  alleviate this relation,  we will introduce an  auxiliary velocity, which depends on molar density and temperature.  Very careful  physical observations are also required  to treat this coupling relation by a way of semi-implicit time discretization.

We first consider the convex-concave splitting of bulk Helmholtz free energy density.
  It is noted that $ \psi_p$ in the formulation of $f_b$ is  the molar heat capacity of  the ideal gas at the constant pressure. We  recall  the following thermodynamical relation  for  the ideal gas
\begin{eqnarray}\label{eqRelationHeatCapacity}
 \psi_p=R+ \psi_v,
 \end{eqnarray}
where $R$ is the  universal gas constant and $\psi_v>0$ is the molar heat capacity at the constant volume for the ideal gas.
\begin{lem}\label{lemConvex}
The bulk Helmholtz free energy density $f_b$ can be split into two parts: one is a convex function with respect to $n$, denoted by $f_b^{\textnormal{convex}}(n,T)$, and the other is a concave function with respect to $n$, denoted by $f_b^{\textnormal{concave}}(n,T)$, which are formulated as
\begin{eqnarray}\label{eqConvex}
f_b^{\textnormal{convex}}(n,T)= f_b^{\textnormal{ideal}}(n,T) + f_b^{\textnormal{repulsion}}(n,T),
 \end{eqnarray}
 \begin{eqnarray}\label{eqConcave}
f_b^{\textnormal{concave}}(n,T)=f_b^{\textnormal{attraction}}(n,T).
 \end{eqnarray}
 Moreover, $f_b$ is concave with respect to the temperature. 
\end{lem}
\begin{proof}
The convexity  and concavity of $f_b$ with respect to molar density have been mostly proved in \cite{qiaosun2014} although there exists a bit difference in the formulation of $f_b^{\textnormal{ideal}}(n,T)$.  We primarily prove the concave property of $f_b$ with respect to the temperature. The second derivative of $f_b$ with respect to $T$ can be calculated as
 \begin{eqnarray*}\label{eqSecDerEntropy}
    \frac{\partial^2 f_b(n,T)}{\partial T^2}&=&  \frac{nR}{T}-n\frac{\psi_p(T)}{T} +\frac{na''(T)}{2\sqrt{2}b}\ln\(\frac{1+(1-\sqrt{2})b n}{1+(1+\sqrt{2})b n}\)\nonumber\\
    &=&  -n\frac{\psi_v(T)}{T} +\frac{na''(T)}{2\sqrt{2}b}\ln\(\frac{1+(1-\sqrt{2})b n}{1+(1+\sqrt{2})b n}\),
\end{eqnarray*}
where we have used the relation \eqref{eqRelationHeatCapacity}.  From the definition of $a(T)$, we calculate
%\begin{eqnarray*}
%   a= 0.45724\frac{R^2T_{c}^2}{P_{c}}\[1+m(1-\sqrt{T_{r}})\]^2,~~~~b= 0.07780\frac{RT_{c}}{P_{c}},
%\end{eqnarray*}
\begin{eqnarray*}
   a'(T)= -\frac{a(T)}{1+m(1-\sqrt{T_{r}})}\frac{m}{\sqrt{TT_{c}}},
\end{eqnarray*}
\begin{eqnarray*}
   a''(T)
%   &=& -\frac{a'(T)\(1+m(1-\sqrt{T_{r}})\)+a(T)\frac{m}{2\sqrt{TT_{c}}}}{\(1+m(1-\sqrt{T_{r}})\)^2}\frac{m}{\sqrt{TT_{c}}}\nonumber\\
%   &&+\frac{m}{2\sqrt{T_{c}}}\frac{a(T)}{1+m(1-\sqrt{T_{r}})}\frac{1}{T^{3/2}}\nonumber\\
%   &=& \frac{m^2}{2TT_{c}}\frac{a(T)}{\(1+m(1-\sqrt{T_{r}})\)^2} +\frac{m}{2\sqrt{T_{c}}}\frac{a(T)}{1+m(1-\sqrt{T_{r}})}\frac{1}{T^{3/2}}\nonumber\\
   &=& \frac{ma(T)}{2T\sqrt{TT_{c}}}\frac{1+m}{\(1+m(1-\sqrt{T_{r}})\)^2} .
\end{eqnarray*}
We can see that $a''(T)\geq0$, and then we conclude that  $\frac{\partial^2 f_b(n,T)}{\partial T^2}\leq0$; that is, $f_b$ is concave with respect to the temperature.
\end{proof}

\begin{lem}\label{lemConvex}
The gradient contribution to the Helmholtz free energy density is always convex with respect to  molar density. Moreover,  it is concave with respect to the temperature if we take the temperature 
%$1+\beta_2/\beta_1\leq T_r\leq\(1+1/m\)^2$ 
such that
\begin{eqnarray}\label{eqCondGradientEntropy}
 \(1+m\)\(\beta_1(1-T_{r})+\beta_2\)+4\beta_1T_r\(1+m(1-\sqrt{T_{r}})\)\leq0,
\end{eqnarray}
where $T_r=T/T_c$.
%In particular, if $\beta_2<-\beta_1$ and $T_r\in\left[1+\frac{\beta_2}{\beta_1},\frac{1}{2}+\frac{1}{m}\right]$, then 
\end{lem}
\begin{proof}
It is obvious that the gradient contribution to the Helmholtz free energy density is convex with respect to  molar density.  We now consider its concavity  with respect to the temperature. 
%We first calculate the derivatives of $a(T)$ as
%\begin{eqnarray*}
%   a= 0.45724\frac{R^2T_{c}^2}{P_{c}}\[1+m(1-\sqrt{T_{r}})\]^2,~~~~b= 0.07780\frac{RT_{c}}{P_{c}}.
%\end{eqnarray*}
%\begin{eqnarray*}
% c(T)=a(T)b^{2/3}\[\beta_1(1-T_{r})+\beta_2\],
%\end{eqnarray*}
The derivatives of $c(T)$  are calculated as
\begin{eqnarray*}
 c'(T)=a'(T)b^{2/3}\[\beta_1(1-T_{r})+\beta_2\]-a(T)b^{2/3}\frac{\beta_1}{T_{c}},
\end{eqnarray*}
\begin{eqnarray*}
 c''(T)=a''(T)b^{2/3}\[\beta_1(1-T_{r})+\beta_2\]-2a'(T)b^{2/3}\frac{\beta_1}{T_{c}}.
\end{eqnarray*}
Substituting $a'(T)$ and $a''(T)$ into $c''(T)$, we obtain 
\begin{eqnarray*}
 c''(T)
% &=&\frac{ma(T)b^{2/3}}{2T\sqrt{TT_{c}}}\frac{\(1+m\)\(\beta_1(1-T_{r})+\beta_2\)}{\(1+m(1-\sqrt{T_{r}})\)^2}\nonumber\\
% &&+ \frac{2ma(T)b^{2/3}}{1+m(1-\sqrt{T_{r}})}\frac{\beta_1}{T_{c}\sqrt{TT_{c}}}\nonumber\\
 &=&\frac{ma(T)b^{2/3}}{2T\sqrt{TT_{c}}}\frac{\(1+m\)\(\beta_1(1-T_{r})+\beta_2\)+4\beta_1T_r\(1+m(1-\sqrt{T_{r}})\)}{\(1+m(1-\sqrt{T_{r}})\)^2}.
 \end{eqnarray*}
which  yields the concavity combining the condition \eqref{eqCondGradientEntropy}.
\end{proof}

We make some remarks on the condition \eqref{eqCondGradientEntropy}. In   \eqref{eqCondGradientEntropy}, the parameter $\beta_1$ has a negative value, while the rest parameters are  positive, so the satisfaction of \eqref{eqCondGradientEntropy} is reasonable.    We have checked in numerical tests that the condition \eqref{eqCondGradientEntropy} is  satisfied for butane when the temperature lies in a large range from 0.1$T_c$ to 3$T_c$, where $T_c$ is the critical temperature of butane.  So in what follows, we assume that the condition \eqref{eqCondGradientEntropy} always holds for our considered problems.

We now contruct the semi-implicit time marching scheme.  
A  time interval $\mathcal{I}=(0,T_{f}]$, where $T_{f}>0$, is considered, and we divide $\mathcal{I}$ into   $M$ subintervals $\mathcal{I}_{k}=(t_{k},t_{k+1}]$, where $t_{0}=0$ and $t_{M}=T_{f}$. The time step size is denoted as $\delta t_{k}=t_{k+1}-t_{k}$.  For a scalar function $v(t)$ or a vector function $\mathbf{v}(t)$, we denote by $v^{k}$ or $\mathbf{v}^{k}$ its approximation at the time $t_{k}$.
First,
a semi-implicit time marching scheme accounting for the convex-splitting of Helmholtz free energy density is used to discretize  the chemical potential
\begin{eqnarray}\label{eqDiscreteChPtl}
 && \mu^{k+1}=\mu_b^{k+1}+ \mu_\grad(n^{k+1},T^{k+1}),\nonumber\\
&&  \mu_b^{k+1}=\mu_b^{\textnormal{convex}}(n^{k+1},T^{k+1})+ \mu_b^{\textnormal{concave}}(n^{k},T^{k+1}).
\end{eqnarray}
We define an auxiliary velocity as
\begin{eqnarray}\label{eqDecoupledDiscreteVelocityStar}
\u_\star^{k} = \u^k-\frac{\delta t_{k}}{\rho^k} \(n^{k}\grad \mu^{k+1}+ s^{k}\grad T^{k+1}\),
\end{eqnarray}
where $\rho^k=n^kM_w$.
  We take $\grad \mu^{k+1}\cdot\bm\nu_{\partial\Omega} =\grad T^{k+1}\cdot\bm\nu_{\partial\Omega} =0$ for $\u_\star^{k}$ on the boundary, and as a result, we have still $ \u_\star^{k}=0$ on the boundary. $\u_\star^{k}$ can be viewed as an approximation of $\u^{k+1} $ obtained by neglecting the   convection   and viscosity terms  in the momentum balance equation. Subsequently, a  semi-implicit scheme is designed as below:
\begin{eqnarray}\label{eqDecoupledDiscreteMassConserveEq01}
\frac{ n^{k+1}-n^k}{\delta t_k}+\div\(n^{k}\u_\star^{k}\)=0,
\end{eqnarray}
\begin{eqnarray}\label{eqDecoupledDiscreteMomentumConserve02}
&& \rho^{k}\(\frac{\u^{k+1}-\u^k}{\delta t_k}+\u_\star^{k}\cdot\grad\u^{k+1}\)= -n^{k}\grad \mu^{k+1}+ \gamma^{k}\grad T^{k+1}\nonumber\\
&&~~~~+ \div\eta^k D\(\u^{k+1}\)+\grad\(\lambda^k\div\u^{k+1}\)  ,
\end{eqnarray}
 \begin{eqnarray}\label{eqDecoupledDiscreteInternalEnergyEq01}
 &&\frac{ \vartheta^{k+1}-\vartheta^k}{\delta t_k}+\div\(\u_\star^{k}s^{k}T^{k+1}\) =-\div \q^{k+1}\nonumber\\
 &&-\div \(\(\div\(\u_\star^{k} n^{k}\)\)c^{k+1}\grad{n^{k+1}}\) - \mu^{k+1}\div(\u_\star^{k} n^{k})- \u_\star^{k}\cdot\gamma^{k}\grad T^{k+1}\nonumber\\
 &&+\eta^k D\(\u^{k+1}\):\grad\u^{k+1}+\lambda^k|\div\u^{k+1}|^2\nonumber\\
 &&+\frac{1}{2\delta t_{k}}\rho^{k}\(|\u^{k+1}-\u_\star^{k}|^2+|\u_\star^{k}-\u^{k}|^2\),
 \end{eqnarray} 
 where 
\begin{eqnarray*}\label{eqDiscretethermalconductivityGrad02}
  \q^{k+1}= -\Theta^{k} \grad T^{k+1},~~~~\Theta^{k}=\Theta(n^{k},T^k).
 \end{eqnarray*} 
 %sIn addition, $\q^{k+1}=\q_B(t^{k+1})$ on the domain boundary.
 
 Using the identity
  \begin{eqnarray*}\label{eqDecoupledDiscreteInternalEnergyProof01}
 \div\(\u_\star^{k}s^{k}T^{k+1}\)=T^{k+1}\div\(\u_\star^{k}s^{k}\)- \u_\star^{k}\cdot\gamma^{k}\grad T^{k+1} ,
  \end{eqnarray*} 
 we can also reformulate \eqref{eqDecoupledDiscreteInternalEnergyEq01} as
 \begin{eqnarray}\label{eqDecoupledDiscreteInternalEnergyEq02}
 &&\frac{ \vartheta^{k+1}-\vartheta^k}{\delta t_k}+T^{k+1}\div\(\u_\star^{k}s^{k}\)=-\div \q^{k+1}\nonumber\\
 &&~~-\div \(\(\div\(\u_\star^{k} n^{k}\)\)c^{k+1}\grad{n^{k+1}}\)- \mu^{k+1}\div(\u_\star^{k} n^{k}) \nonumber\\
 &&~~+\eta^k D\(\u^{k+1}\):\grad\u^{k+1}+\lambda^k|\div\u^{k+1}|^2\nonumber\\
 &&~~+\frac{1}{2\delta t_{k}}\rho^{k}\(|\u^{k+1}-\u_\star^{k}|^2+|\u_\star^{k}-\u^{k}|^2\).
 \end{eqnarray}

 We now prove that the above semi-implicit scheme  obeys the   laws of thermodynamics. To do this,  we  define the discrete formulations of  total energy, kinetic energy and  internal energy over the domain at the time $t_k$ as
 \begin{eqnarray*}
\mathcal{E}^k =\mathcal{H}^k+\mathcal{U}^k,~~ \mathcal{H}^k=\frac{1}{2}\int_\Omega\rho^k|\u^k|^2d\x,~~\mathcal{U}^k=\int_\Omega\vartheta^k d\x.
\end{eqnarray*}

\begin{thm}\label{lemDecoupledDiscreteFirstLaw}
The semi-implicit scheme given by  \eqref{eqDiscreteChPtl}-\eqref{eqDecoupledDiscreteInternalEnergyEq01}  satisfies  the first law of thermodynamics as
\begin{eqnarray}\label{eqDecoupledDiscreteFirstLaw}
   \frac{ \mathcal{E}^{k+1}-\mathcal{E}^{k}}{\delta t_k} &=& -\int_{\partial\Omega}\q_{\partial\Omega}^{k+1}\cdot\bm\nu_{\partial\Omega} d\bm{s},
 \end{eqnarray}
 where $\q_{\partial\Omega}^{k+1}$ denotes the  heat transfer flux between the system and  its environment  at the time $t^{k+1}$.
\end{thm}
\begin{proof}
Multiplying both sides of \eqref{eqDecoupledDiscreteMomentumConserve02} by $\u^{k+1}$ and   integrating it over $\Omega$,   we obtain
\begin{eqnarray}\label{eqDiscreteMultiComponentMomentumConserveProof01}
&& \(\rho^{k}\frac{\u^{k+1}-\u_\star^{k}}{\delta t_k},\u^{k+1}\)+\(\rho^{k}\u_\star^{k}\cdot\grad\u^{k+1},\u^{k+1}\)\nonumber\\
&&~~=- \left\|\sqrt{\lambda^k}\div\u^{k+1}\right\|^2 -\frac{1}{2} \left\|\sqrt{\eta^k}D\(\u^{k+1}\)\right\|^2.
\end{eqnarray}
Using \eqref{eqDecoupledDiscreteMassConserveEq01} and taking into account $\rho^k=n^kM_w$, we estimate
 \begin{eqnarray}\label{eqDiscreteMultiComponentMomentumConserveProof02}
&&\(\rho^{k}\(\u^{k+1}-\u_\star^{k}\),\u^{k+1}\)-\frac{1}{2}\(\rho^{k},|\u^{k+1}-\u_\star^{k}|^2\)\nonumber\\
&=&\frac{1}{2}\(\rho^{k},|\u^{k+1}|^2-|\u_\star^{k}|^2\)\nonumber\\
   &=&\mathcal{H}^{k+1}-\mathcal{H}_\star^{k}-\frac{1}{2}\(\rho^{k+1}-\rho^{k},|\u^{k+1}|^2\)\nonumber\\
   &=&\mathcal{H}^{k+1}-\mathcal{H}_\star^{k}+\frac{\delta t_{k}}{2}\(\div(\rho^{k}\u_\star^{k}),|\u^{k+1}|^2\)\nonumber\\
   &=&\mathcal{H}^{k+1}-\mathcal{H}_\star^{k}-\delta t_{k}\(\rho^{k}\u_\star^{k}\cdot\grad\u^{k+1},\u^{k+1}\) ,
\end{eqnarray}
where $$\mathcal{H}_\star^k=\frac{1}{2}\int_\Omega\rho^k|\u_\star^k|^2d\x.$$
Substituting \eqref{eqDiscreteMultiComponentMomentumConserveProof02} into \eqref{eqDiscreteMultiComponentMomentumConserveProof01} yields
\begin{eqnarray}\label{eqDecoupledDiscreteMultiComponentMomentumConserveProof03}
\frac{\mathcal{H}^{k+1}-\mathcal{H}_\star^{k}}{\delta t_{k}}&=&-\frac{1}{2\delta t_{k}}\(\rho^{k},|\u^{k+1}-\u_\star^{k}|^2\)\nonumber\\
   &&-\left\|\sqrt{\lambda^k}\div\u^{k+1}\right\|^2-\frac{1}{2}\left\|\sqrt{\eta^k}D\(\u^{k+1}\)\right\|^2.
\end{eqnarray}
We multiply both sides of \eqref{eqDecoupledDiscreteVelocityStar} by $\u_\star^{k}$ and then integrate it over $\Omega$ 
\begin{eqnarray}\label{eqDecoupledDiscreteVelocityStarProof01}
\mathcal{H}_\star^{k}-\mathcal{H}^{k}&=&\(\rho^k\(\u_\star^{k}-\u^k\), \u_\star^{k}\)-\frac{1}{2}\(\rho^{k},|\u_\star^{k}-\u^{k}|^2\)\nonumber\\
&=&-\delta t_{k}\(n^{k}\grad \mu^{k+1}+ s^{k}\grad T^{k+1}, \u_\star^{k}\)-\frac{1}{2}\(\rho^{k},|\u_\star^{k}-\u^{k}|^2\)\nonumber\\
&=&\delta t_{k}\(\(\div\(n^{k}\u_\star^{k}\), \mu^{k+1}\)- \(s^{k}\grad T^{k+1} ,\u_\star^{k}\)\)-\frac{1}{2}\(\rho^{k},|\u_\star^{k}-\u^{k}|^2\).
\end{eqnarray}
Integrating \eqref{eqDecoupledDiscreteInternalEnergyEq01} over $\Omega$ yields
\begin{eqnarray}\label{eqDecoupledDiscreteInternalEnergy03}
   &&\frac{ \mathcal{U}^{k+1}-\mathcal{U}^k}{\delta t_k}=-\int_{\partial\Omega}\q_{\partial\Omega}^{k+1}\cdot\bm\nu_{\partial\Omega} d\bm{s} - \(\mu^{k+1},\div(\u_\star^{k} n^{k})\) \nonumber\\
 &&- \(\u_\star^{k},\gamma^{k}\grad T^{k+1}\)+\left\|\sqrt{\lambda^k}\div\u^{k+1}\right\|^2+\frac{1}{2}\left\|\sqrt{\eta^k}D\(\u^{k+1}\)\right\|^2\nonumber\\
 &&+\frac{1}{2\delta t_{k}}\(\rho^{k},|\u^{k+1}-\u_\star^{k}|^2+|\u_\star^{k}-\u^{k}|^2\).
 \end{eqnarray}
 Summing  \eqref{eqDecoupledDiscreteMultiComponentMomentumConserveProof03}, \eqref{eqDecoupledDiscreteVelocityStarProof01} and \eqref{eqDecoupledDiscreteInternalEnergy03} yields \eqref{eqDecoupledDiscreteFirstLaw}.
  \end{proof}
  
%The equation \eqref{eqDecoupledDiscreteFirstLaw} indicates that for an isolated fluid system (i.e. $\q_{\partial\Omega}=0$), the total energy of this system is always conserved. 
We turn to prove that the proposed semi-implicit scheme obeys the second law of thermodynamics.  We first need to prove a discrete analog of Lemma \ref{lemTransportHelmholtzTransportEqnPureGrad}. 
 \begin{lem}\label{lemDecoupledDiscreteHelmholtzTransportEqnPureGrad}
Assume that the condition \eqref{eqCondGradientEntropy} holds. The discrete Helmholtz free energy densities satisfy 
\begin{eqnarray}\label{eqDecoupledDiscreteHelmholtzTransportEqnPureGrad}
 \frac{ f^{k+1}-f^k}{\delta t_k}
 &\leq&\gamma^{k}\frac{ T^{k+1}-T^{k}}{\delta t_k}-\mu^{k+1}\div\(n^{k}\u_\star^{k}\)\nonumber\\
 &&-\div \(\(\div\(\u_\star^{k} n^{k}\)\)c^{k+1}\grad{n^{k+1}}\),
\end{eqnarray}
where $f^k=f(n^{k},T^k)$.
\end{lem}
\begin{proof}
We utilize  the properties of convex and concave functions to estimate the bulk Helmholtz free energy $f_b$ as
\begin{eqnarray}\label{eqDecoupledDiscreteHelmholtzTransport01}
 \frac{ f_b^{k+1}-f_b^k}{\delta t_k}
 &=& \frac{ f_b(n^{k+1},T^{k+1})-f_b(n^{k},T^{k+1})}{\delta t_k}+\frac{ f_b(n^{k},T^{k+1})-f_b(n^{k},T^k)}{\delta t_k}\nonumber\\
 &\leq&\mu^{k+1}_b\frac{ n^{k+1}-n^{k}}{\delta t_k}+\gamma_b^{k}\frac{ T^{k+1}-T^{k}}{\delta t_k}\nonumber\\
 &=&-\mu^{k+1}_b\div\(n^{k}\u_\star^{k}\)+\gamma_b^{k}\frac{ T^{k+1}-T^{k}}{\delta t_k},
\end{eqnarray}
where we have also used the discrete equation of mass balance. Next, we consider the difference  of the gradient contributions to Helmholtz free energy between two time steps as
\begin{eqnarray}\label{eqDecoupledDiscreteHelmholtzTransport02}
 f_\grad^{k+1}-f_\grad^k
 &=& \frac{1}{2} c^{k+1} \nabla n^{k+1}\cdot\nabla n^{k+1}-\frac{1}{2} c^k \nabla n^k\cdot\nabla n^k\nonumber\\
 &=& \frac{1}{2} c^{k+1} \(\nabla n^{k+1}\cdot\nabla n^{k+1}- \nabla n^k\cdot\nabla n^k\)+\frac{1}{2} \(c^{k+1}-c^k\) \nabla n^k\cdot\nabla n^k\nonumber\\
 &\leq& \frac{1}{2} c^{k+1} \nabla \(n^{k+1}+n^k\)\cdot\nabla \(n^{k+1}-  n^k\)+\gamma_\grad^k \(T^{k+1}-T^k\),
\end{eqnarray}
where the concavity of the function $c$ with respect to the temperature is used to get the last inequality.  Using the discrete equation of mass balance and the definition of $\mu^{k+1}_\grad$, we derive
\begin{eqnarray}\label{eqDecoupledDiscreteHelmholtzTransport03}
&&\frac{1}{2} c^{k+1} \nabla \(n^{k+1}+n^k\)\cdot\nabla \(n^{k+1}-  n^k\) \nonumber\\
 &&~~=c^{k+1} \nabla n^{k+1}\cdot\nabla \(n^{k+1}-  n^k\)-\frac{1}{2} c^{k+1} \nabla \(n^{k+1}-n^k\)\cdot\nabla \(n^{k+1}-  n^k\) \nonumber\\
 && ~~\leq- \delta t_k c^{k+1} \nabla n^{k+1}\cdot\nabla \(\div\(\u_\star^{k}n^{k}\)\)  \nonumber\\
 &&~~\leq- \delta t_k\div \(\(\div\(\u_\star^{k} n^{k}\)\)c^{k+1}\grad{n^{k+1}}\)- \delta t_k\mu^{k+1}_\grad\div\(\u_\star^{k}n^{k}\).
\end{eqnarray}
Substituting \eqref{eqDecoupledDiscreteHelmholtzTransport03} into \eqref{eqDecoupledDiscreteHelmholtzTransport02} yields 
 \begin{eqnarray}\label{eqDecoupledDiscreteHelmholtzTransport04}
 \frac{f_\grad^{k+1}-f_\grad^k}{\delta t_k}
 &\leq& \gamma_\grad^k \frac{T^{k+1}-T^k}{\delta t_k}- \mu^{k+1}_\grad\div\(\u_\star^{k}n^{k}\)\nonumber\\
 &&- \div \(\(\div\(\u_\star^{k} n^{k}\)\)c^{k+1}\grad{n^{k+1}}\).
\end{eqnarray}
Finally, \eqref{eqDecoupledDiscreteHelmholtzTransportEqnPureGrad} is a result of summing \eqref{eqDecoupledDiscreteHelmholtzTransport01} and  \eqref{eqDecoupledDiscreteHelmholtzTransport04}.
\end{proof}

\begin{thm}\label{lemDecoupledDiscreteSecondLaw}
Assume that the condition \eqref{eqCondGradientEntropy} holds. The semi-implicit scheme given by  \eqref{eqDiscreteChPtl}-\eqref{eqDecoupledDiscreteInternalEnergyEq01}   satisfies  the second law of thermodynamics as
\begin{eqnarray}\label{eqDecoupledDiscreteSecondLaw}
  && \frac{ \mathcal{S}^{k+1}-\mathcal{S}^{k}}{\delta t_{k}}+\int_{\partial\Omega}\frac{\q_{\partial\Omega}^{k+1}\cdot\bm\nu_{\partial\Omega} }{T^{k+1}}d\bm{s}\geq\left\|\frac{\sqrt{\Theta^{k+1}}}{T^{k+1}} \grad{T^{k+1}}\right\|^2+\(\eta^k D\(\u^{k+1}\):\grad\u^{k+1},\frac{1}{T^{k+1}}\)\nonumber\\
  &&~~~~~~+\(\lambda^k|\div\u^{k+1}|^2+\frac{1}{2\delta t_{k}}\rho^{k}\(|\u^{k+1}-\u_\star^{k}|^2+|\u_\star^{k}-\u^{k}|^2\),\frac{1}{T^{k+1}}\)\geq0,
 \end{eqnarray}
  where $\q_{\partial\Omega}^{k+1}$ denotes the  heat transfer flux between the system and its environment at the time $t^{k+1}$.
\end{thm}
\begin{proof}
Since $s=\frac{1}{T}\(\vartheta-f\)$,  we obtain 
 \begin{eqnarray}\label{eqDecoupledDiscreteEntropyProduction01}
   \frac{ \mathcal{S}^{k+1}-\mathcal{S}^k}{\delta t_k}&=&\(\frac{ s^{k+1}-s^k}{\delta t_k},1\)\nonumber\\
   &=&\(\frac{\vartheta^{k+1}-\vartheta^k-f^{k+1}+f^k}{\delta t_k},\frac{1}{T^{k+1}}\)-\(\frac{ T^{k+1}-T^k}{\delta t_k},\frac{s^k}{T^{k+1}}\).
%   \nonumber\\
%   &=&\(\frac{\vartheta^{k+1}-\vartheta^k-f^{k+1}+f^k}{\delta t_k},\frac{1}{T^{k}}\)-\(\frac{ T^{k+1}-T^k}{\delta t_k},\frac{s^{k+1}}{T^{k}}\).
 \end{eqnarray}
 Taking into account the relation $s^k=-\gamma^{k}$, we substitute \eqref{eqDecoupledDiscreteHelmholtzTransportEqnPureGrad} into \eqref{eqDecoupledDiscreteEntropyProduction01} and obtain 
 \begin{eqnarray}\label{eqDecoupledDiscreteEntropyProduction02}
 \frac{ \mathcal{S}^{k+1}-\mathcal{S}^k}{\delta t_k}
 &\geq&\(\frac{\vartheta^{k+1}-\vartheta^k}{\delta t_k},\frac{1}{T^{k+1}}\) +\(\mu^{k+1}\div\(n^{k}\u_\star^{k}\),\frac{1}{T^{k+1}}\)\nonumber\\
 &&+\(\div \(\(\div\(\u_\star^{k} n^{k}\)\)c^{k+1}\grad{n^{k+1}}\),\frac{1}{T^{k+1}}\).
\end{eqnarray}
It can be obtained from \eqref{eqDecoupledDiscreteInternalEnergyEq02} that
 \begin{eqnarray}\label{eqDecoupledDiscreteInternalEnergyProof02}
  \(\frac{\vartheta^{k+1}-\vartheta^k}{\delta t_k},\frac{1}{T^{k+1}}\)&=&-\(\div \q^{k+1},\frac{1}{T^{k+1}}\)\nonumber\\
  &&-\(\div \(\(\div\(\u_\star^{k} n^{k}\)\)c^{k+1}\grad{n^{k+1}}\),\frac{1}{T^{k+1}}\)\nonumber\\
  &&-\( \mu^{k+1}\div(\u_\star^{k} n^{k}),\frac{1}{T^{k+1}}\)\nonumber\\
  &&+\(\lambda^k|\div\u^{k+1}|^2+ \eta^k D\(\u^{k+1}\):\grad\u^{k+1},\frac{1}{T^{k+1}}\)\nonumber\\
  &&+\(\frac{1}{2\delta t_{k}}\rho^{k}\(|\u^{k+1}-\u_\star^{k}|^2+|\u_\star^{k}-\u^{k}|^2\),\frac{1}{T^{k+1}}\).
 \end{eqnarray}
Substituting \eqref{eqDecoupledDiscreteInternalEnergyProof02} into \eqref{eqDecoupledDiscreteEntropyProduction02}, and taking into account 
\begin{eqnarray}\label{eqDecoupledDiscreteInternalEnergyProof03}
   -\(\div \q^{k+1},\frac{1}{T^{k+1}}\)&=& -\int_{\partial\Omega}\frac{\q_{\partial\Omega}^{k+1}\cdot\bm\nu_{\partial\Omega} }{T^{k+1}}d\bm{s}+\left\|\frac{\sqrt{\Theta^{k+1}}}{T^{k+1}} \grad{T^{k+1}}\right\|^2,
 \end{eqnarray}
we obtain   \eqref{eqDecoupledDiscreteSecondLaw}.
\end{proof}

Although the nonlinear coupling relationship between molar density, temperature and velocity has been alleviated to a great extent by a series of semi-implicit treatments, the above time-discrete system  still suffers from weakly nonlinear  coupling. To solve the discrete systems efficiently,  with the help of the auxiliary velocity, we propose the following
fully decoupled, linearized iterative method:
 \begin{eqnarray}\label{eqIterativeDecoupledDiscreteVelocityStar}
\u_\star^{k,l} = \u^k-\frac{\delta t_{k}}{\rho^k} \(n^{k}\grad \mu^{k+1,l+1}+ s^{k}\grad T^{k+1,l}\),
\end{eqnarray}
\begin{eqnarray}\label{eqIterativeDecoupledDiscreteMassConserve1C03}
\frac{ n^{k+1,l+1}-n^k}{\delta t_k}+\div\(n^{k}\u_\star^{k,l}\)=0,
\end{eqnarray}
\begin{eqnarray}\label{eqIterativeDecoupledDiscreteMomentumConserve02}
&& \rho^{k}\(\frac{\u^{k+1,l+1}-\u^k}{\delta t_k}+\u_\star^{k,l}\cdot\grad\u^{k+1,l+1}\)= -n^{k}\grad \mu^{k+1,l+1}+ \gamma^{k}\grad T^{k,l}\nonumber\\
&&+ \div\eta^k D\(\u^{k+1,l+1}\)+\grad\(\lambda^k\div\u^{k+1,l+1}\)  ,
\end{eqnarray}
 \begin{eqnarray}\label{eqIterativeDecoupledDiscreteInternalEnergy02}
 &&\frac{ \vartheta^{k+1,l+1}-\vartheta^k}{\delta t_k}+\div\(\u_\star^{k,l}s^{k}T^{k+1,l}\)\nonumber\\
 &&=\div \Theta^{k} \grad T^{k+1,l+1}-\div \(\(\div\(\u_\star^{k,l} n^{k}\)\)c^{k+1,l}\grad{n^{k+1,l+1}}\) \nonumber\\
 &&- \mu^{k+1,l+1}\div(\u_\star^{k,l} n^{k})- \u_\star^{k,l}\cdot\gamma^{k}\grad T^{k,l}+\eta^k D\(\u^{k+1,l+1}\):\grad\u^{k+1,l+1}\nonumber\\
 &&+\lambda^k|\div\u^{k+1,l+1}|^2+\frac{1}{2\delta t_{k}}\rho^{k}\(|\u^{k+1,l+1}-\u_\star^{k,l}|^2+|\u_\star^{k,l}-\u^{k}|^2\),
 \end{eqnarray} 
 where the superscripts $l$ and $l+1$ denote the $l$th and $(l+1)$th iterations respectively and $ \mu^{k+1,l+1}, \vartheta^{k+1,l+1}$ are defined as
 \begin{eqnarray}\label{eqIterativeDiscreteChPtl}
  \mu^{k+1,l+1}&=&\mu_b^{\textnormal{convex}}(n^{k+1,l},T^{k+1,l})+\frac{\partial\mu_b^{\textnormal{convex}}}{\partial n}(n^{k+1,l},T^{k+1,l})\(n^{k+1,l+1}-n^{k+1,l}\)\nonumber\\
  &&+ \mu_b^{\textnormal{concave}}(n^{k},T^{k+1,l})+ \mu_\grad(n^{k+1,l+1},T^{k+1,l}),
\end{eqnarray}
\begin{eqnarray}\label{eqIterativeDiscreteInterEnergy}
  \vartheta^{k+1,l+1}&=&\vartheta(n^{k+1,l+1},T^{k+1,l})+\frac{\partial\vartheta}{\partial T}(n^{k+1,l+1},T^{k+1,l})\(T^{k+1,l+1}-T^{k+1,l}\).
\end{eqnarray}

For the above iterative method, using the similar techniques in the proof of Theorem \ref{lemDecoupledDiscreteFirstLaw}, we can prove the following theorem.
\begin{thm}\label{lemIterDecoupledIterDiscreteFirstLaw}
The iterative method given by  \eqref{eqIterativeDecoupledDiscreteVelocityStar}-\eqref{eqIterativeDiscreteInterEnergy}  satisfies  the first law of thermodynamics as 
\begin{eqnarray}\label{eqDecoupledIterDiscreteFirstLaw}
   \frac{ \mathcal{E}^{k+1,l+1}-\mathcal{E}^{k}}{\delta t_k} &=& -\int_{\partial\Omega}\q_{\partial\Omega}^{k+1}\cdot\bm\nu_{\partial\Omega} d\bm{s},
 \end{eqnarray}
 where 
 \begin{eqnarray*}\label{eqEnergy02}
\mathcal{E}^{k+1,l+1} =\mathcal{H}^{k+1,l+1}+\mathcal{U}^{k+1,l+1},
\end{eqnarray*}
\begin{eqnarray*}\label{eqEnergy03}
\mathcal{H}^{k+1,l+1}=\frac{1}{2}\int_\Omega\rho^{k+1,l+1}|\u^{k+1,l+1}|^2d\x,~~~\mathcal{U}^{k+1,l+1}=\int_\Omega\vartheta^{k+1,l+1} d\x.
\end{eqnarray*}
\end{thm}
Thanks to the  feature that this iterative method satisfies  the first law of thermodynamics, it  converges rapidly in practical applications. 

%%%%%%%%%%%%%% Conclusions %%%%%%%%%
\section{Numerical results}

%%%%%%%%%%%%%%%%%%%%%%%%%%%%%%%%%%%%%%%%
In this section, we employ the proposed method to carry out a series of numerical tests. The simulated substance  is n-butane (nC$_4$), and its physical data is listed in Table \ref{tabParametersPREOS}, in which $\vartheta_0$,  $T_0$ and $P_0$ take the values suggested in \cite{smejkal2017phase}.  The correlation coefficients for  the molar heat capacity in \eqref{eqHeatCapacity01} are taken as \cite{smejkal2017phase,Reid1987book} $$\alpha_0=9.487, \alpha_1=3.313\times10^{-1}, \alpha_2=-1.108\times10^{-4},\alpha_3=-2.822\times10^{-9}.$$ 
 The heat conduction coefficient is set to be a constant as $\Theta=0.1$W/m/K. The volumetric  viscosity and    the shear viscosity are taken as $\xi=\eta=10^{-4}$Pa$\cdot$s.  In all  numerical tests,  to initialize the molar density distributions, we use the  following gas and liquid molar densities, denoted by  $$n_G = 358.2996\textnormal{mol/m}^3,~~~~~~n_L = 9058.3724\textnormal{mol/m}^3.$$

We use the rectangular domains,  and  denote  the spatial coordinate $\x=(x,y)\in \mathbb{R}^2$.   The cell-centered finite difference method and the upwind scheme are employed  to discretize  the mass balance equation and energy balance equation, while   the finite volume method  on the staggered mesh \cite{Tryggvason2011book} is used for the momentum balance equation. These spatial discretization schemes  can be equivalent to special mixed finite element methods with quadrature rules \cite{arbogast1997mixed,Girault1996Mac}. The stop criterion of the iterative method for solving the discrete equations  is that the 2-norm of  the relative variation of molar density, velocity and temperature  between the current and previous  iterations  is less than $10^{-3}$, and the maximum nonlinear iterations are also set to be not larger than 10 for preventing too many  loops.  These settings are enough to ensure the convergence of nonlinear iterations in the most cases.

 \begin{table}[htp]
\caption{Physical parameters of nC$_4$}
\begin{center}
\begin{tabular}{ccccccccc}
\hline
$M_w$(g/mol) &  $P_c$(bar)   &  $T_c$(K)   & $\omega$  & $\vartheta_0$(J/mol)  & $T_0$(K)   & $P_0$(bar)  \\
\hline
58.12               &  38.0               & 425.2          &0.199        & -2478.95687512     & 298.15      &   1   \\
\hline
\end{tabular}
\end{center}
\label{tabParametersPREOS}
\end{table}

\subsection{Isolated system}
 In this example,  we consider an ideal isolated system, which exchanges no mass or heat energy with its environment.  The computational   domain is a square as $\Omega=(-L,L)^2$, where $L=10$nm, and   a  uniform  rectangular mesh with $40\times40$ elements is applied. We take a  fixed time step size  $\delta t=3\times10^{-13}$s, and  simulate the dynamics of this system for 500 time steps. The initial temperature of this system is homogeneous and equal to 345K.
The initial molar density is defined by the following function
 $$n = \left\{\begin{array}{cc}n_L, & |x|\leq r~~ \textnormal{and}~~|y|\leq r, \\n_G, & \textnormal{elsewhere},\end{array}\right.$$
 where  $r=0.35L$. Namely,   a square droplet is initially located at the center of the domain.
The discrete  initial molar density is also illustrated in Figure \ref{IsolatedSysnC4MolarDensityOfnC4iT0}.  
The initial  velocity is zero.  In the simulation time, we always set the boundary conditions   $\q\cdot\bm\nu_\Omega=0$ on the  boundary $\partial\Omega$.

In Figures \ref{IsolatedSysnC4MolarDensityOfnC4}, we illustrate the  molar density profiles at different time steps, while the  temperature profiles and velocity fields at different time steps  are depicted  in  Figures \ref{IsolatedSysnC4TemperatureOfnC4} and  Figures \ref{IsolatedSysnC4Velocity} respectively.

It is obviously observed from Figures \ref{IsolatedSysnC4MolarDensityOfnC4} that the droplet changes from a square to a circle due to  the effect of the interfacial tension.  Figures \ref{IsolatedSysnC4TemperatureOfnC4} show that  the region around the droplet  has higher temperatures than the gas region, but  the temperature fields tend towards a homogeneous distribution during the evolution of this system.   The velocity fields in Figures \ref{IsolatedSysnC4Velocity} depict the flow evolutions  with the mass and temperature variations; in particular, we can see that the magnitudes of both velocity components decrease with time steps especially after the 50th time step. From these results, it can be predicted that the system will tend towards a equilibrium state infinitely. 
 
\begin{figure}
            \centering \subfigure[]{
            \begin{minipage}[b]{0.31\textwidth}
               \centering
             \includegraphics[width=0.95\textwidth,height=1.33in]{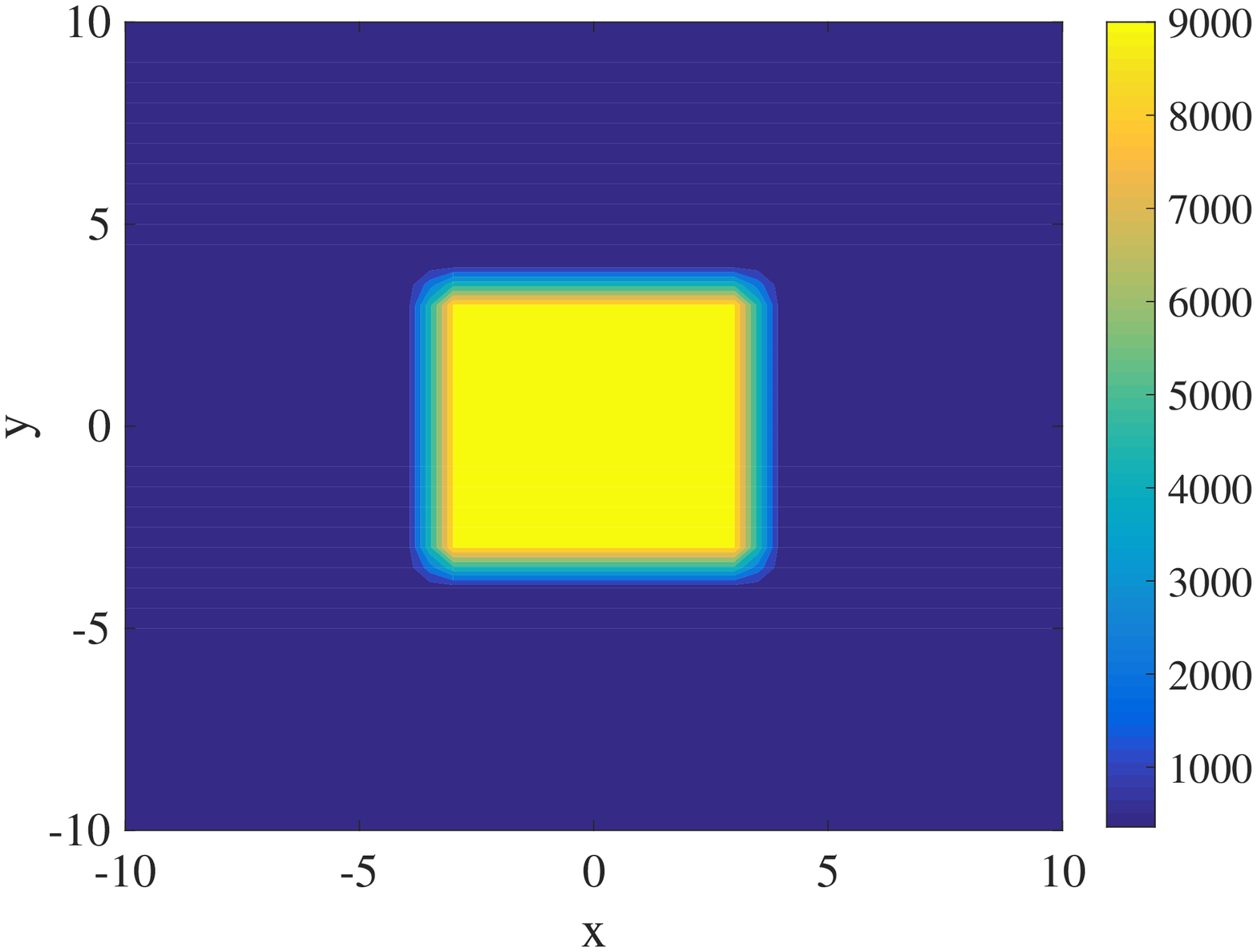}
              \label{IsolatedSysnC4MolarDensityOfnC4iT0}
            \end{minipage}
            }
            \centering \subfigure[]{
            \begin{minipage}[b]{0.31\textwidth}
            \centering
             \includegraphics[width=0.95\textwidth,height=1.33in]{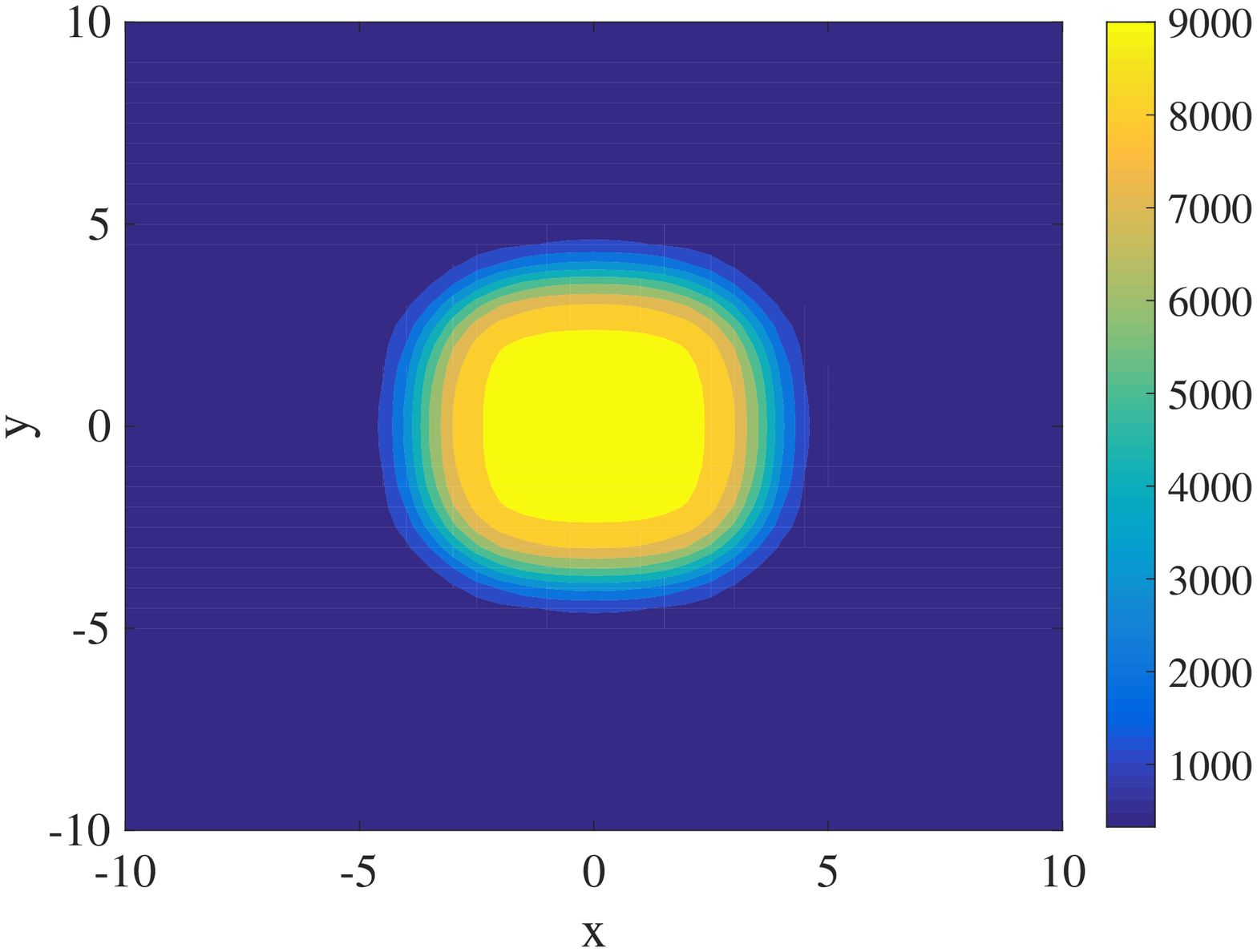}
            \end{minipage}
            }
           \centering \subfigure[]{
            \begin{minipage}[b]{0.3\textwidth}
            \centering
             \includegraphics[width=0.95\textwidth,height=1.33in]{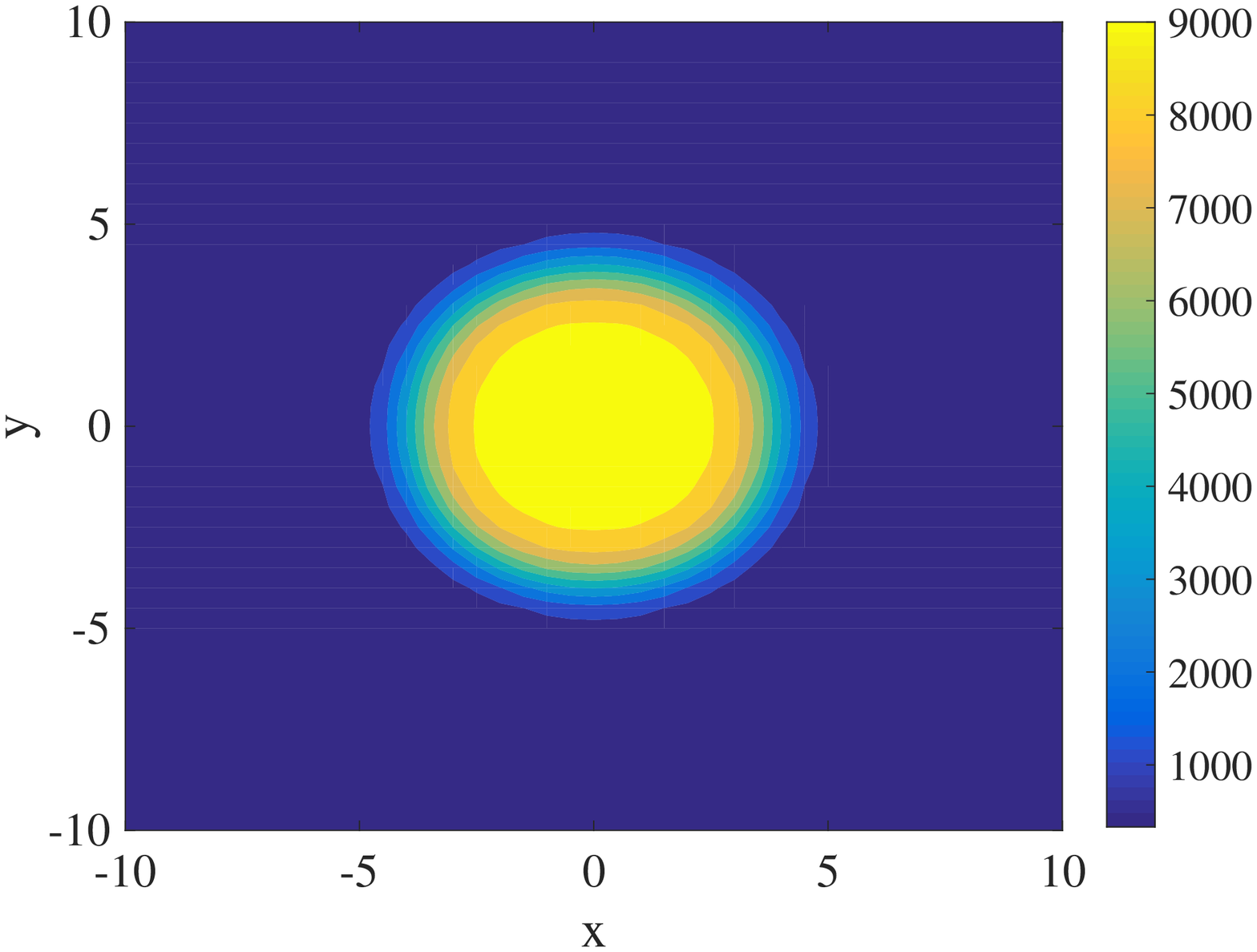}
            \end{minipage}
            }
           \caption{Dynamics of an isolated system: the initial molar density  distribution(a) and the molar density distributions at the 200th(b) and  500th(c) time step respectively.}
            \label{IsolatedSysnC4MolarDensityOfnC4}
 \end{figure}

\begin{figure}
            \centering \subfigure[]{
            \begin{minipage}[b]{0.31\textwidth}
               \centering
             \includegraphics[width=0.95\textwidth,height=1.33in]{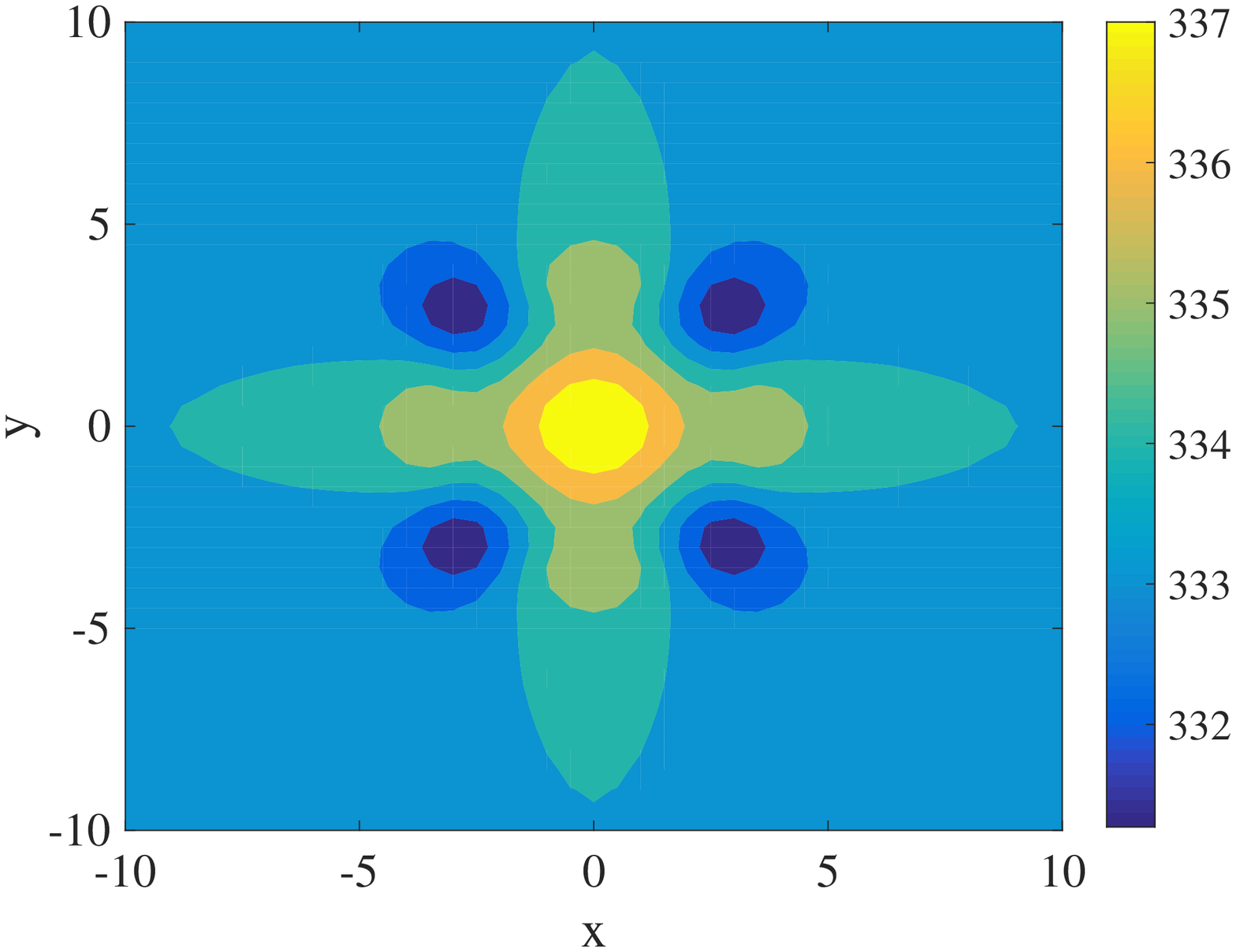}
            \end{minipage}
            }
            \centering \subfigure[]{
            \begin{minipage}[b]{0.31\textwidth}
            \centering
             \includegraphics[width=0.95\textwidth,height=1.33in]{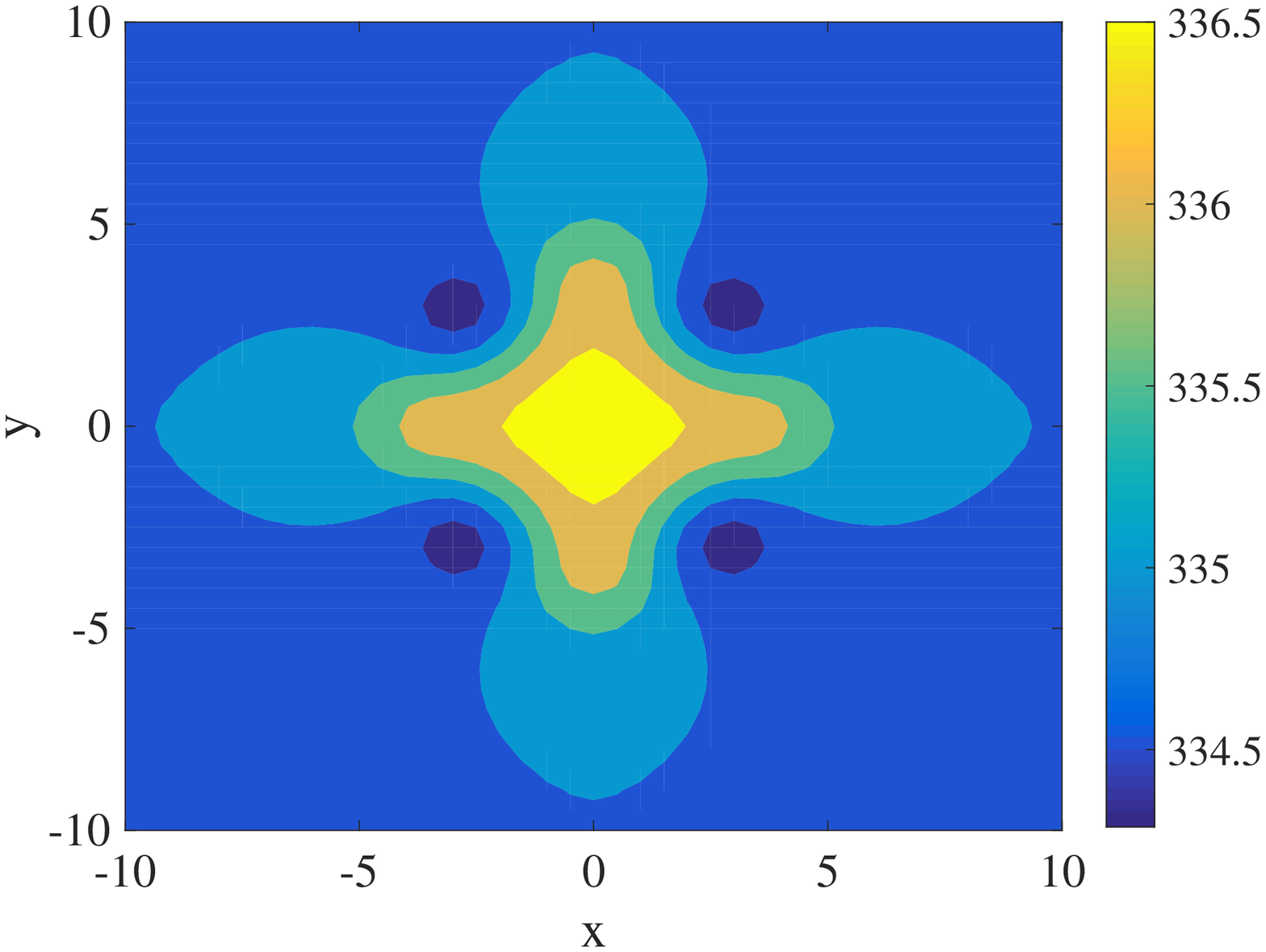}
            \end{minipage}
            }
            \centering \subfigure[]{
            \begin{minipage}[b]{0.3\textwidth}
            \centering
             \includegraphics[width=0.95\textwidth,height=1.3in]{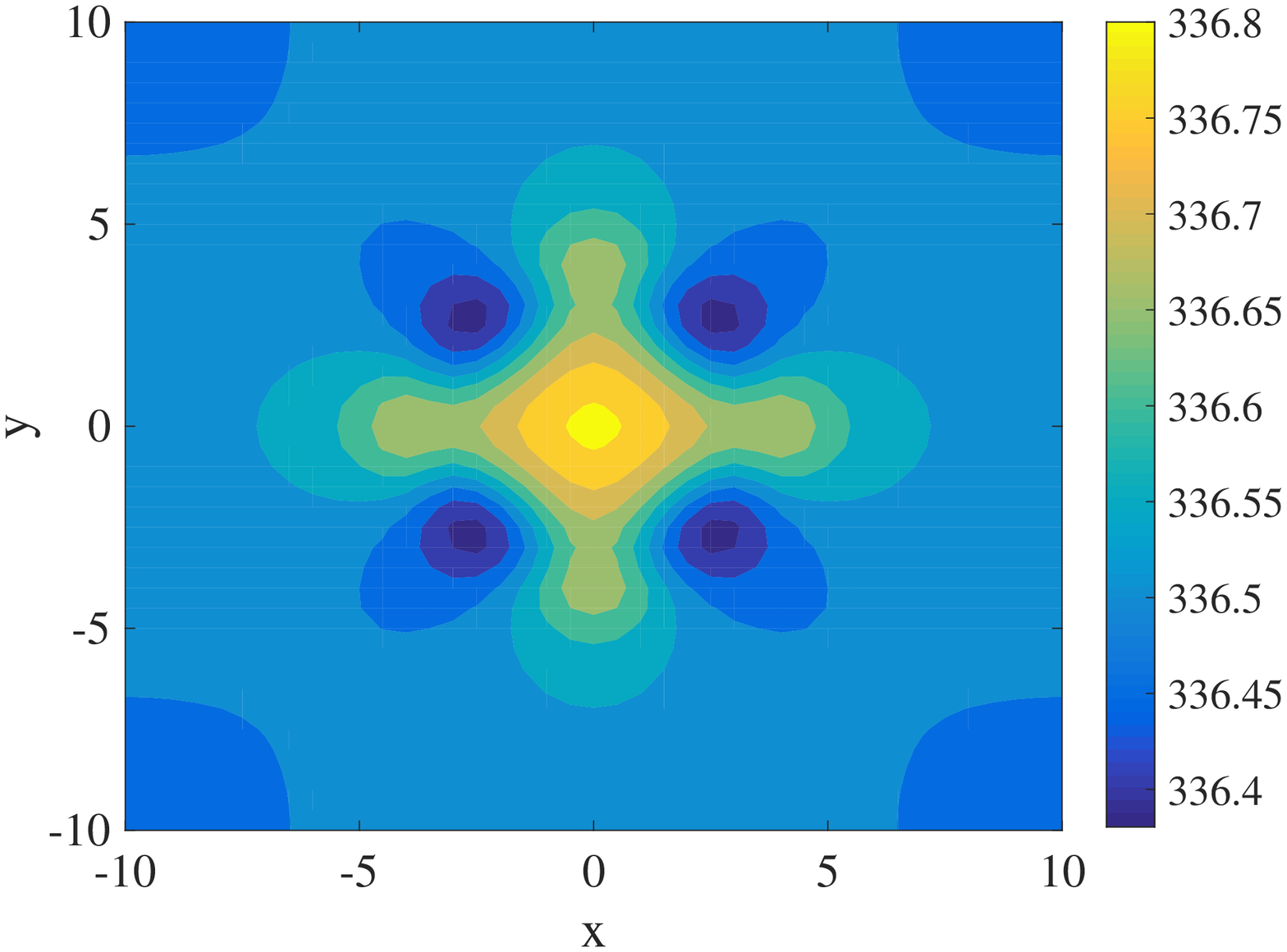}
            \end{minipage}
            }
           \caption{Dynamics of an isolated system:   the temperature profiles at the 50th(a), 200th(b) and  500th(c) time step respectively.}
            \label{IsolatedSysnC4TemperatureOfnC4}
 \end{figure}

\begin{figure}
            \centering \subfigure[]{
            \begin{minipage}[b]{0.31\textwidth}
               \centering
             \includegraphics[width=0.95\textwidth,height=1.33in]{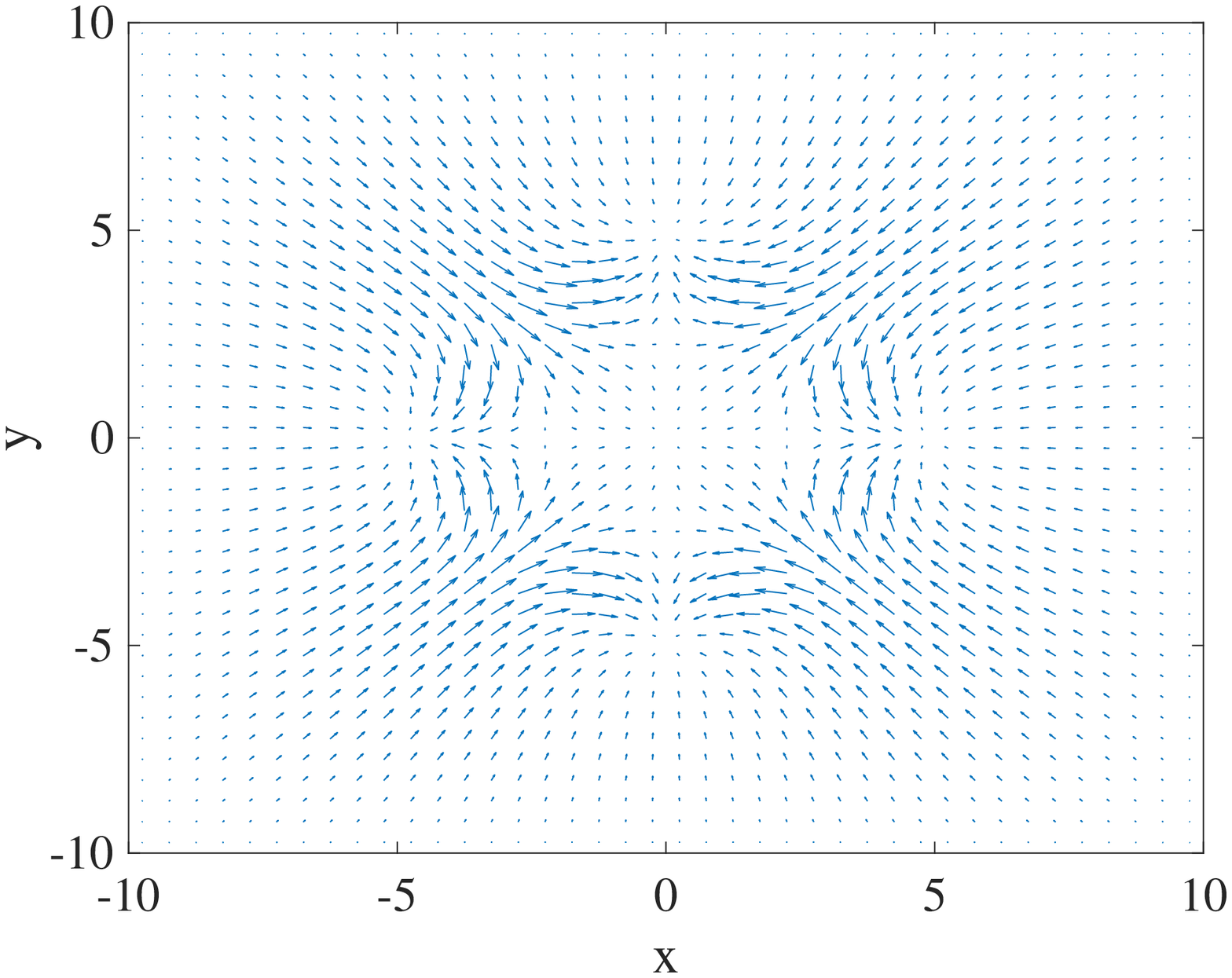}
            \end{minipage}
            }
            \centering \subfigure[]{
            \begin{minipage}[b]{0.31\textwidth}
            \centering
             \includegraphics[width=0.95\textwidth,height=1.33in]{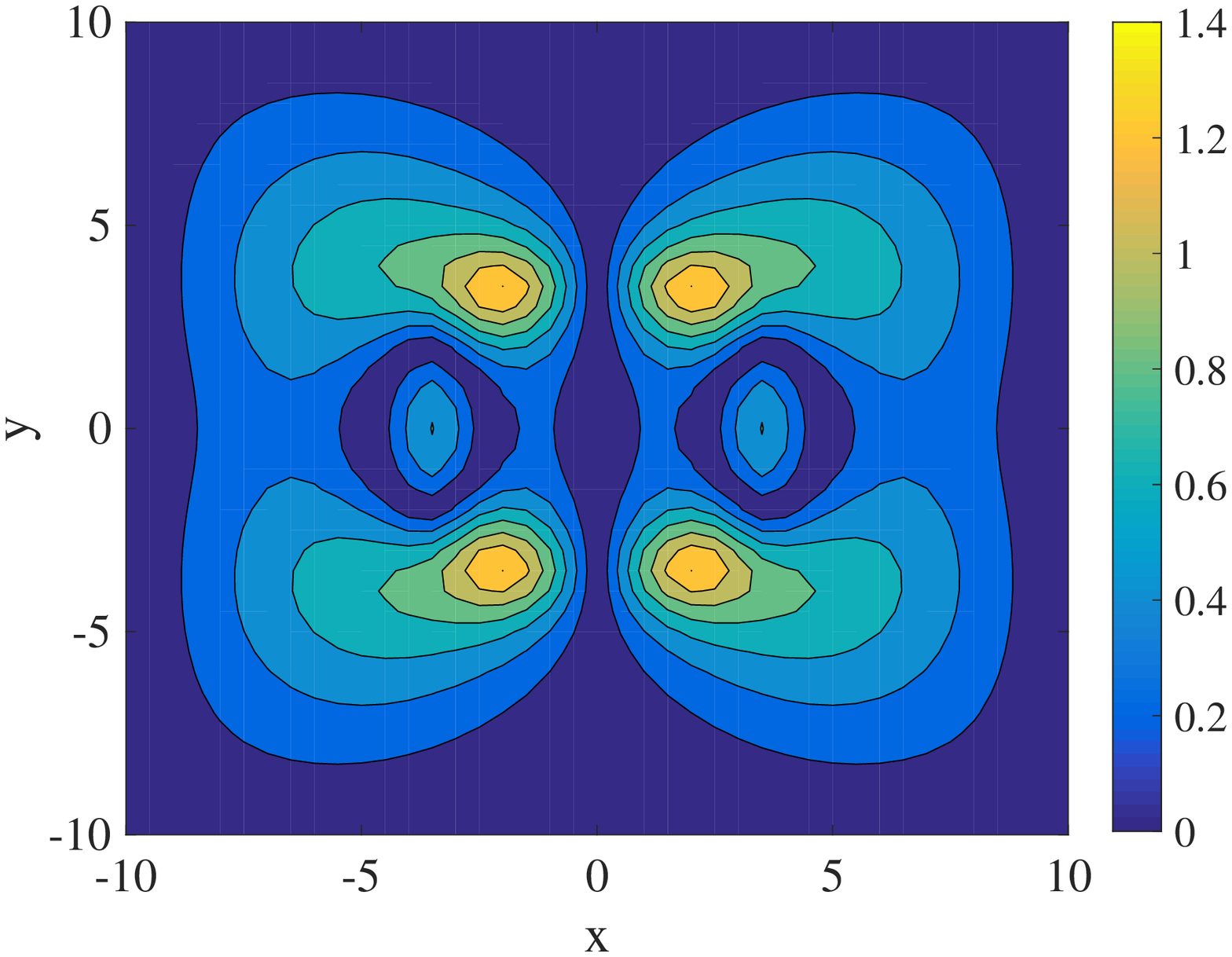}
            \end{minipage}
            }
            \centering \subfigure[]{
            \begin{minipage}[b]{0.3\textwidth}
            \centering
             \includegraphics[width=0.95\textwidth,height=1.3in]{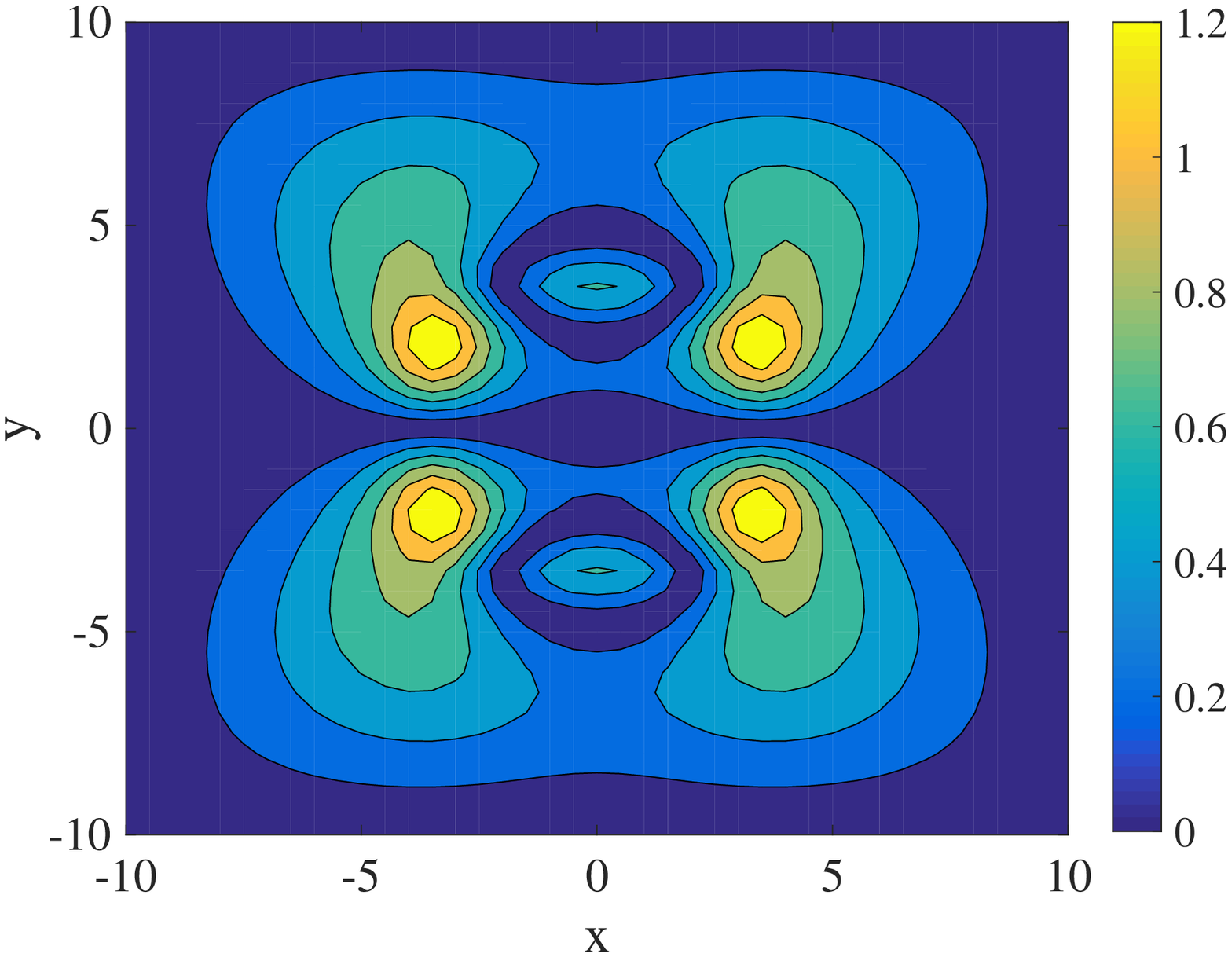}
            \end{minipage}
            }
           \centering \subfigure[]{
            \begin{minipage}[b]{0.31\textwidth}
               \centering
             \includegraphics[width=0.95\textwidth,height=1.33in]{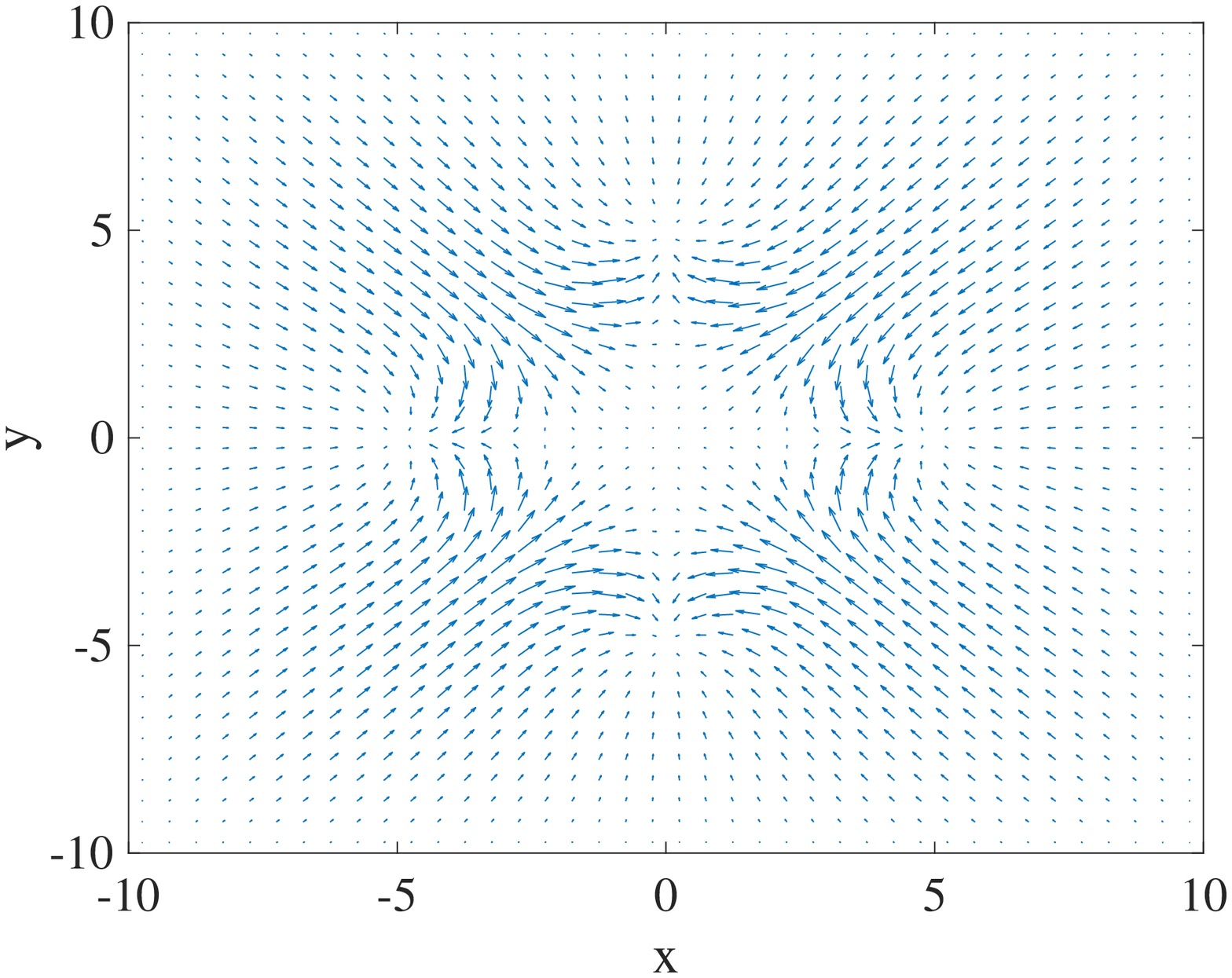}
            \end{minipage}
            }
            \centering \subfigure[]{
            \begin{minipage}[b]{0.31\textwidth}
            \centering
             \includegraphics[width=0.95\textwidth,height=1.33in]{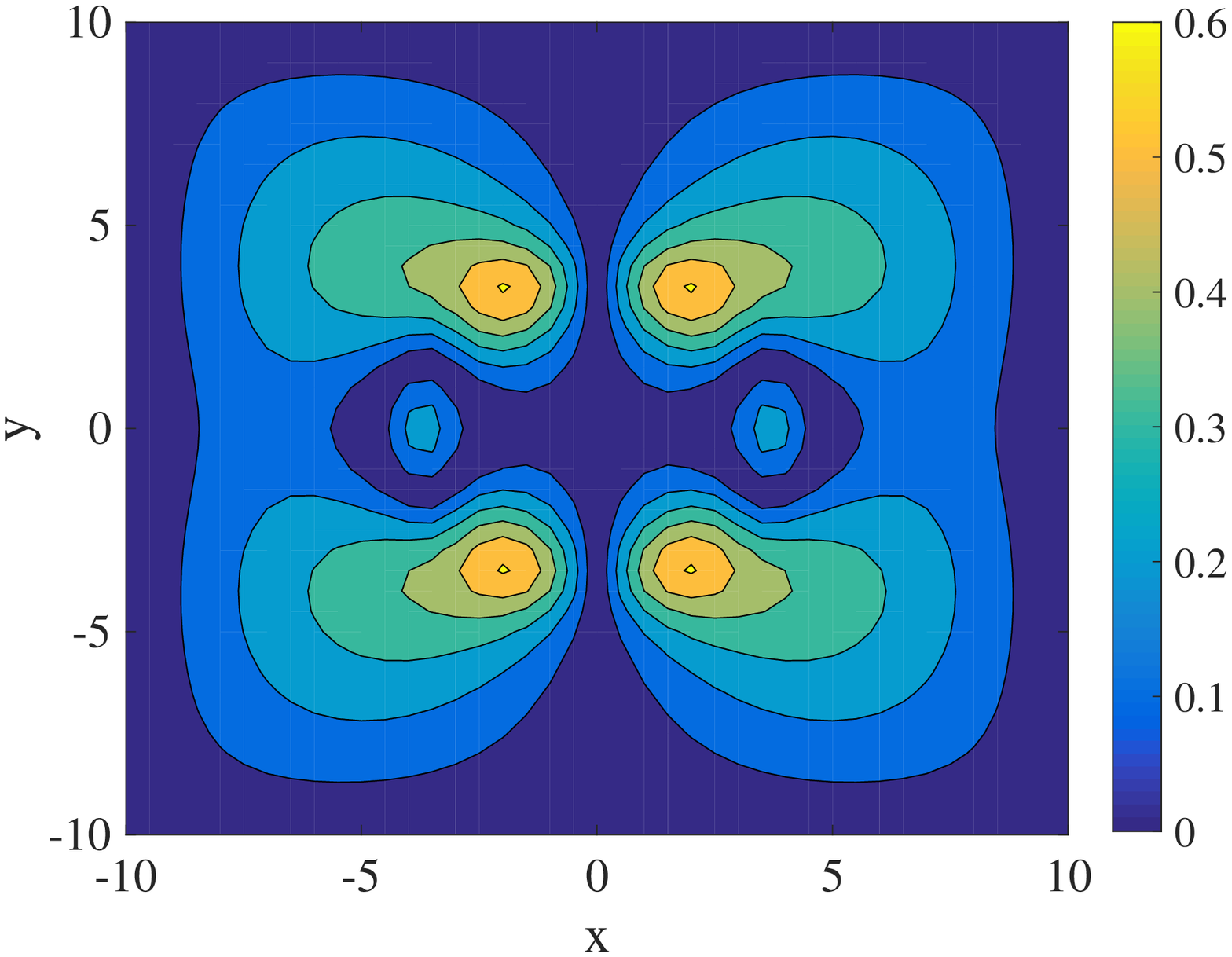}
            \end{minipage}
            }
            \centering \subfigure[]{
            \begin{minipage}[b]{0.3\textwidth}
            \centering
             \includegraphics[width=0.95\textwidth,height=1.3in]{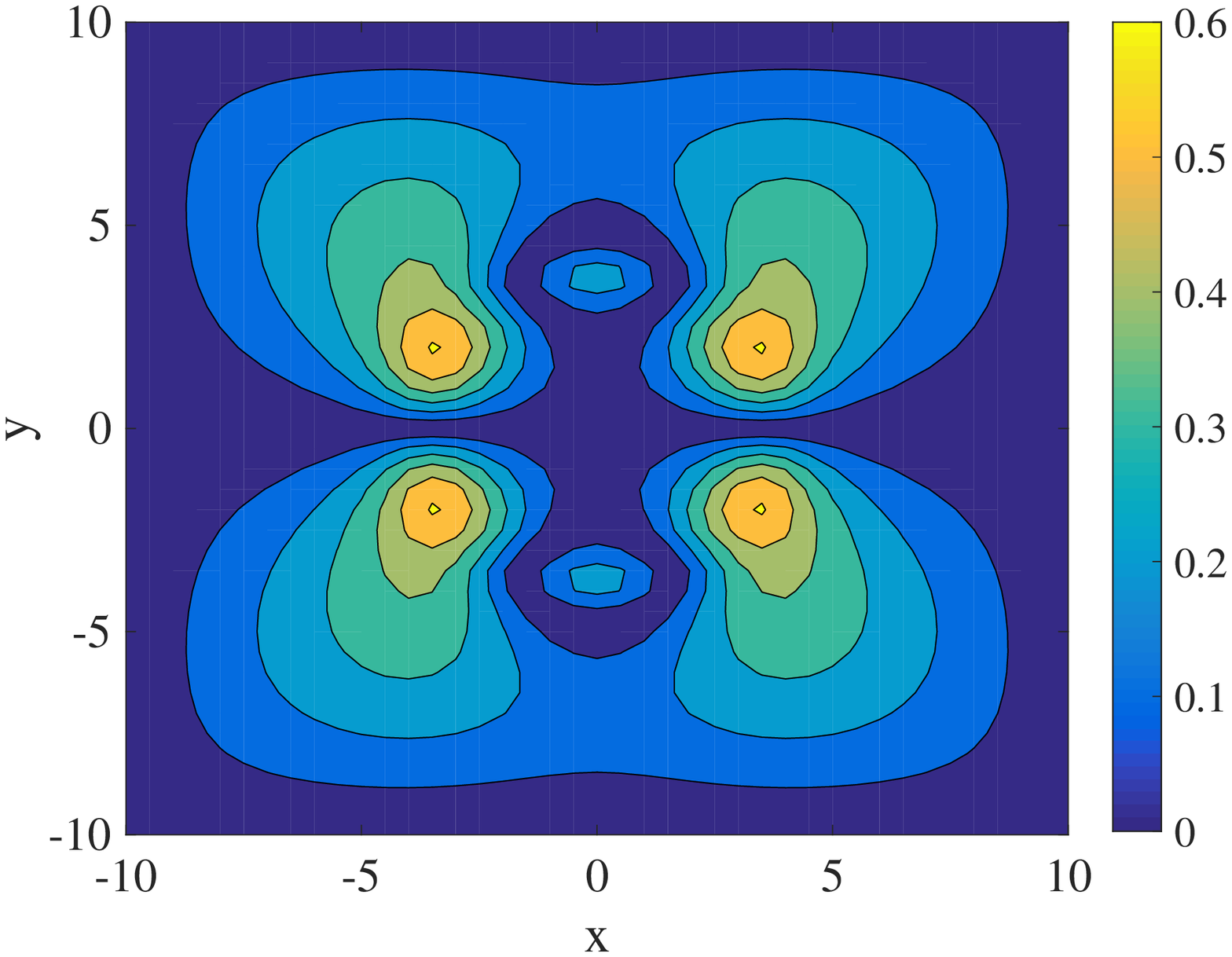}
            \end{minipage}
            }
           \centering \subfigure[]{
            \begin{minipage}[b]{0.31\textwidth}
               \centering
             \includegraphics[width=0.95\textwidth,height=1.33in]{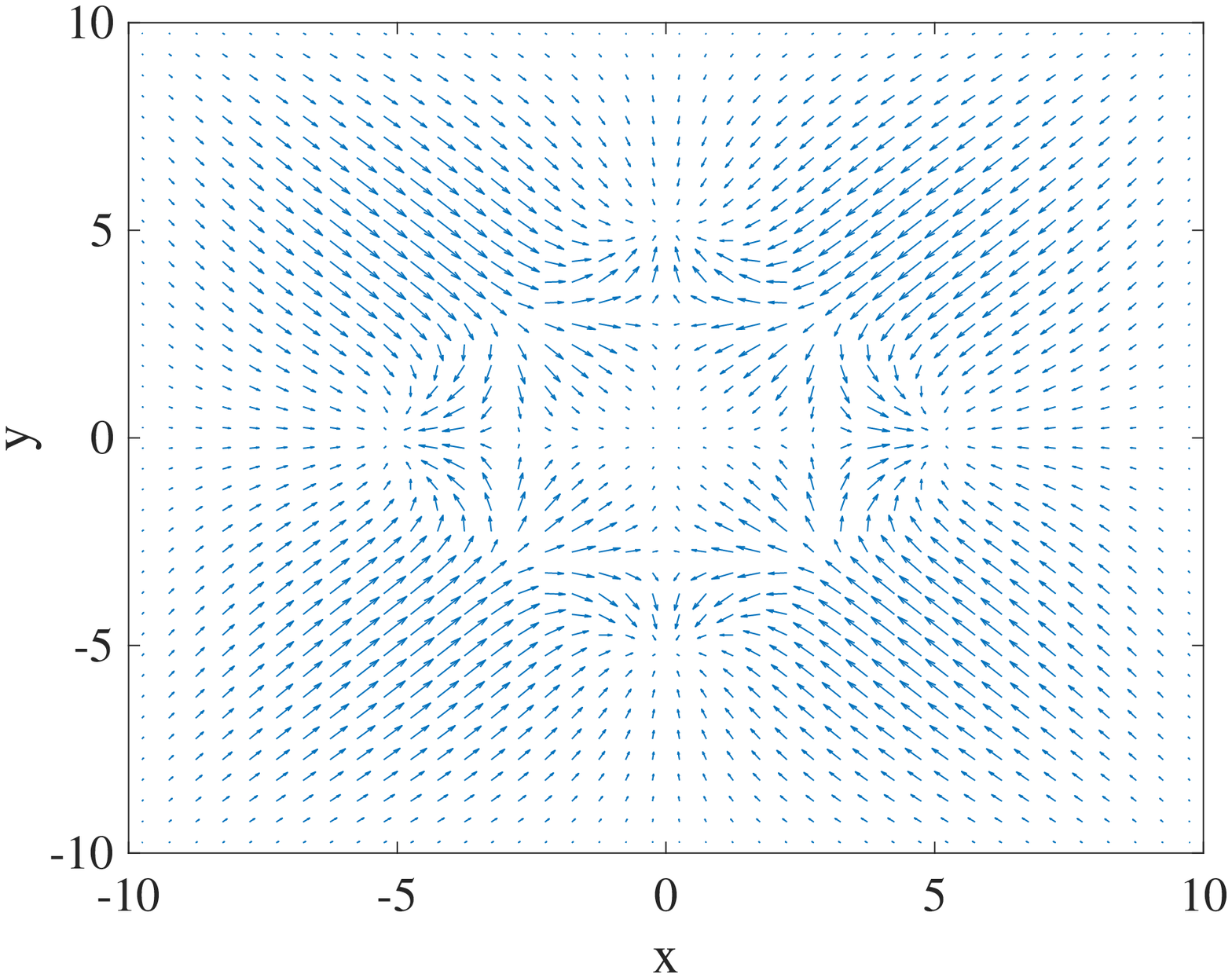}
            \end{minipage}
            }
            \centering \subfigure[]{
            \begin{minipage}[b]{0.31\textwidth}
            \centering
             \includegraphics[width=0.95\textwidth,height=1.33in]{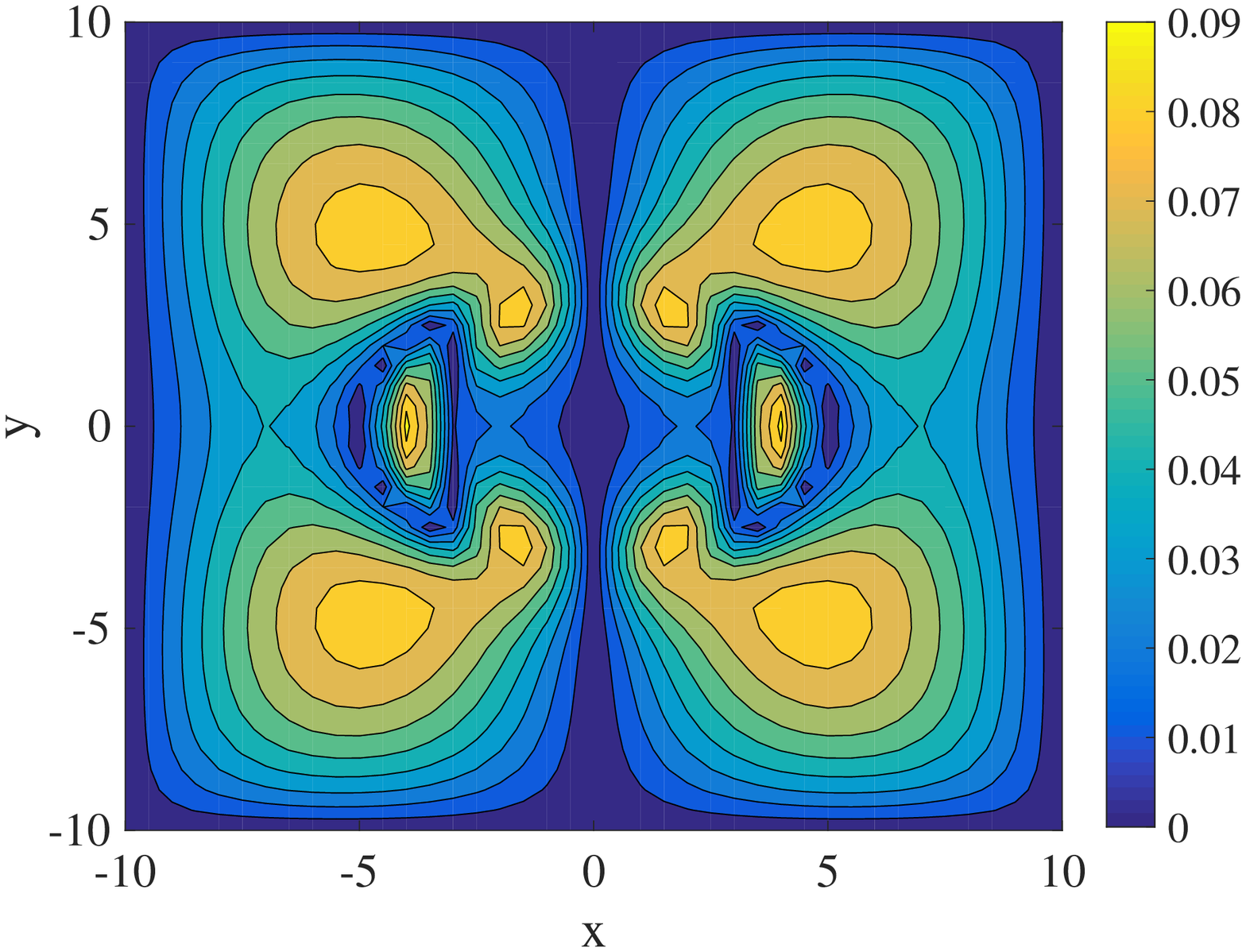}
            \end{minipage}
            }
            \centering \subfigure[]{
            \begin{minipage}[b]{0.3\textwidth}
            \centering
             \includegraphics[width=0.95\textwidth,height=1.3in]{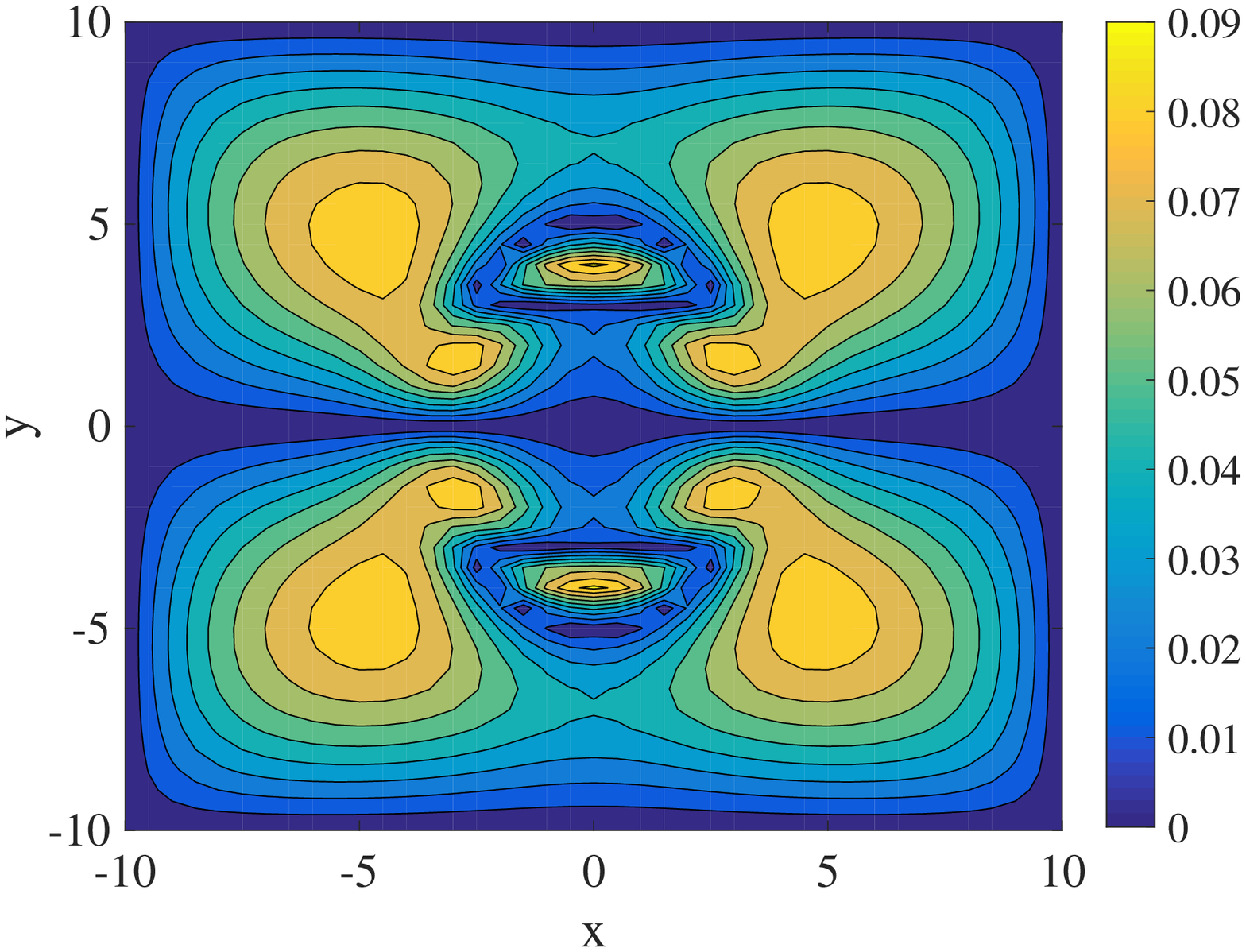}
            \end{minipage}
            }
           \caption{Dynamics of an isolated system:   the flow quivers (left column), magnitude contours of x-direction velocity component (center column), and magnitude contours of y-direction velocity component (right column)   at the 50th(top row), 200th(center row), and 500th(bottom row) time step  respectively.}
            \label{IsolatedSysnC4Velocity}
 \end{figure}

%%%%%%%%%%%%%%%%
\subsection{Bubble dropping   under a boundary temperature contrast}
%%%%%%%%%%%%%%%%%%%
 In this example,  we consider  a square  domain $\Omega=(-L,L)^2$, where $L=10$nm, and we use a  uniform  rectangular mesh with $40\times40$ elements.  The initial molar density is distributed by the following function
 $$n = \frac{1}{2}(n_L+n_G) + \frac{1}{2}(n_L-n_G)\tanh(w(d-r)/L),$$
 where $w=10^5$, $r=0.45L$ and   $d=\sqrt{x^2+y^2}$. The discrete distribution of the initial molar density is illustrated in Figure \ref{BubbleDropingnC4MolarDensityOfnC4iT0}.
 
 We partition the domain boundary $\partial\Omega$   into three non-overlapping subdivisions as
 $$\partial\Omega=\Gamma_n\cup\Gamma_t\cup\Gamma_b,$$
 where 
$$\Gamma_n=\left\{\partial\Omega\cap\{\x\in \mathbb{R}^2|x=-L\}\right\}\cup\left\{\partial\Omega\cap\{\x\in \mathbb{R}^2|x=L\}\right\},$$
$$\Gamma_t=\partial\Omega\cap\left\{\x\in \mathbb{R}^2|y=L\right\},$$
$$\Gamma_b=\partial\Omega\cap\left\{\x\in \mathbb{R}^2|y=-L\right\}.$$
The initial temperature is uniformly equal to 345K inside $\Omega$, 
and the boundary  conditions are imposed through the simulation time, 
$$T=345\textnormal{K}~~\textnormal{on} ~~\Gamma_t,~~~T=348\textnormal{K}~~ \textnormal{on} ~~\Gamma_b,~~~\q\cdot\bm\nu_\Omega=0~~ \textnormal{on} ~~\Gamma_n.$$
Namely,  there exists a temperature contrast between the top and bottom of this domain.
The initial  velocity is  fixed to be zero.
%, and in the simulation time, we always set the boundary condition  $\u\cdot\bm\nu_\Omega=0$ on the entire boundary $\partial\Omega$.

We take a  fixed time step size $\delta t=5\times10^{-13}$s, and use 50000 time steps to simulate this problem.

In Figures \ref{BubbleDropingnC4MolarDensityOfnC4}, we illustrate the molar density profiles at various time steps. In Figures \ref{BubbleDroppingnC4TemperatureOfnC4}, the temperature profiles are depicted at various time steps. In Figures \ref{BubbleDropingnC4Velocity}, we illustrate    the velocity fields, especially  magnitudes of both velocity components at different time steps.  

From Figures \ref{BubbleDropingnC4MolarDensityOfnC4}, we can see that the initial bubble with a rough shape becomes a smooth circle, and it is gradually dropping towards the bottom. Figures \ref{BubbleDroppingnC4TemperatureOfnC4} show that there exists a temperature contrast between the top and bottom of this domain, although the temperature fields vary due to the bubble motion. Figures \ref{BubbleDropingnC4Velocity} depict  that the velocity fields are generated by the temperature contrast and the fluid flows towards the bottom. Finally, the system will reach  a steady state; i.e., the bubble spreads into a semicircle  on the heated bottom,  the velocity field vanishes and the temperature has the layered distribution along the y-direction.

\begin{figure}
            \centering \subfigure[]{
            \begin{minipage}[b]{0.31\textwidth}
               \centering
             \includegraphics[width=0.95\textwidth,height=1.35in]{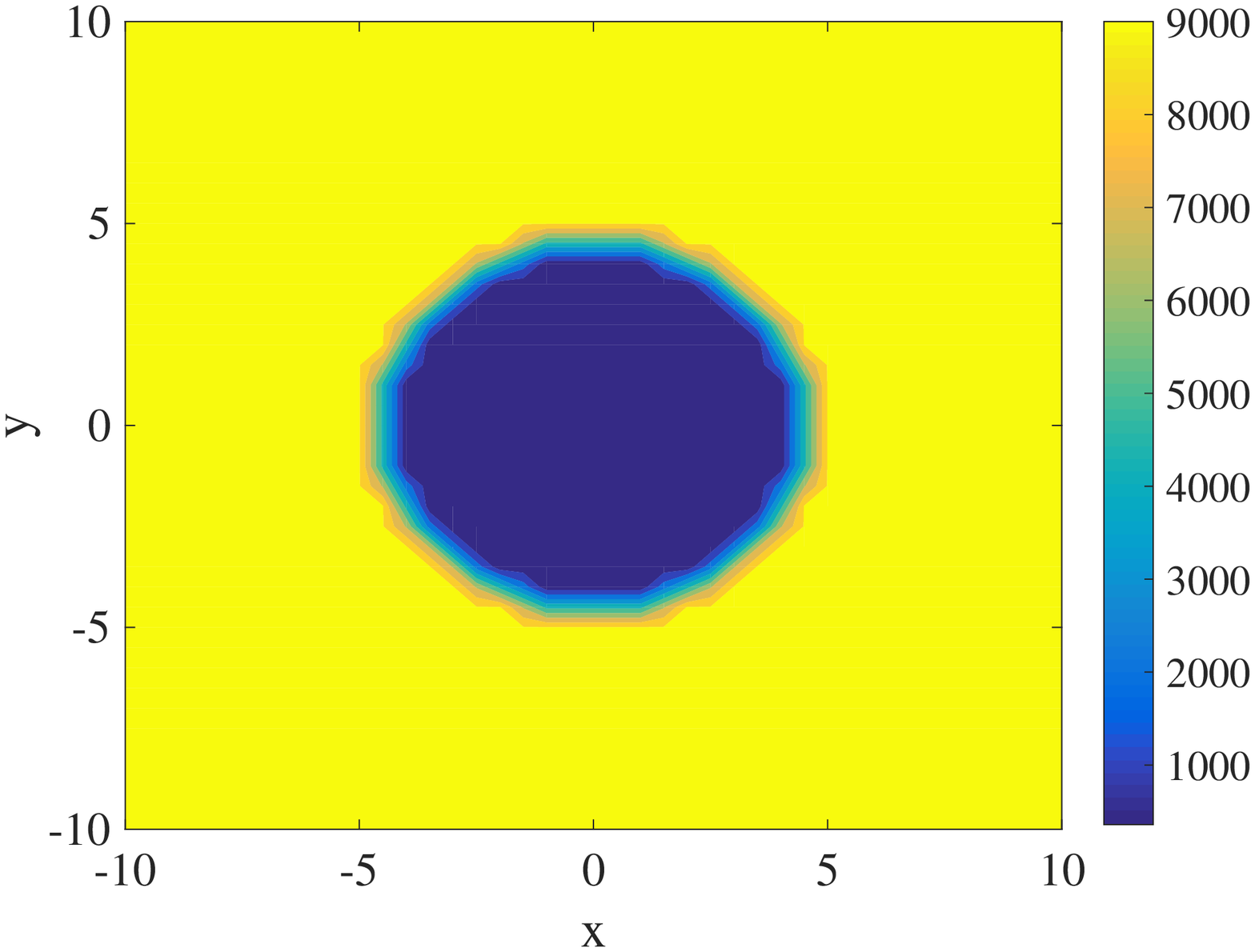}
              \label{BubbleDropingnC4MolarDensityOfnC4iT0}
            \end{minipage}
            }
            \centering \subfigure[]{
            \begin{minipage}[b]{0.31\textwidth}
            \centering
             \includegraphics[width=0.95\textwidth,height=1.35in]{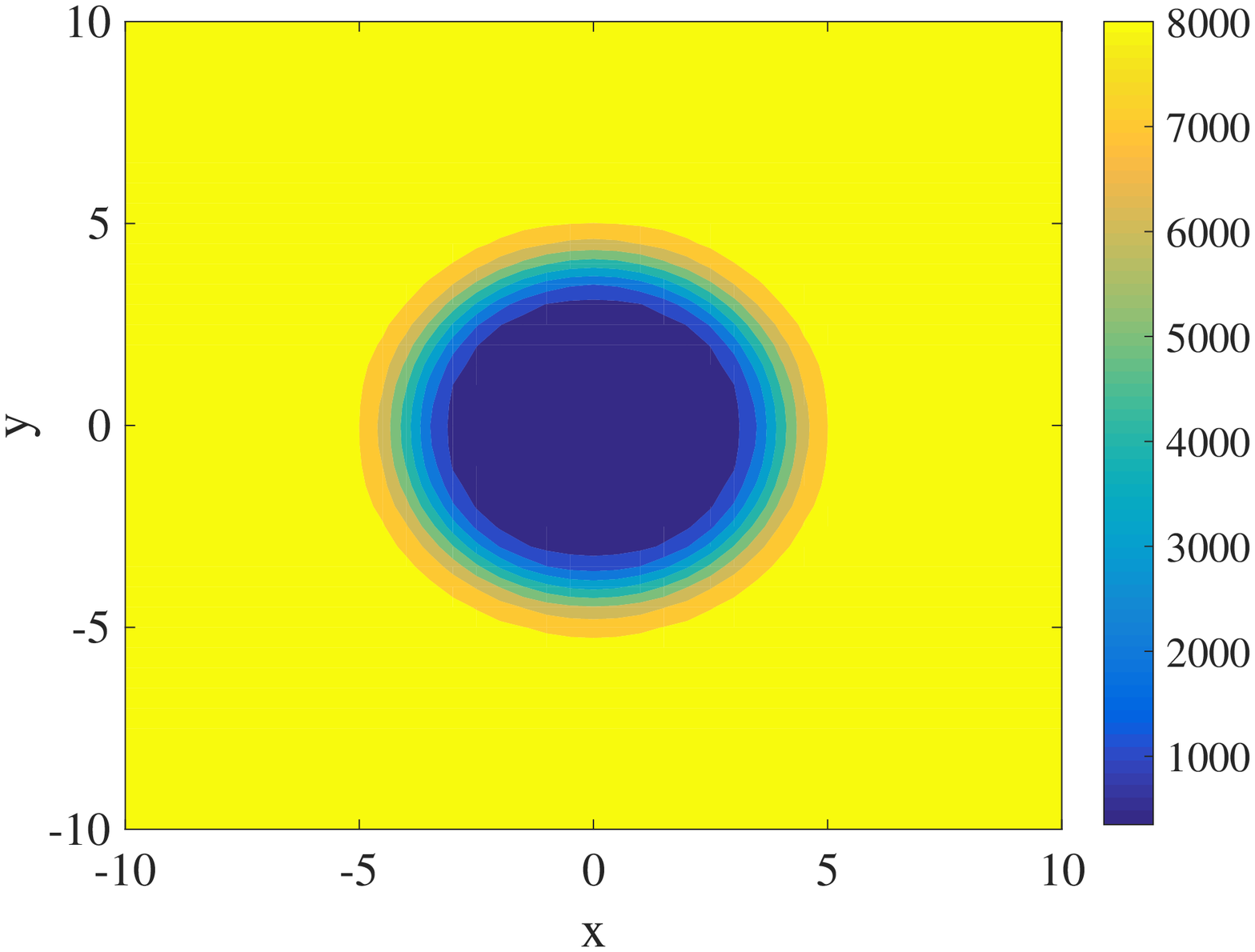}
            \end{minipage}
            }
             \centering \subfigure[]{
            \begin{minipage}[b]{0.31\textwidth}
            \centering
             \includegraphics[width=0.95\textwidth,height=1.35in]{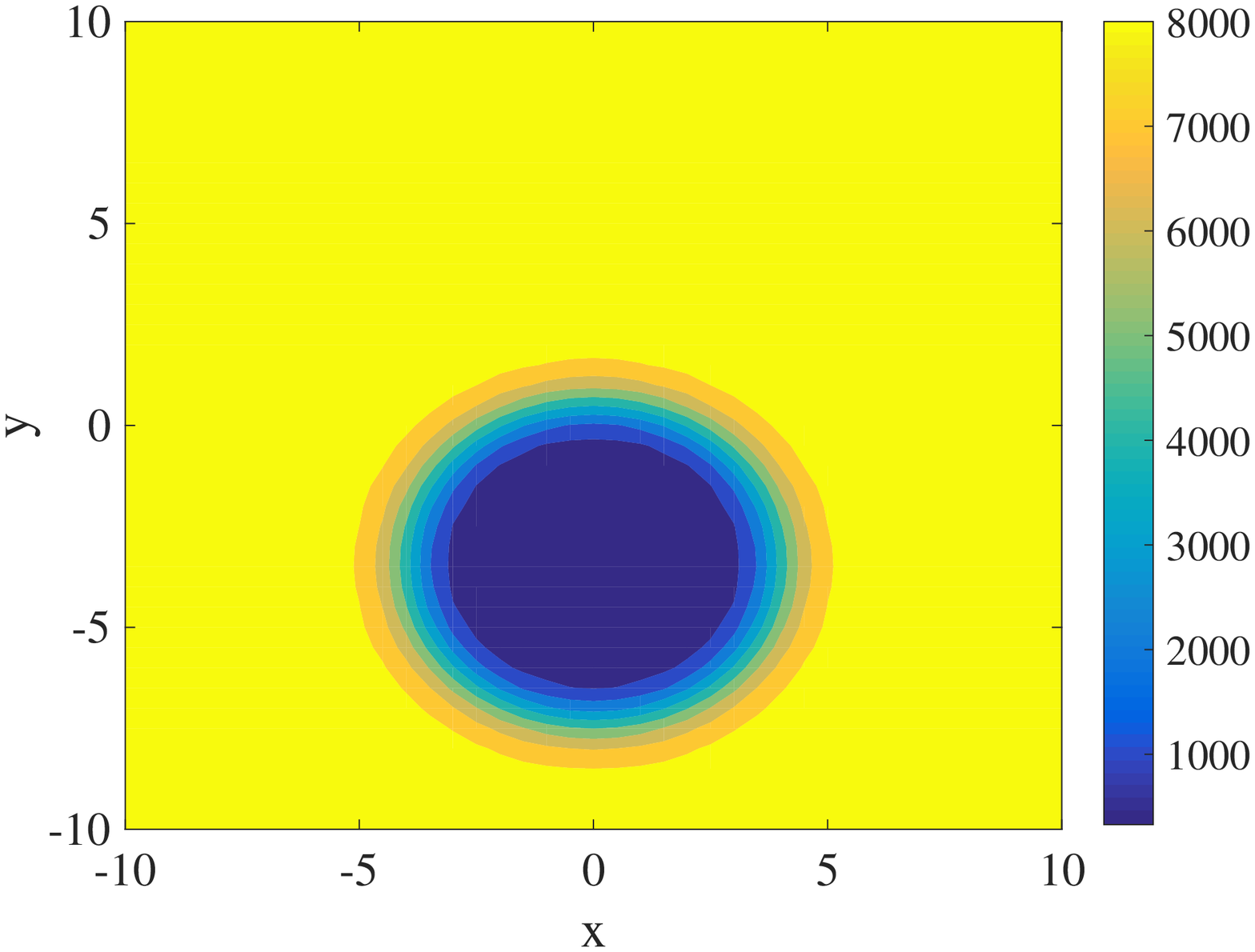}
            \end{minipage}
            }
           \centering \subfigure[]{
            \begin{minipage}[b]{0.31\textwidth}
               \centering
             \includegraphics[width=0.95\textwidth,height=1.35in]{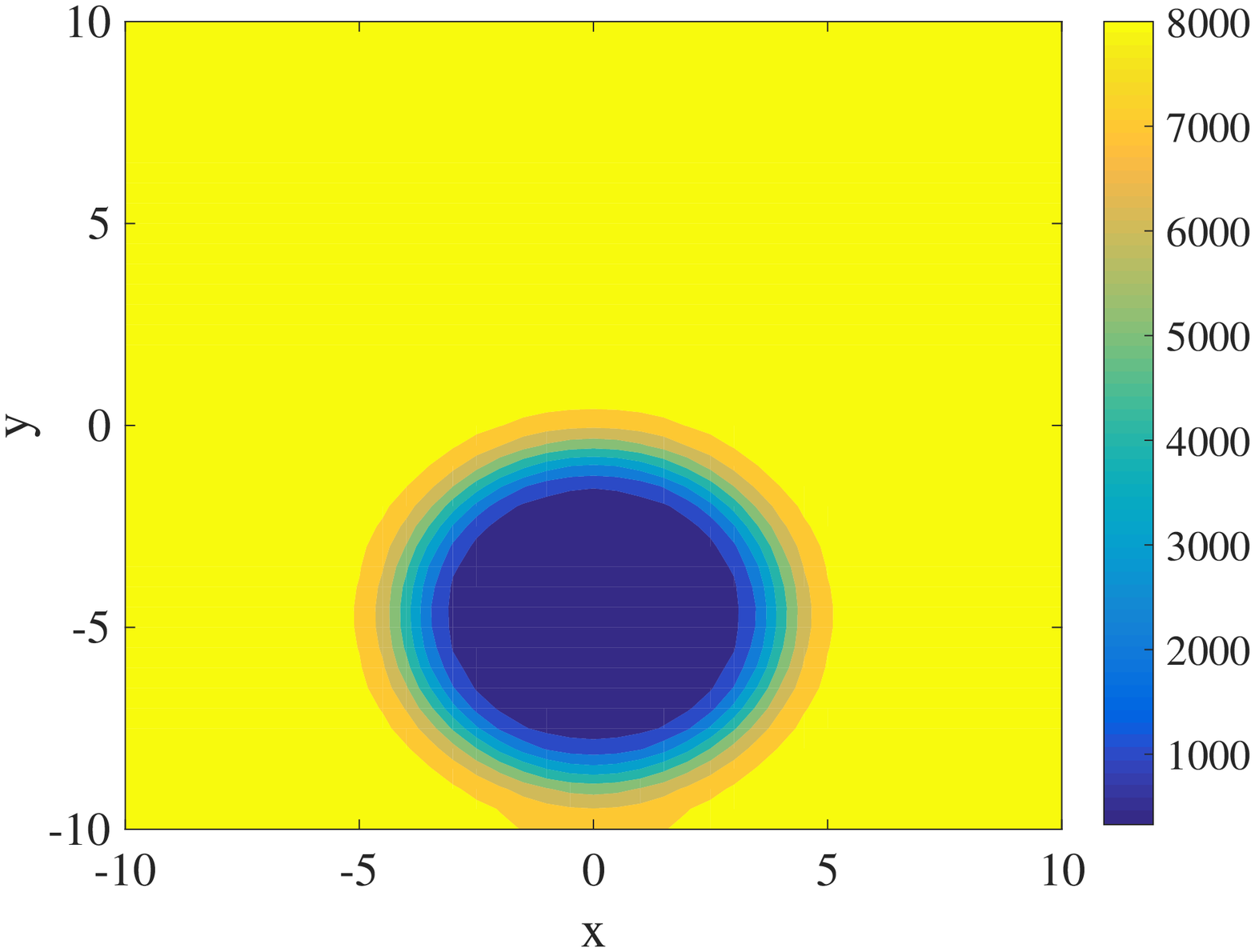}
            \end{minipage}
            }
            \centering \subfigure[]{
            \begin{minipage}[b]{0.3\textwidth}
            \centering
             \includegraphics[width=0.95\textwidth,height=1.35in]{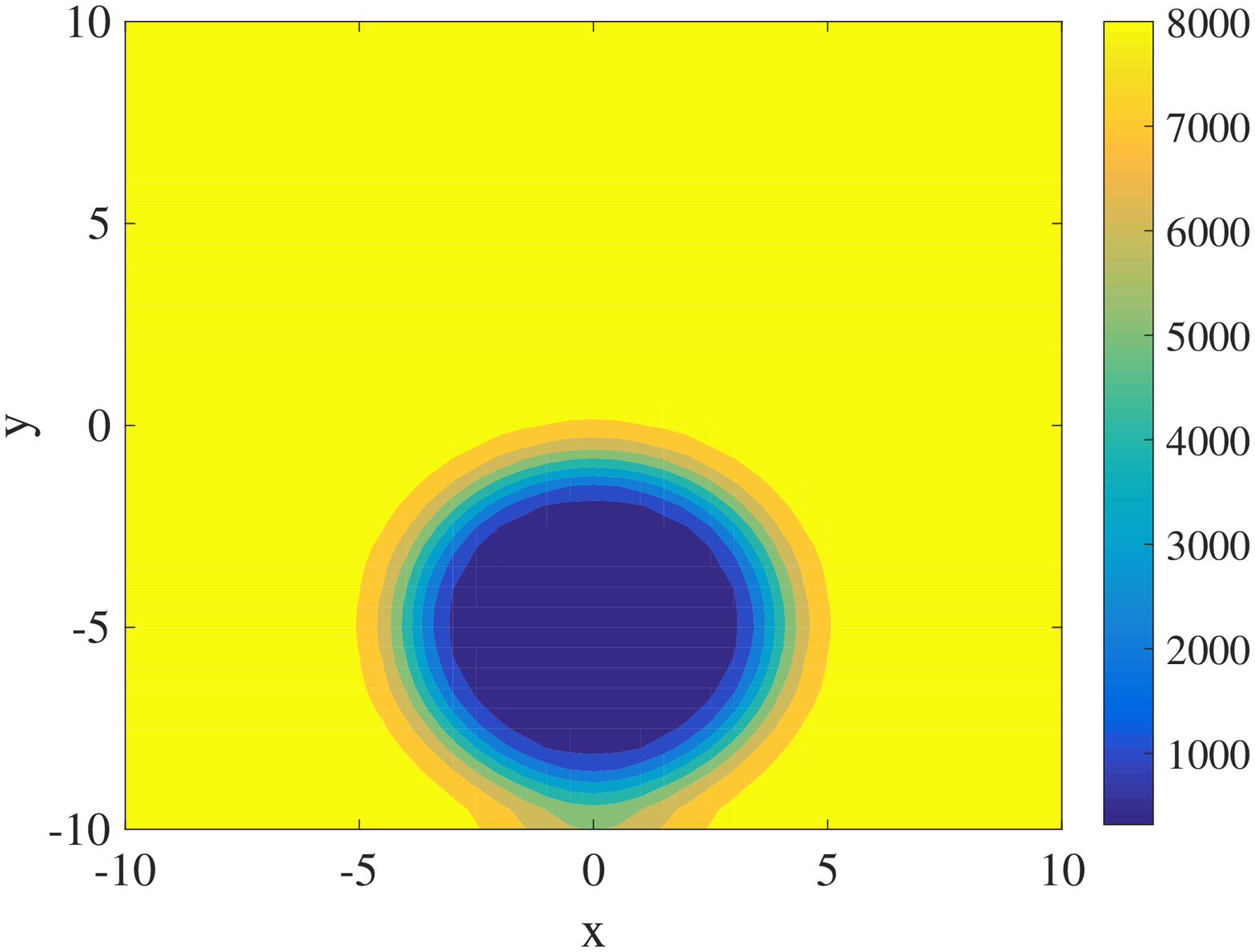}
            \end{minipage}
            }
             \centering \subfigure[]{
            \begin{minipage}[b]{0.3\textwidth}
            \centering
             \includegraphics[width=0.95\textwidth,height=1.35in]{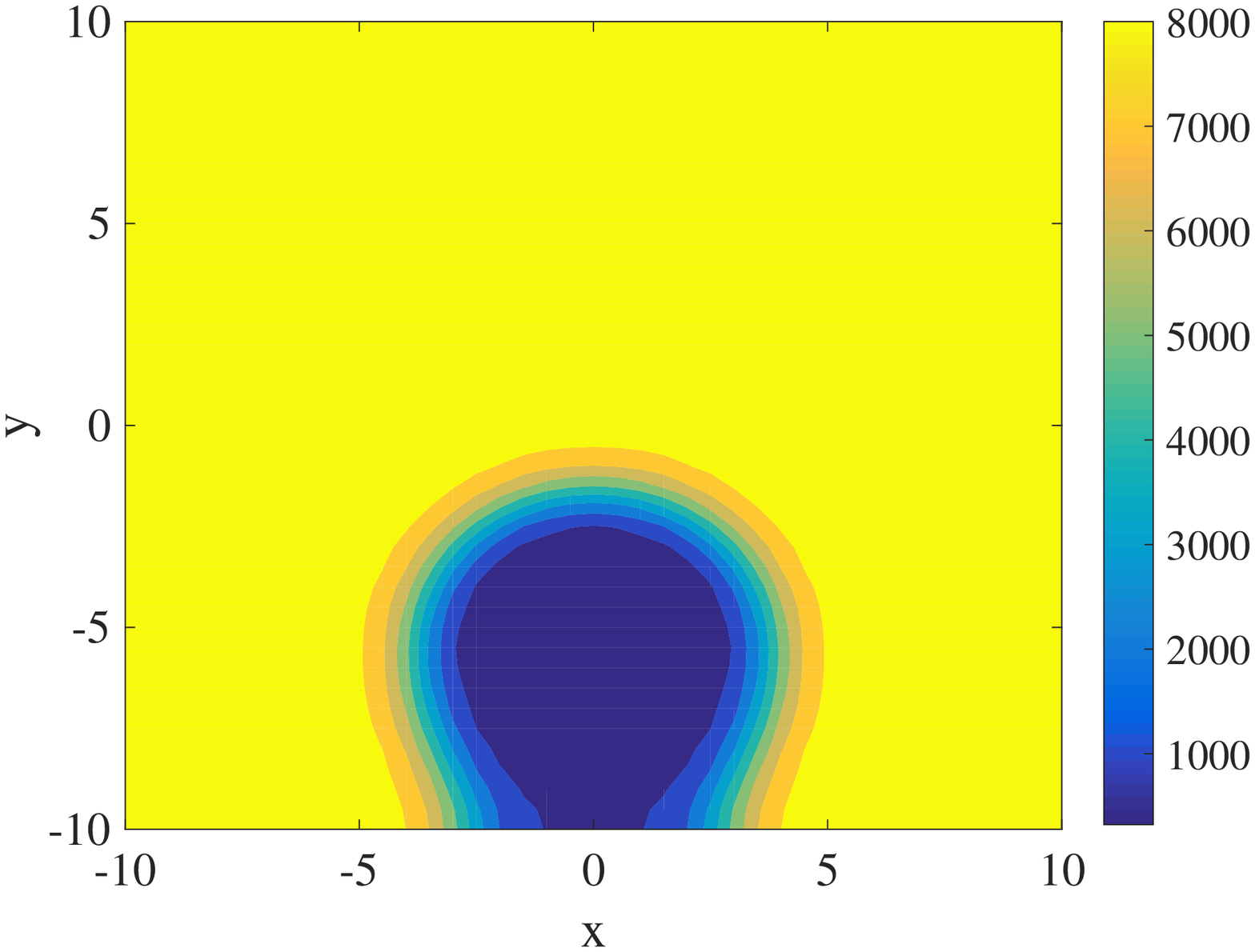}
            \end{minipage}
            }
           \centering \subfigure[]{
            \begin{minipage}[b]{0.3\textwidth}
               \centering
             \includegraphics[width=0.95\textwidth,height=1.35in]{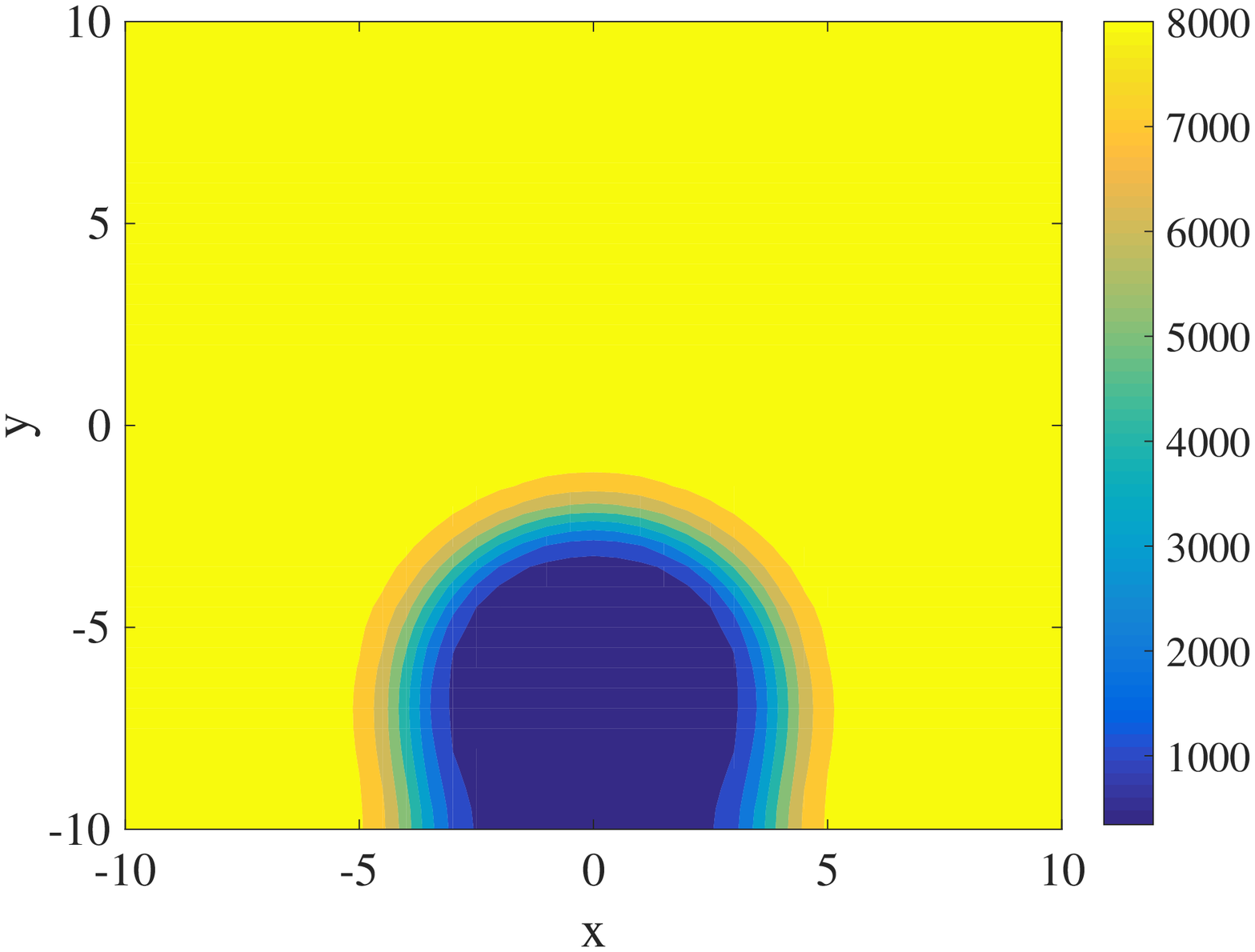}
            \end{minipage}
            }
            \centering \subfigure[]{
            \begin{minipage}[b]{0.3\textwidth}
            \centering
             \includegraphics[width=0.95\textwidth,height=1.35in]{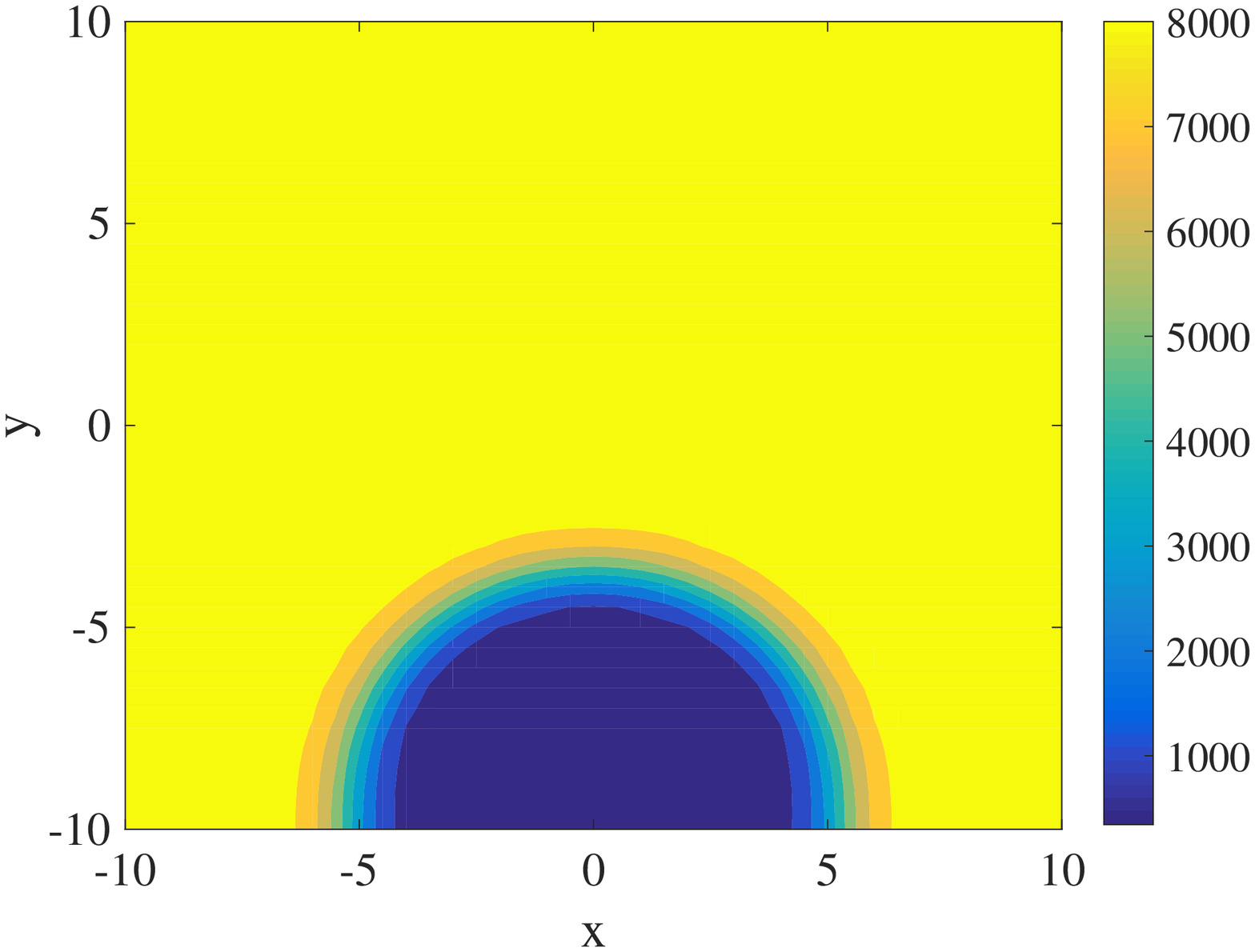}
            \end{minipage}
            }
             \centering \subfigure[]{
            \begin{minipage}[b]{0.3\textwidth}
            \centering
             \includegraphics[width=0.95\textwidth,height=1.35in]{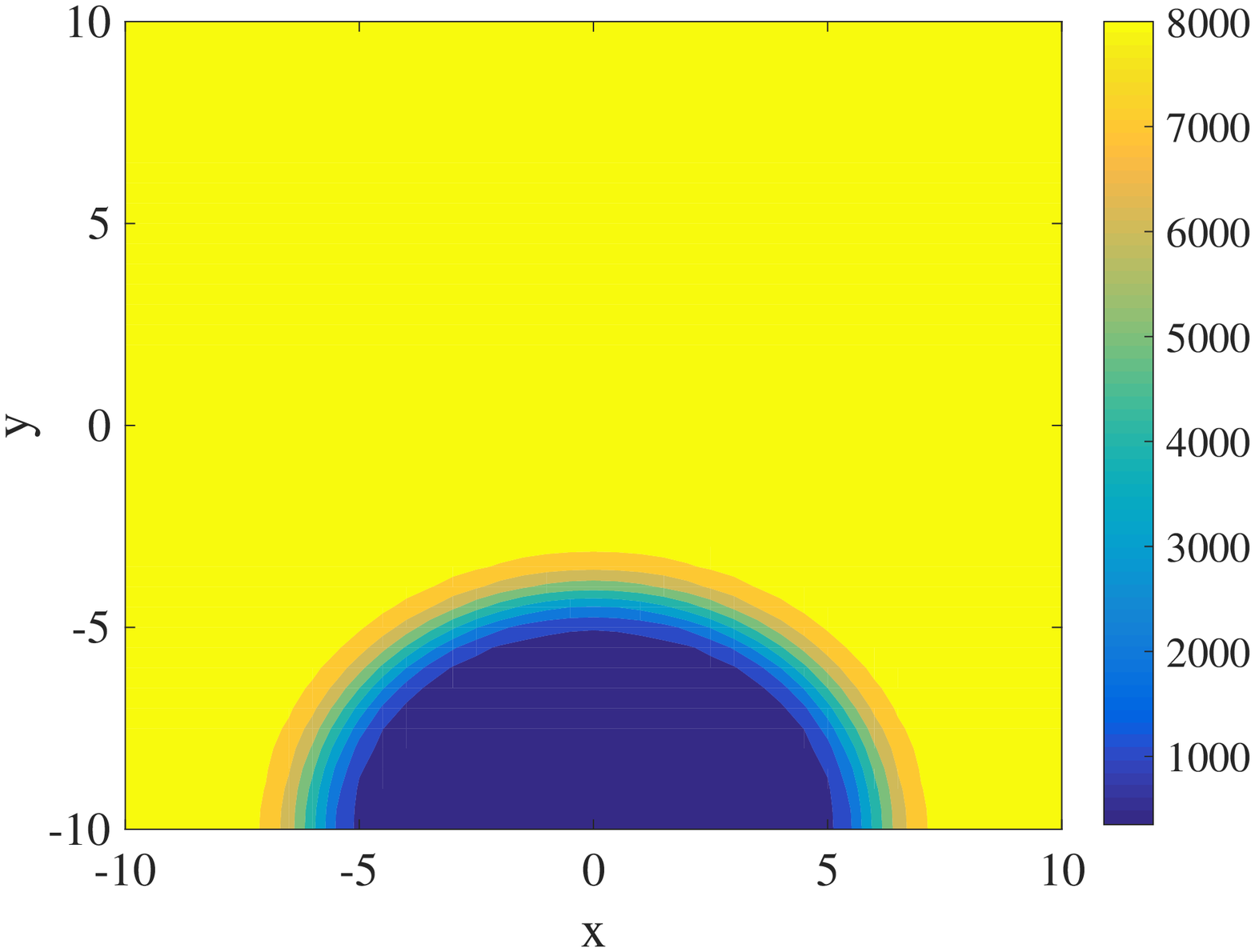}
            \end{minipage}
            }
           \caption{Bubble dropping problem: the molar density profiles at the initial time(a), 1000th(b), 20000th(c), 28000th(d), 29000th(e), 30000th(f), 30500th(g), 32000th(h) and 50000th(i) time step respectively.}
            \label{BubbleDropingnC4MolarDensityOfnC4}
 \end{figure}

\begin{figure}
            \centering \subfigure[]{
            \begin{minipage}[b]{0.31\textwidth}
               \centering
             \includegraphics[width=0.95\textwidth,height=1.35in]{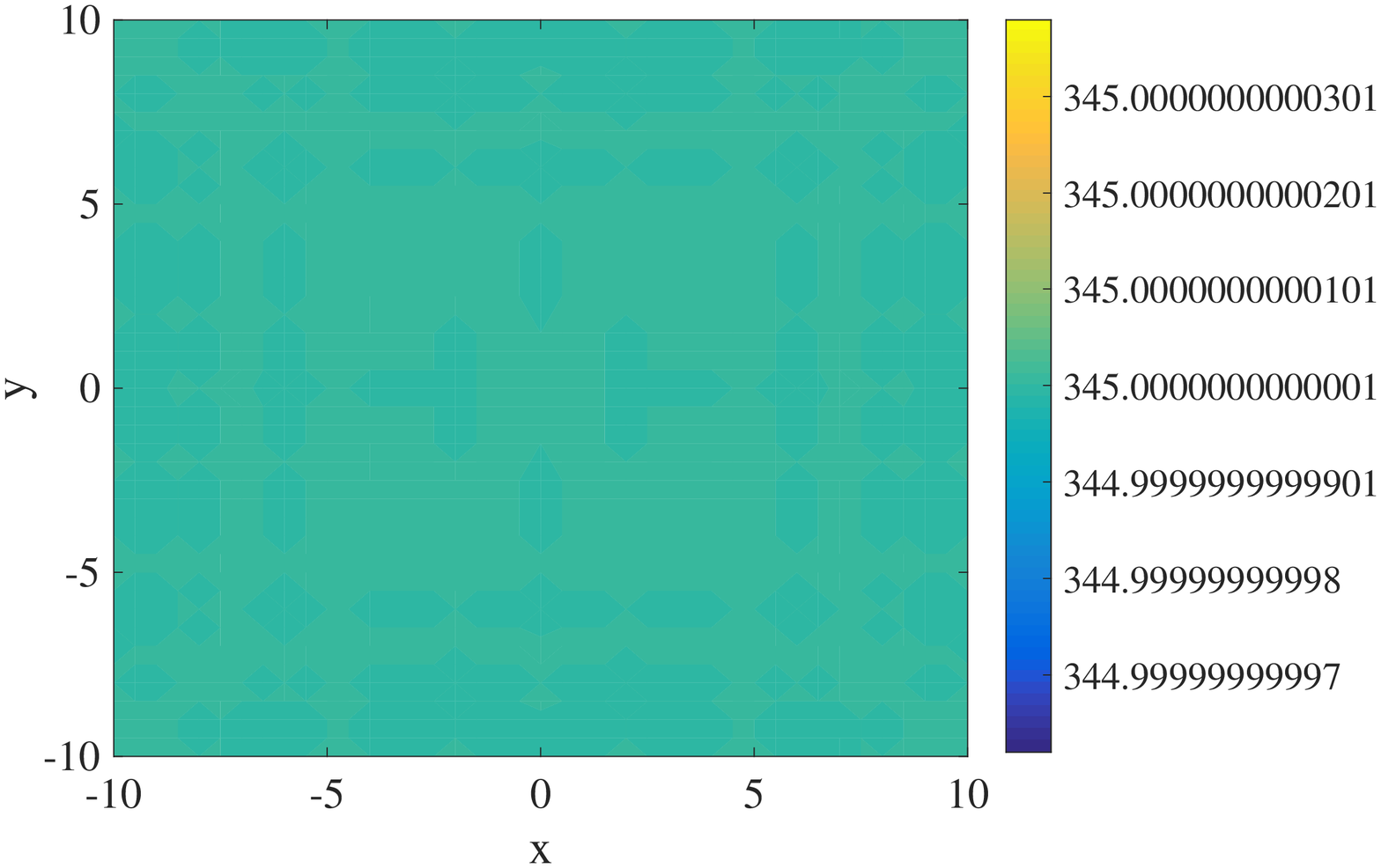}
            \end{minipage}
            }
            \centering \subfigure[]{
            \begin{minipage}[b]{0.31\textwidth}
            \centering
             \includegraphics[width=0.95\textwidth,height=1.35in]{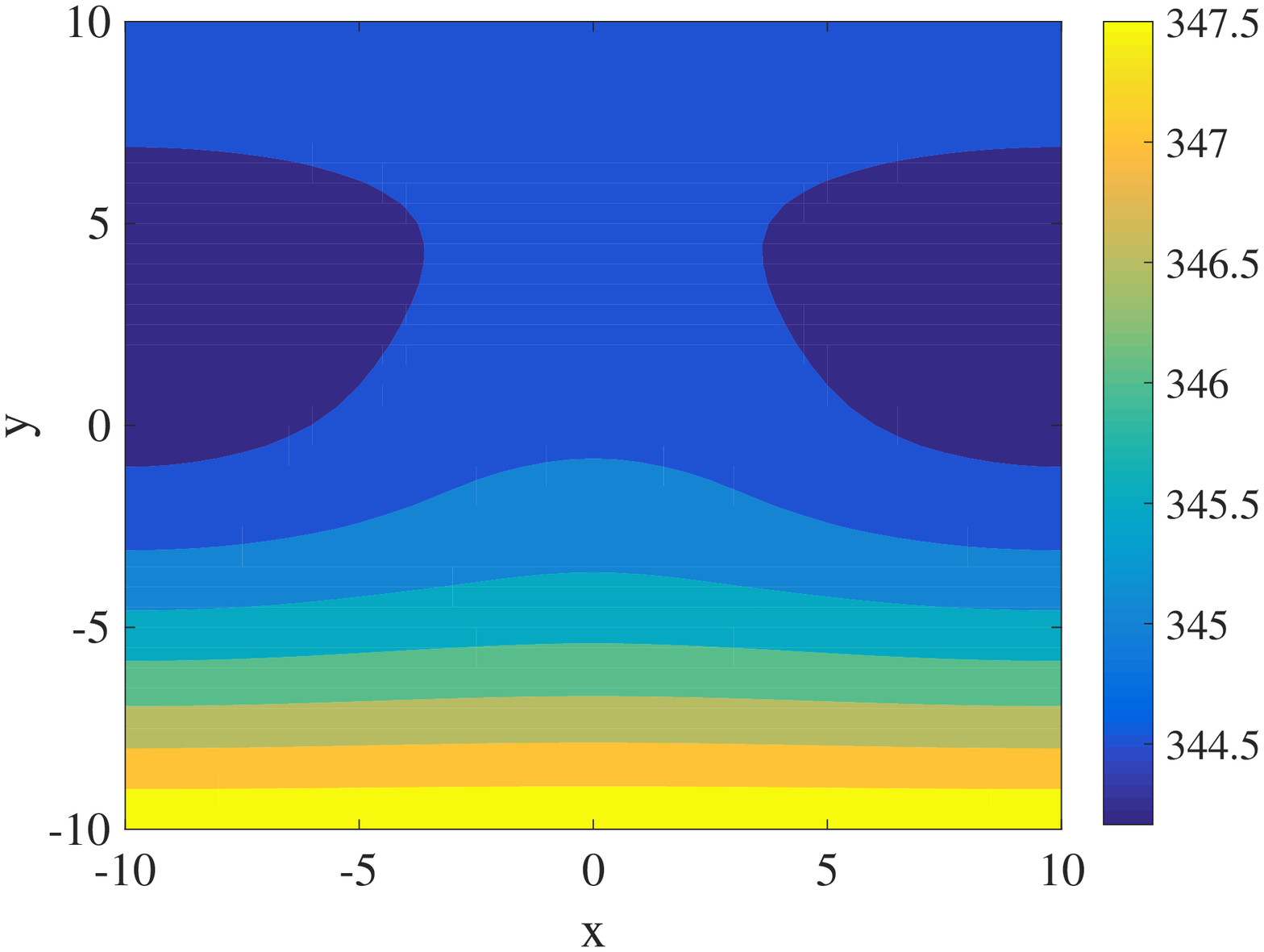}
            \end{minipage}
            }
             \centering \subfigure[]{
            \begin{minipage}[b]{0.31\textwidth}
            \centering
             \includegraphics[width=0.95\textwidth,height=1.35in]{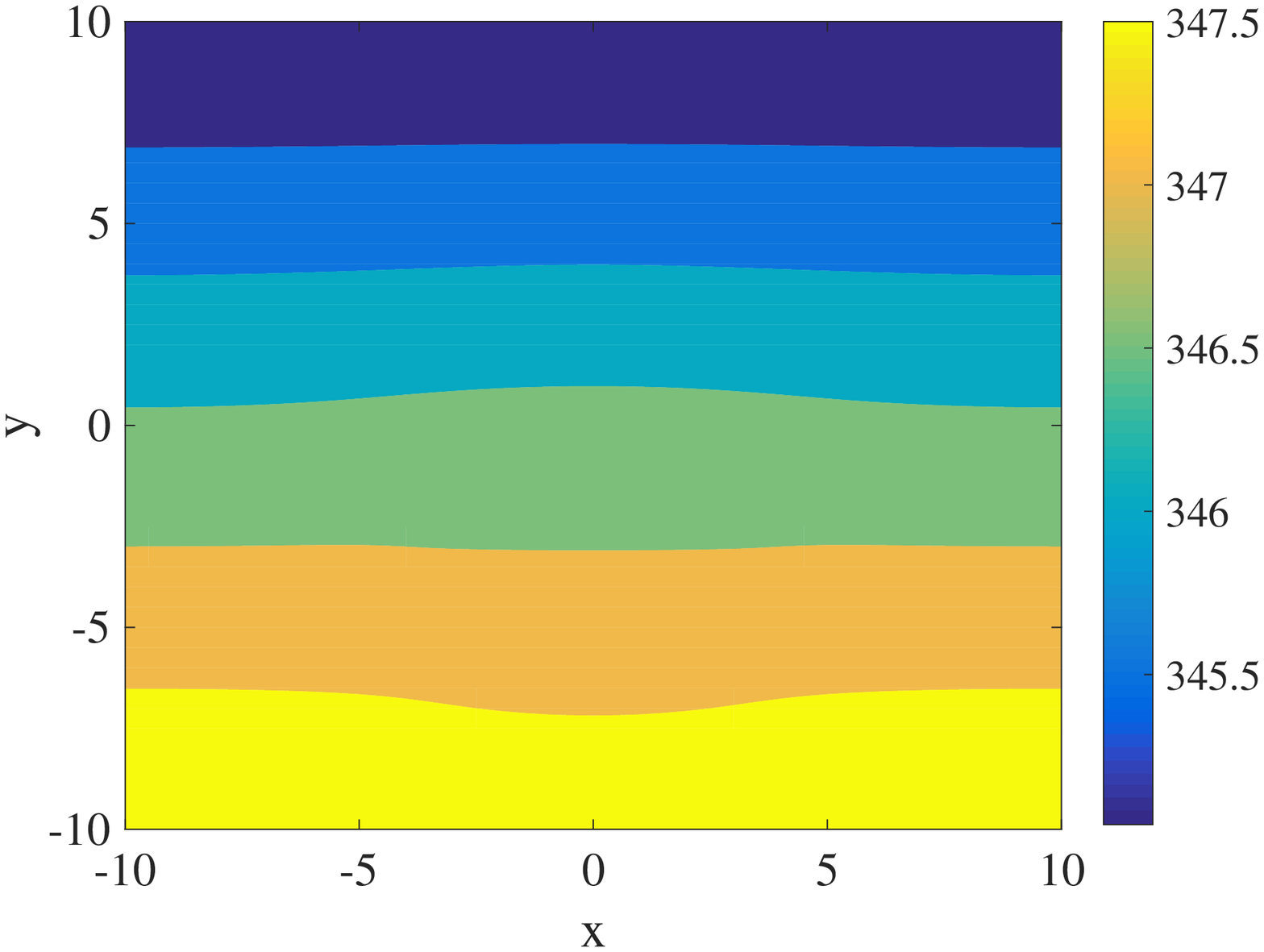}
            \end{minipage}
            }
           \centering \subfigure[]{
            \begin{minipage}[b]{0.31\textwidth}
               \centering
             \includegraphics[width=0.95\textwidth,height=1.35in]{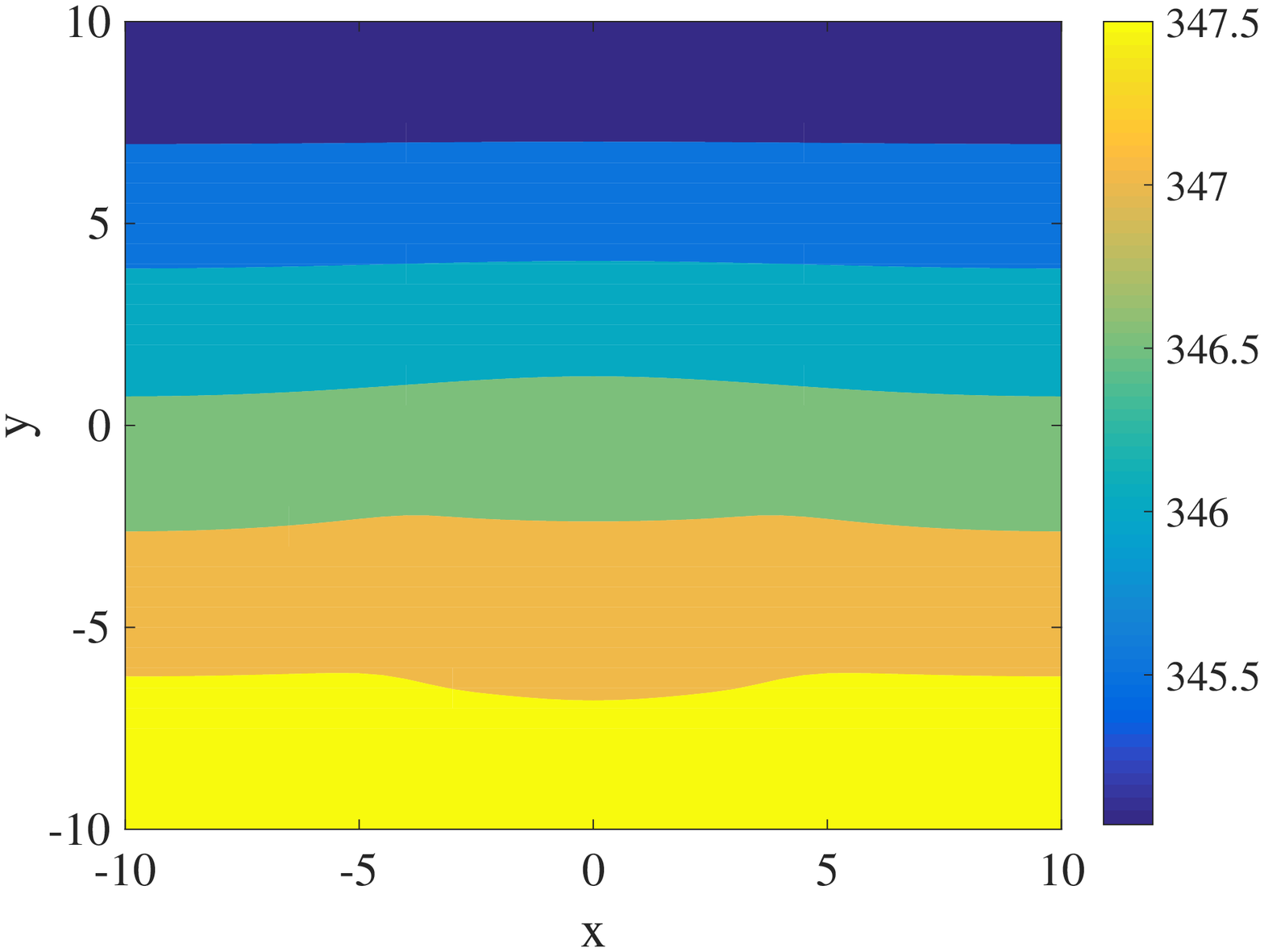}
            \end{minipage}
            }
            \centering \subfigure[]{
            \begin{minipage}[b]{0.3\textwidth}
            \centering
             \includegraphics[width=0.95\textwidth,height=1.35in]{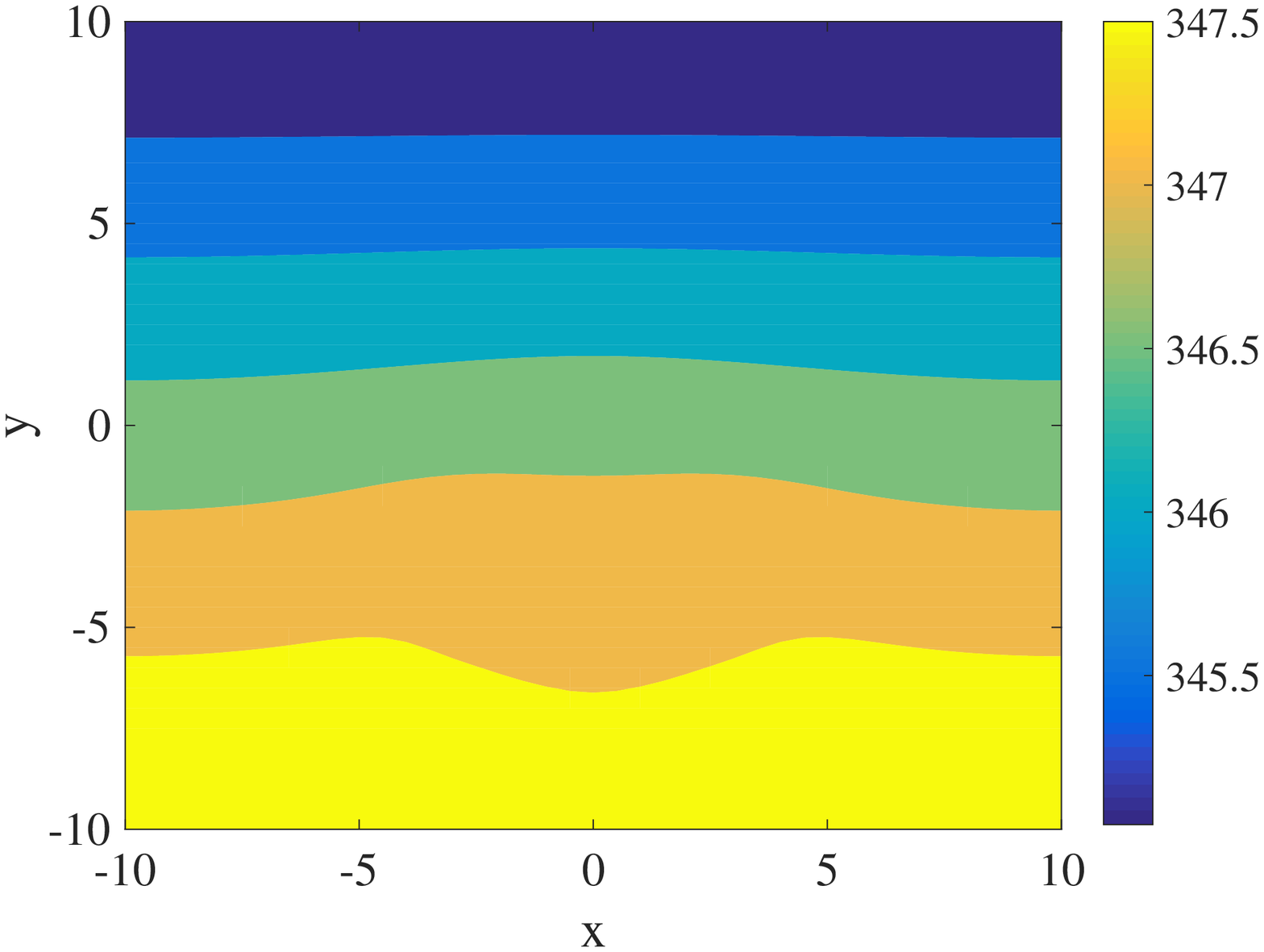}
            \end{minipage}
            }
             \centering \subfigure[]{
            \begin{minipage}[b]{0.3\textwidth}
            \centering
             \includegraphics[width=0.95\textwidth,height=1.35in]{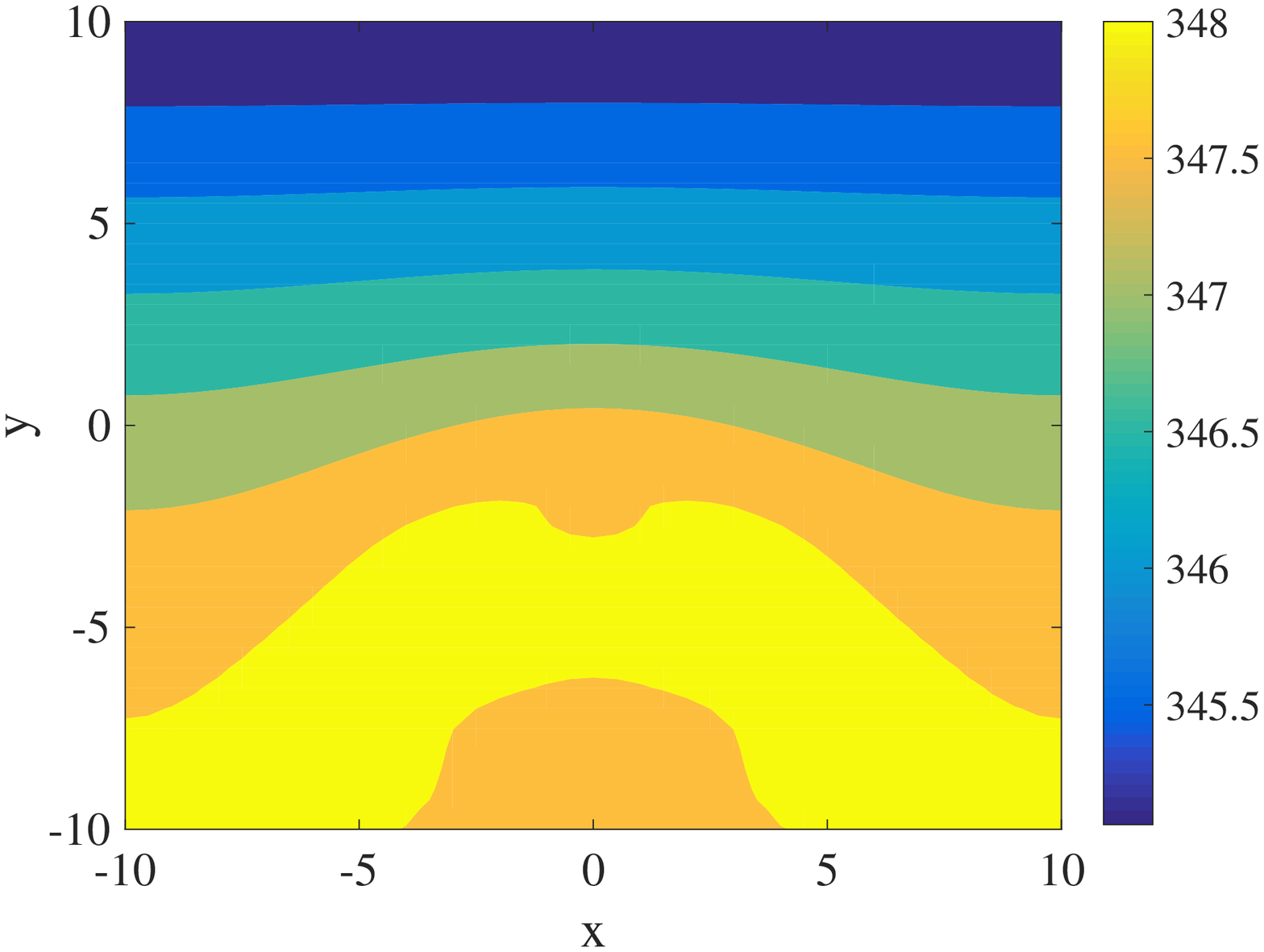}
            \end{minipage}
            }
           \centering \subfigure[]{
            \begin{minipage}[b]{0.3\textwidth}
               \centering
             \includegraphics[width=0.95\textwidth,height=1.35in]{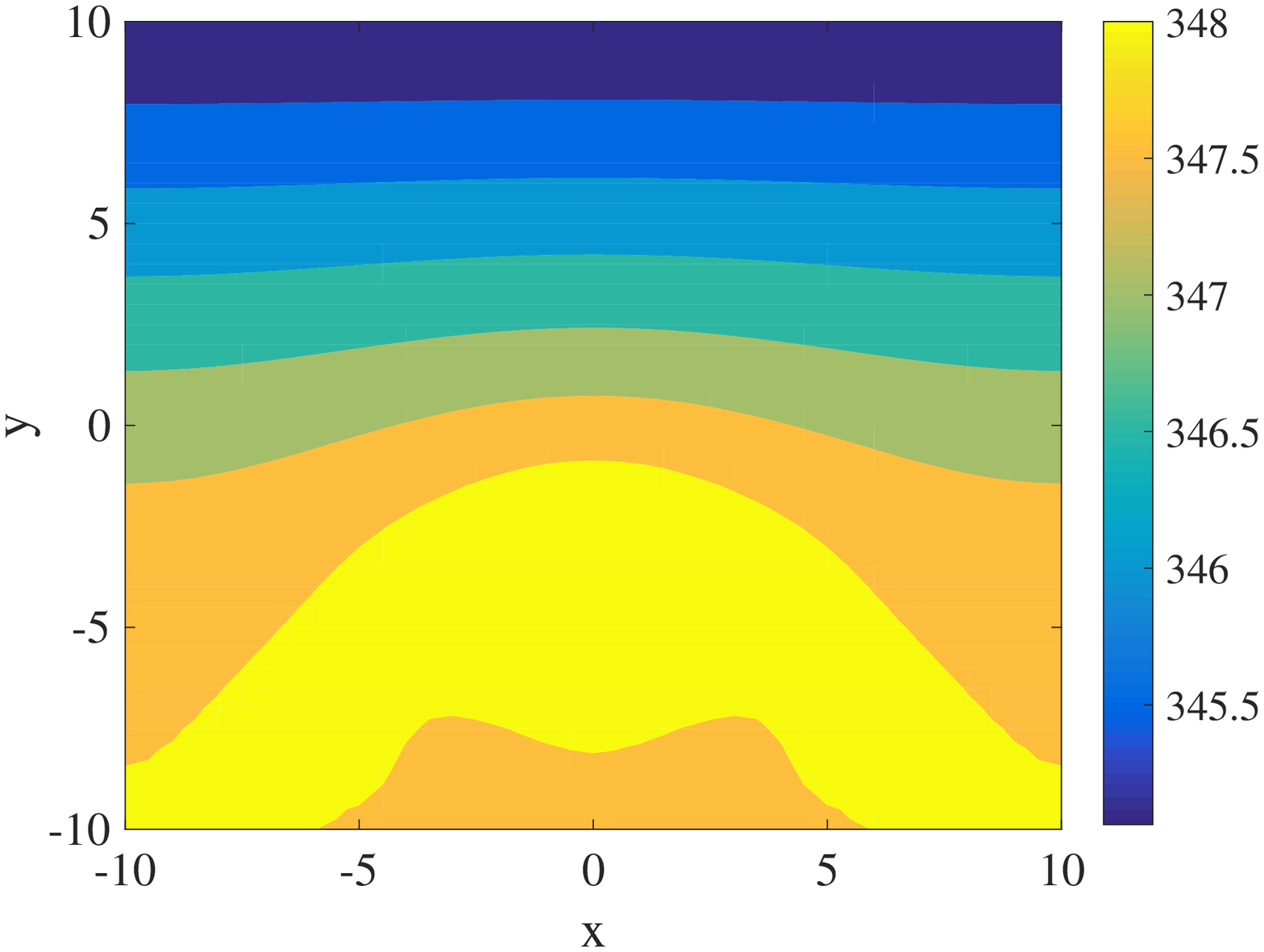}
            \end{minipage}
            }
            \centering \subfigure[]{
            \begin{minipage}[b]{0.3\textwidth}
            \centering
             \includegraphics[width=0.95\textwidth,height=1.35in]{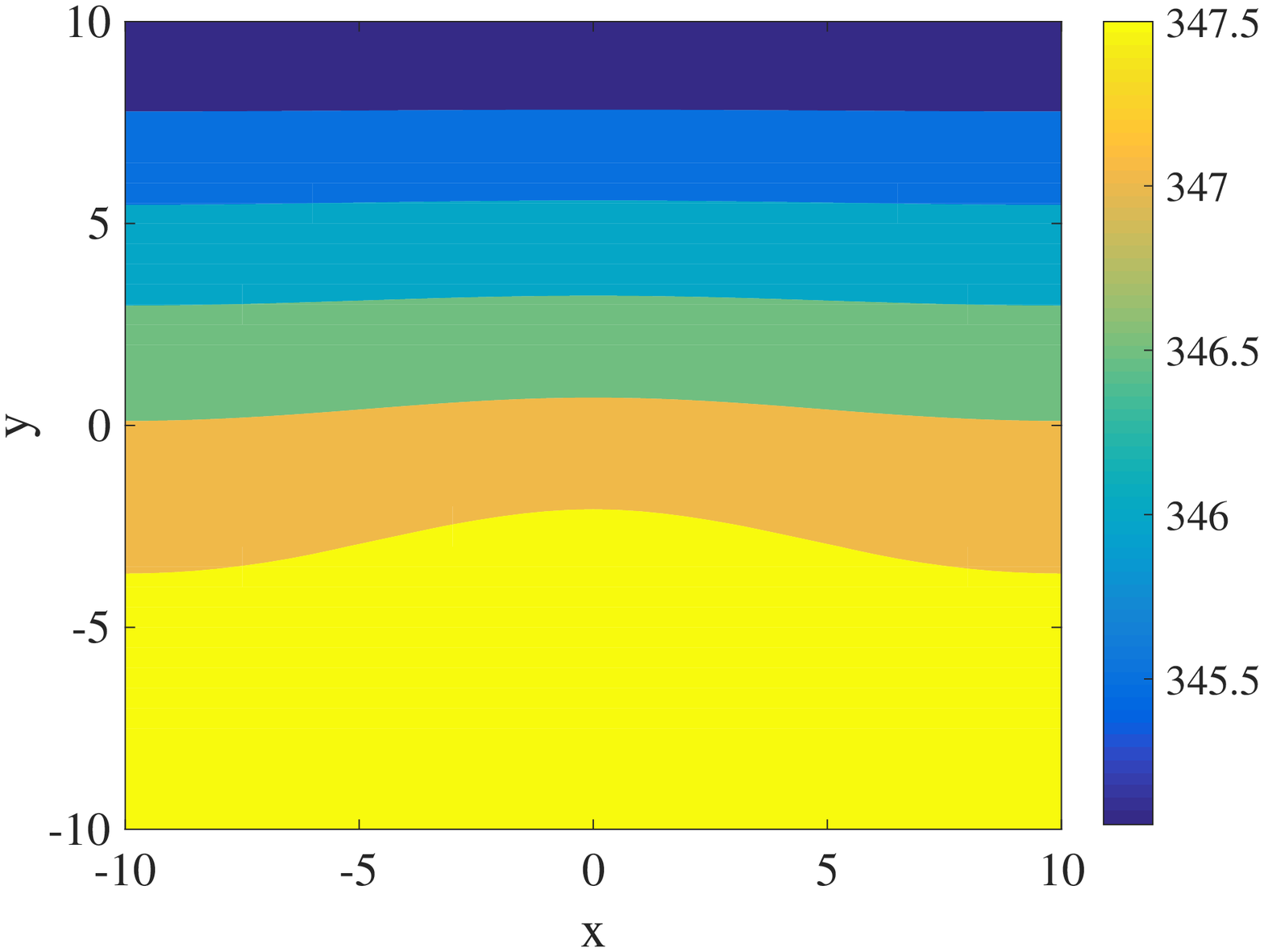}
            \end{minipage}
            }
             \centering \subfigure[]{
            \begin{minipage}[b]{0.3\textwidth}
            \centering
             \includegraphics[width=0.95\textwidth,height=1.35in]{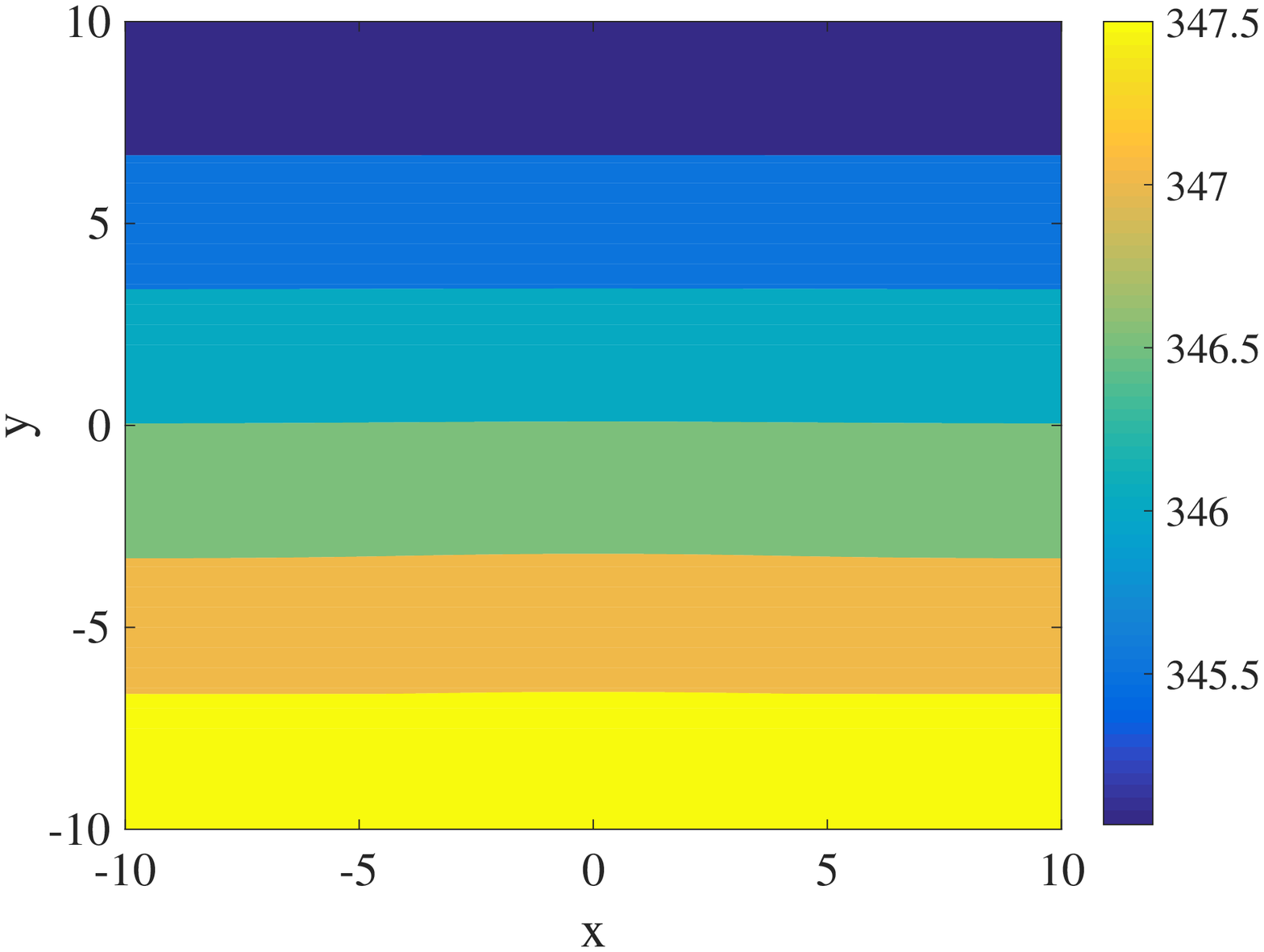}
            \end{minipage}
            }
           \caption{Bubble dropping problem: the temperature profiles at the initial time(a), 1000th(b), 20000th(c), 28000th(d), 29000th(e), 30000th(f), 30500th(g), 32000th(h) and 50000th(i) time step respectively.}
            \label{BubbleDroppingnC4TemperatureOfnC4}
 \end{figure}

\begin{figure}
            \centering \subfigure[]{
            \begin{minipage}[b]{0.31\textwidth}
               \centering
             \includegraphics[width=0.95\textwidth,height=1.35in]{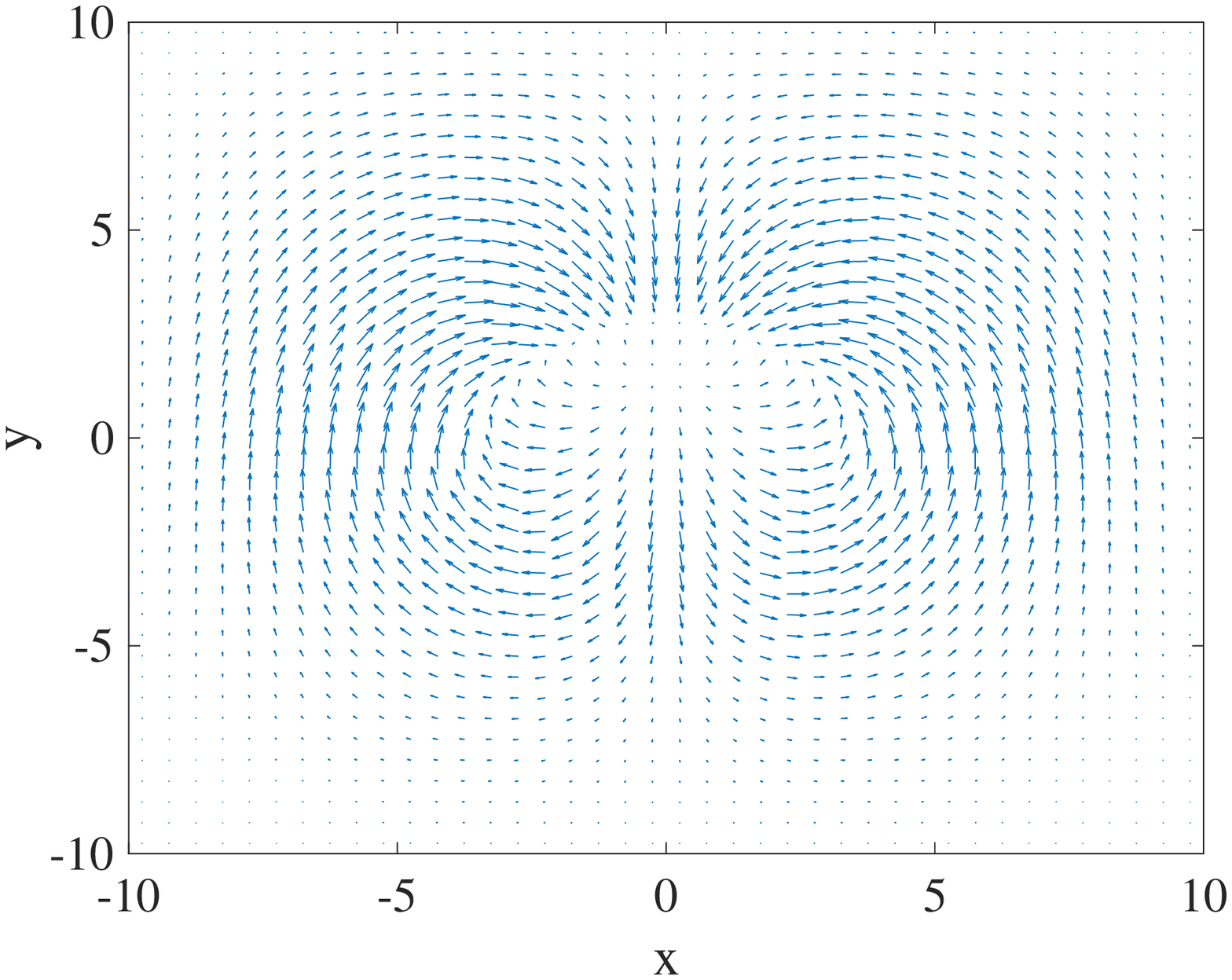}
            \end{minipage}
            }
            \centering \subfigure[]{
            \begin{minipage}[b]{0.31\textwidth}
            \centering
             \includegraphics[width=0.95\textwidth,height=1.35in]{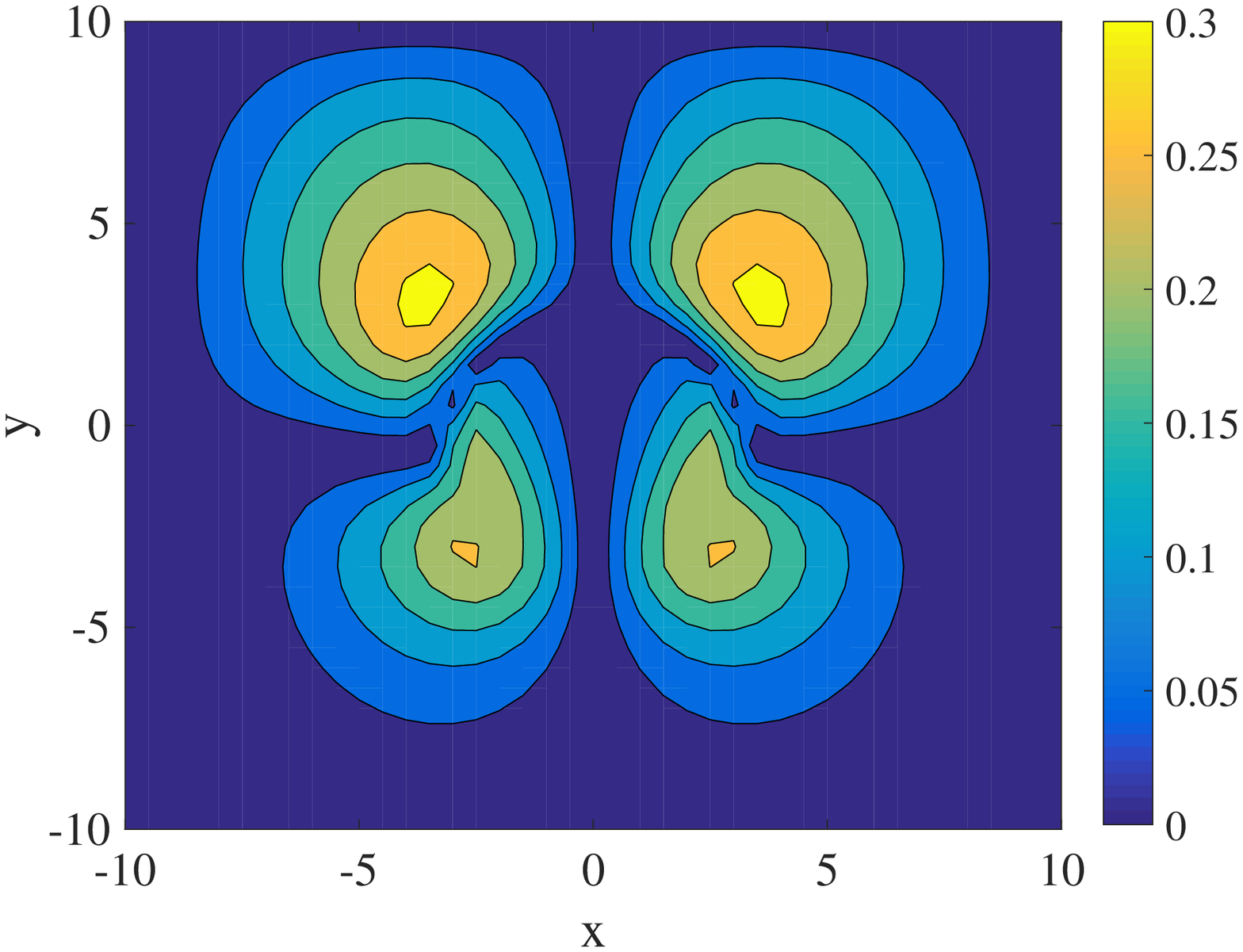}
            \end{minipage}
            }
             \centering \subfigure[]{
            \begin{minipage}[b]{0.31\textwidth}
            \centering
             \includegraphics[width=0.95\textwidth,height=1.35in]{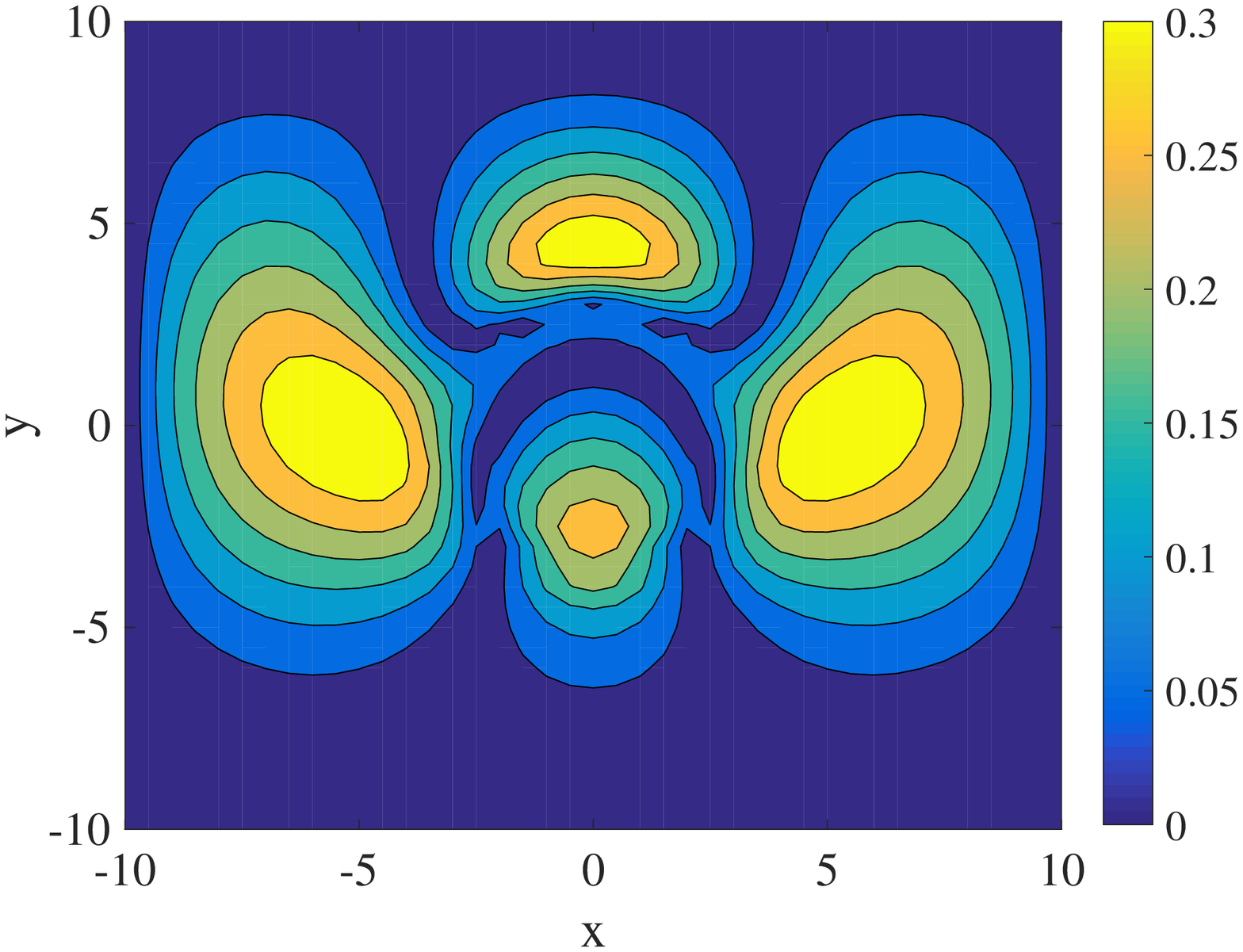}
            \end{minipage}
            }
           \centering \subfigure[]{
            \begin{minipage}[b]{0.31\textwidth}
               \centering
             \includegraphics[width=0.95\textwidth,height=1.35in]{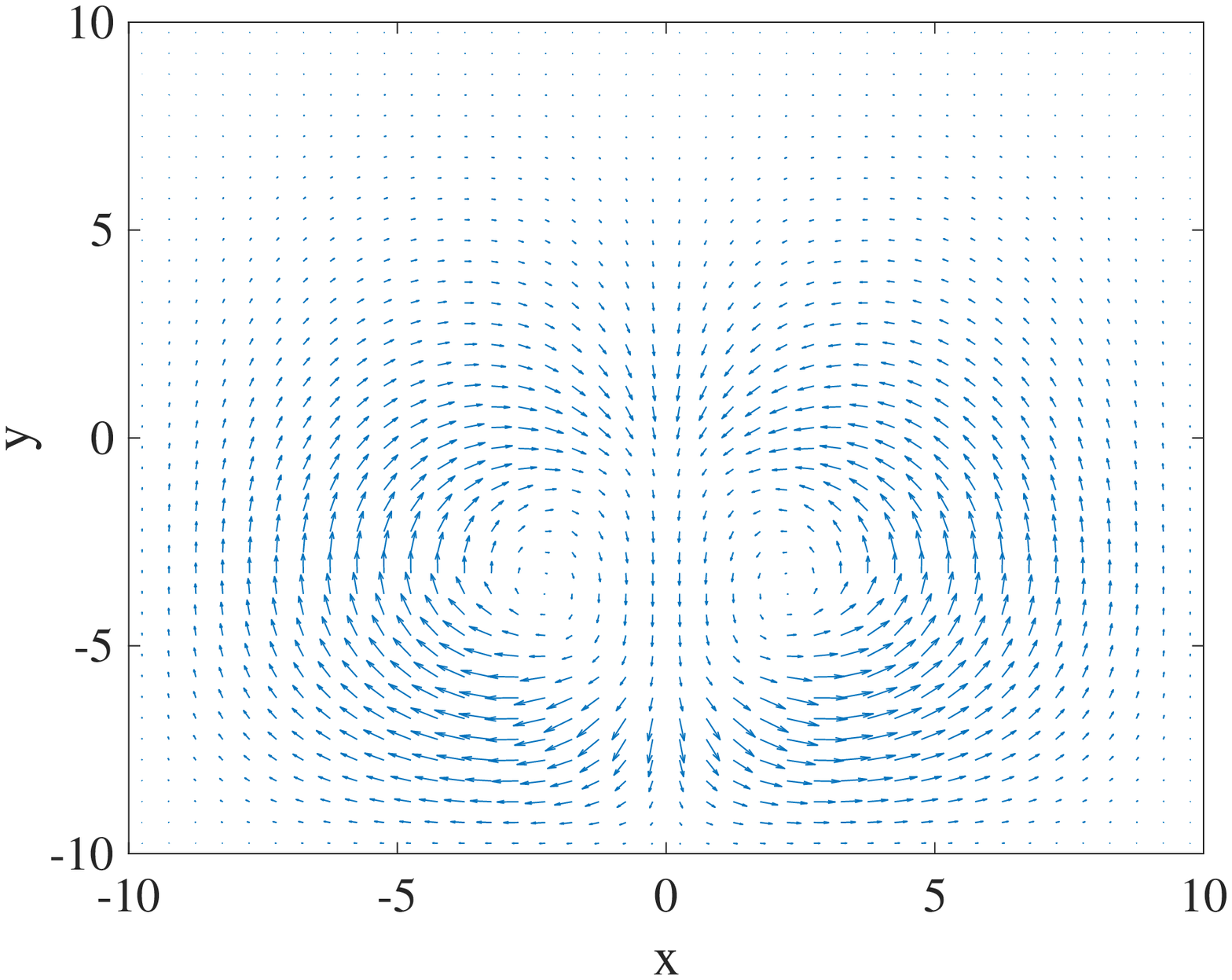}
            \end{minipage}
            }
            \centering \subfigure[]{
            \begin{minipage}[b]{0.3\textwidth}
            \centering
             \includegraphics[width=0.95\textwidth,height=1.35in]{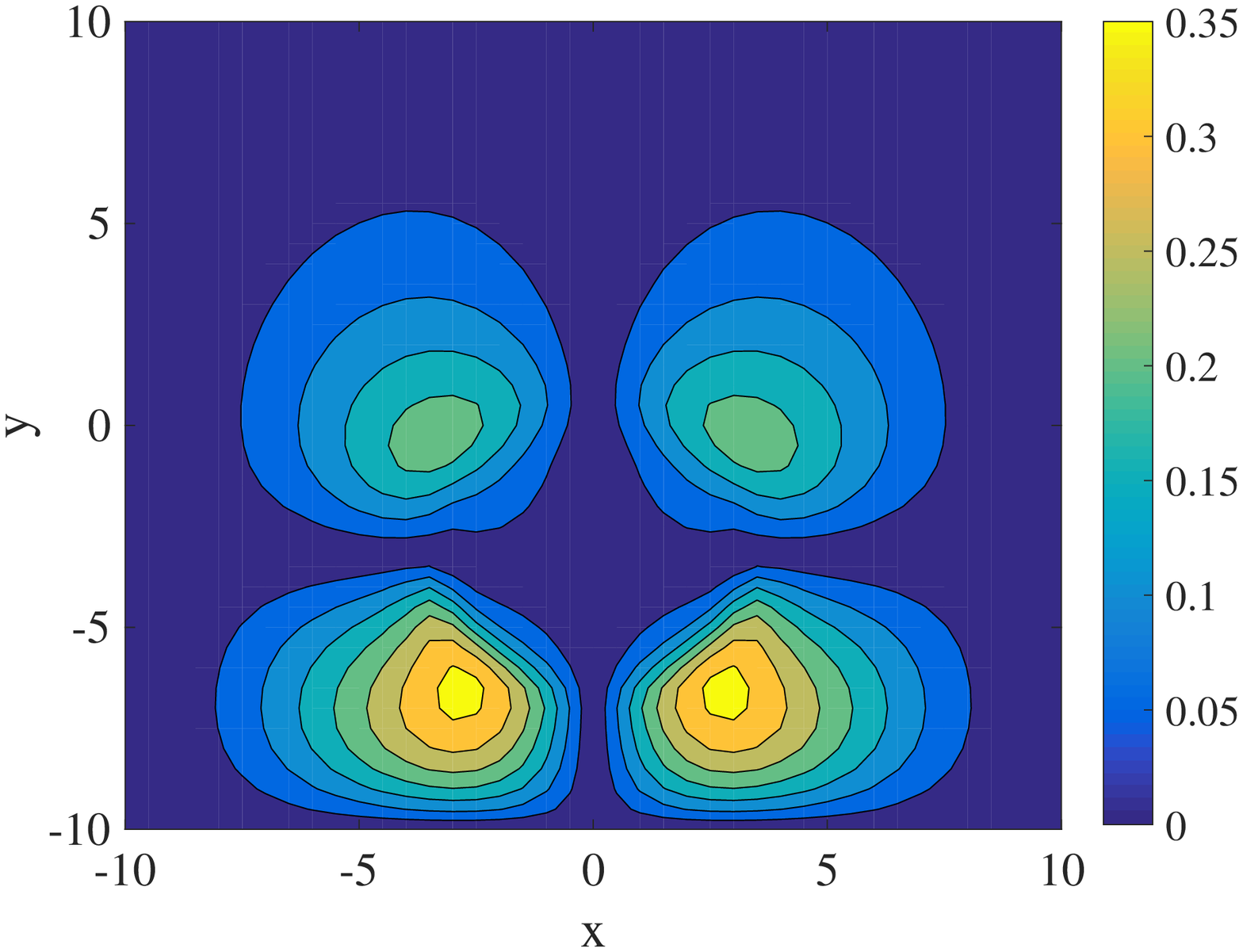}
            \end{minipage}
            }
             \centering \subfigure[]{
            \begin{minipage}[b]{0.3\textwidth}
            \centering
             \includegraphics[width=0.95\textwidth,height=1.35in]{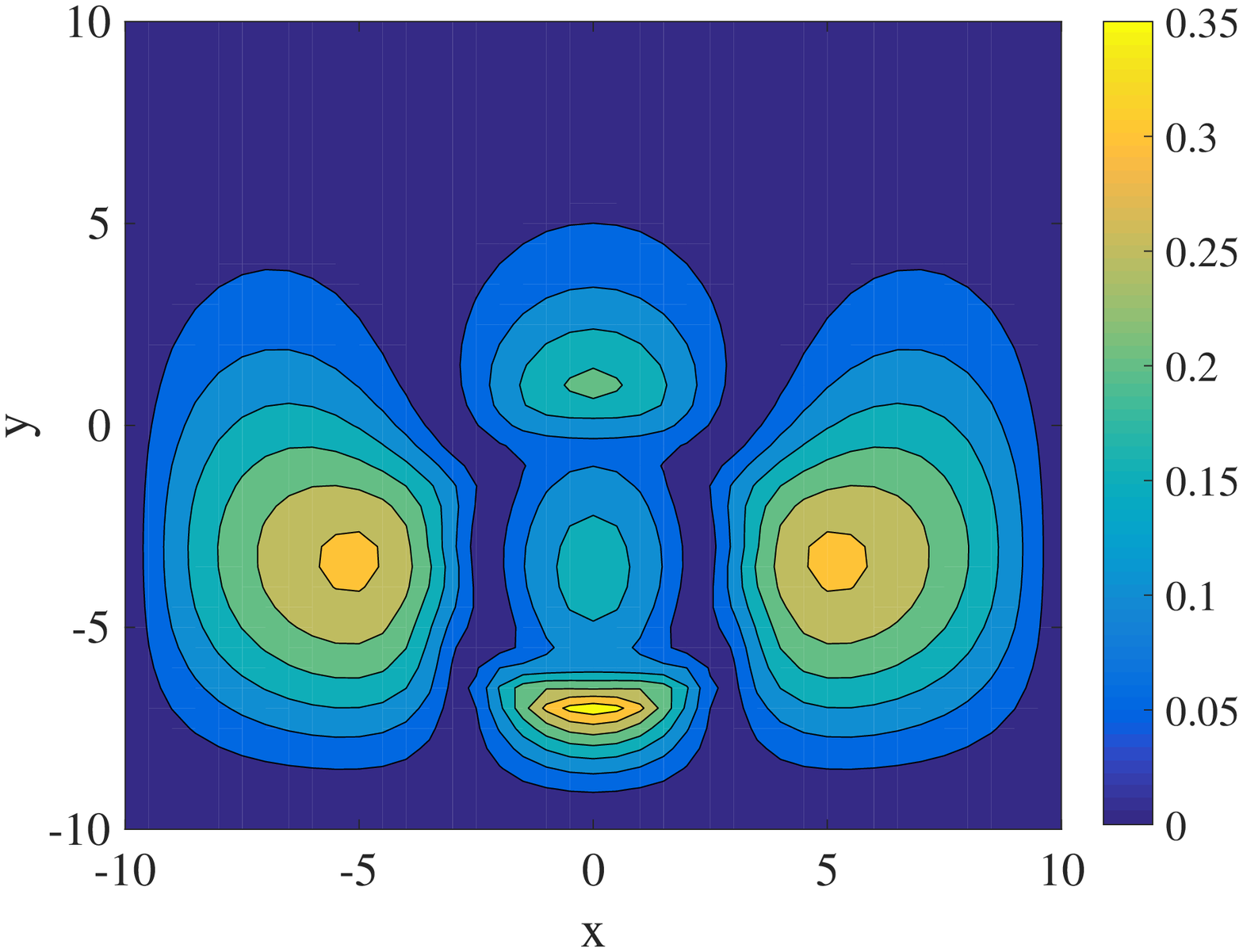}
            \end{minipage}
            }
           \centering \subfigure[]{
            \begin{minipage}[b]{0.3\textwidth}
               \centering
             \includegraphics[width=0.95\textwidth,height=1.35in]{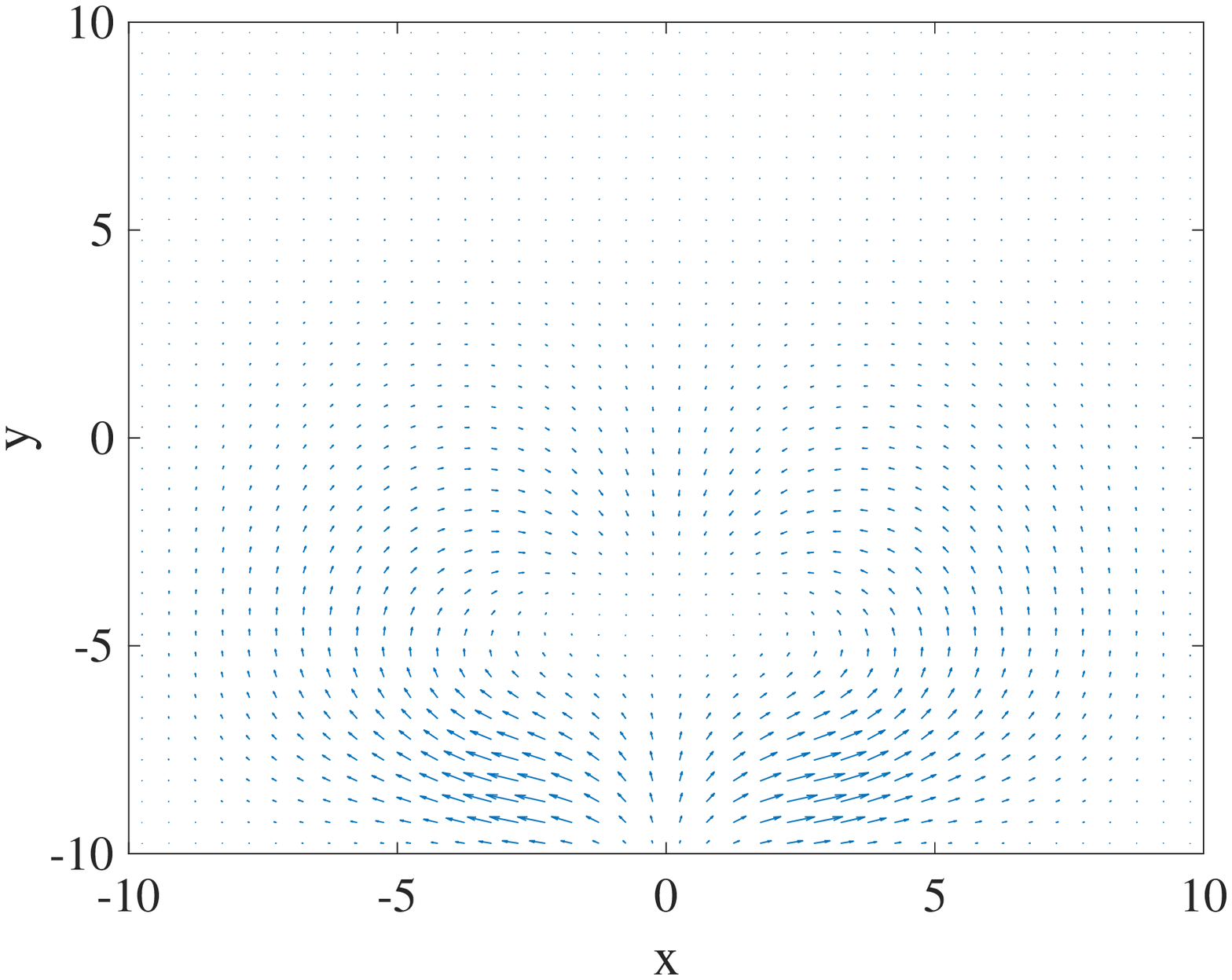}
            \end{minipage}
            }
            \centering \subfigure[]{
            \begin{minipage}[b]{0.3\textwidth}
            \centering
             \includegraphics[width=0.95\textwidth,height=1.35in]{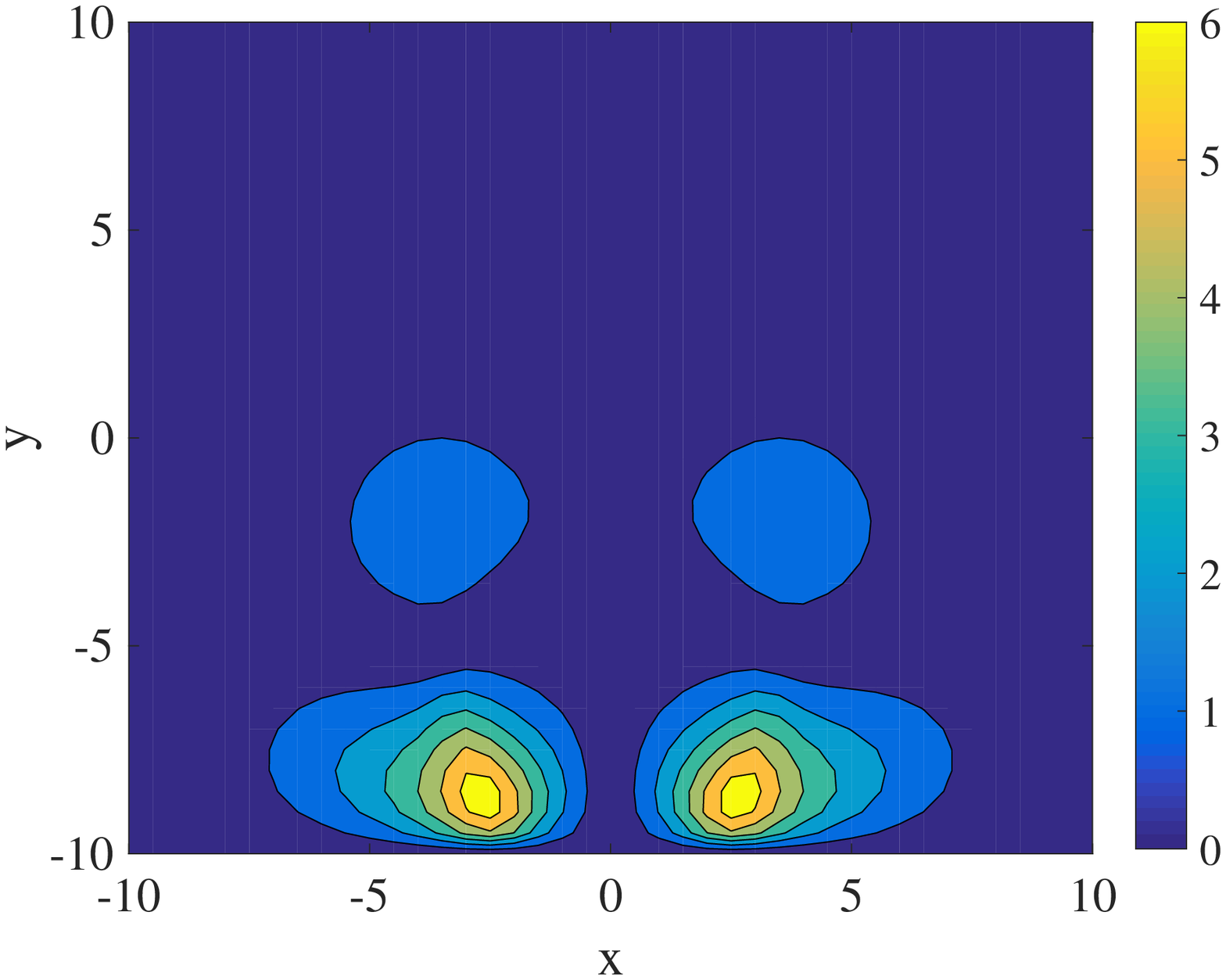}
            \end{minipage}
            }
             \centering \subfigure[]{
            \begin{minipage}[b]{0.3\textwidth}
            \centering
             \includegraphics[width=0.95\textwidth,height=1.35in]{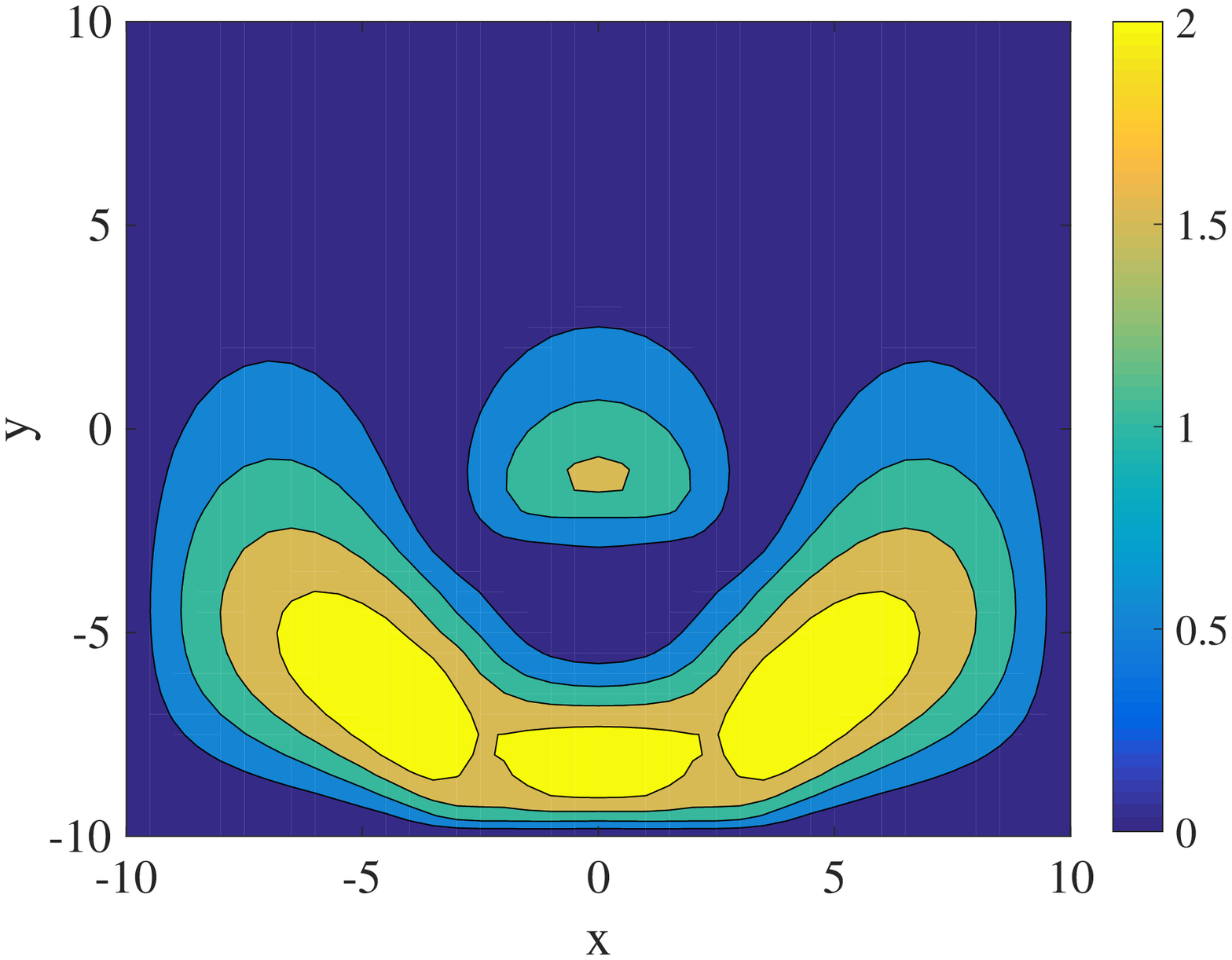}
            \end{minipage}
            }
           \centering \subfigure[]{
            \begin{minipage}[b]{0.3\textwidth}
               \centering
             \includegraphics[width=0.95\textwidth,height=1.35in]{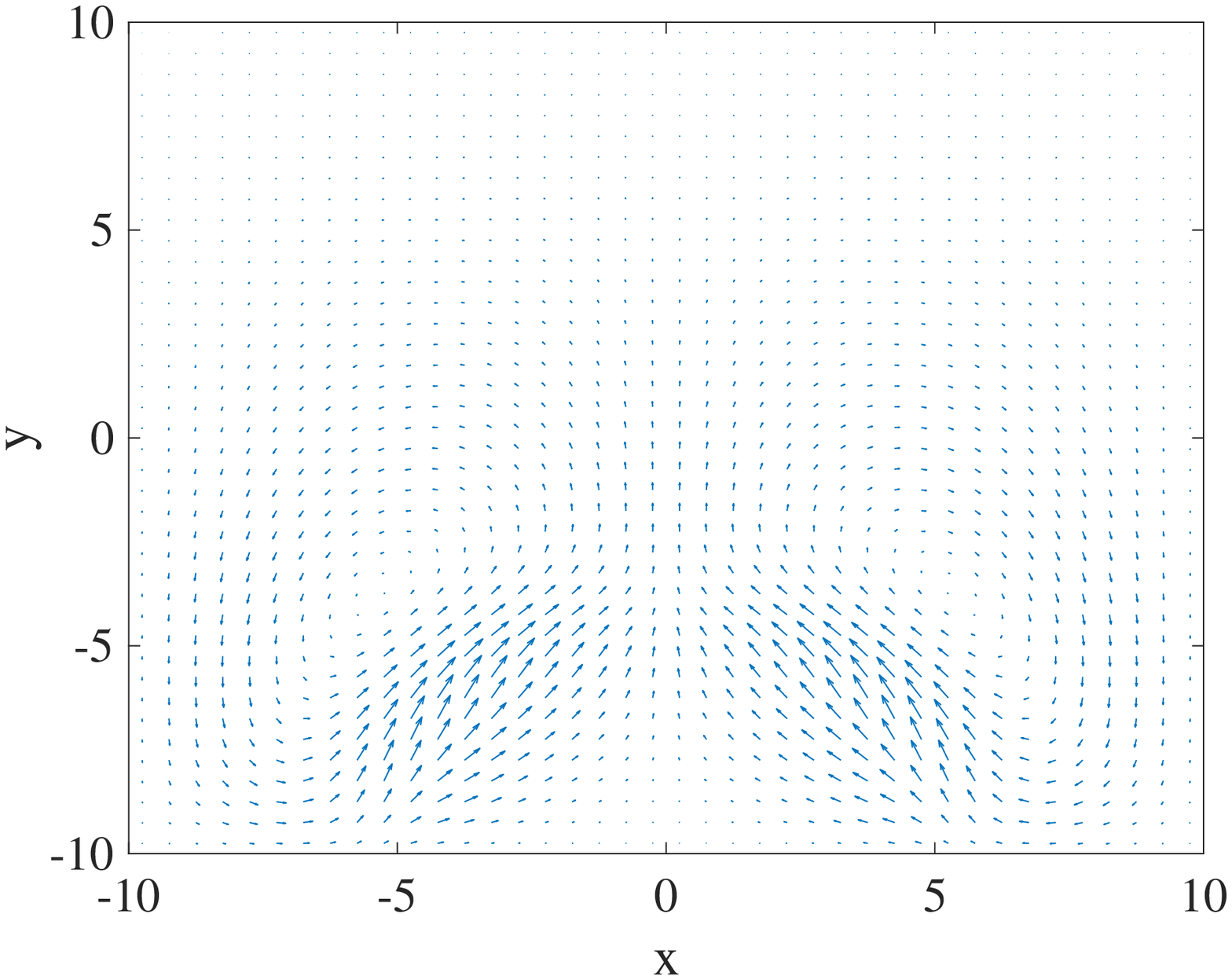}
            \end{minipage}
            }
            \centering \subfigure[]{
            \begin{minipage}[b]{0.3\textwidth}
            \centering
             \includegraphics[width=0.95\textwidth,height=1.35in]{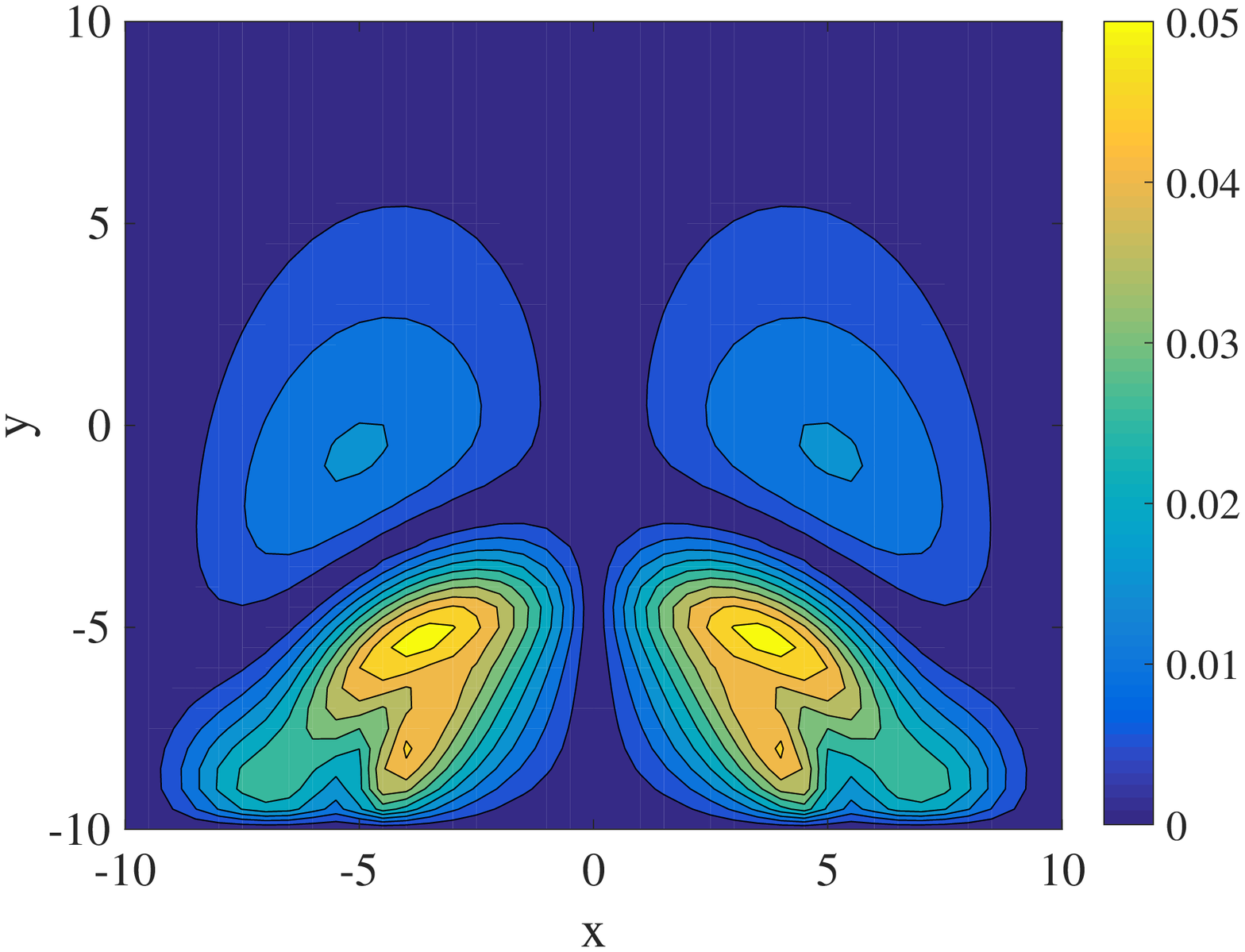}
            \end{minipage}
            }
             \centering \subfigure[]{
            \begin{minipage}[b]{0.3\textwidth}
            \centering
             \includegraphics[width=0.95\textwidth,height=1.35in]{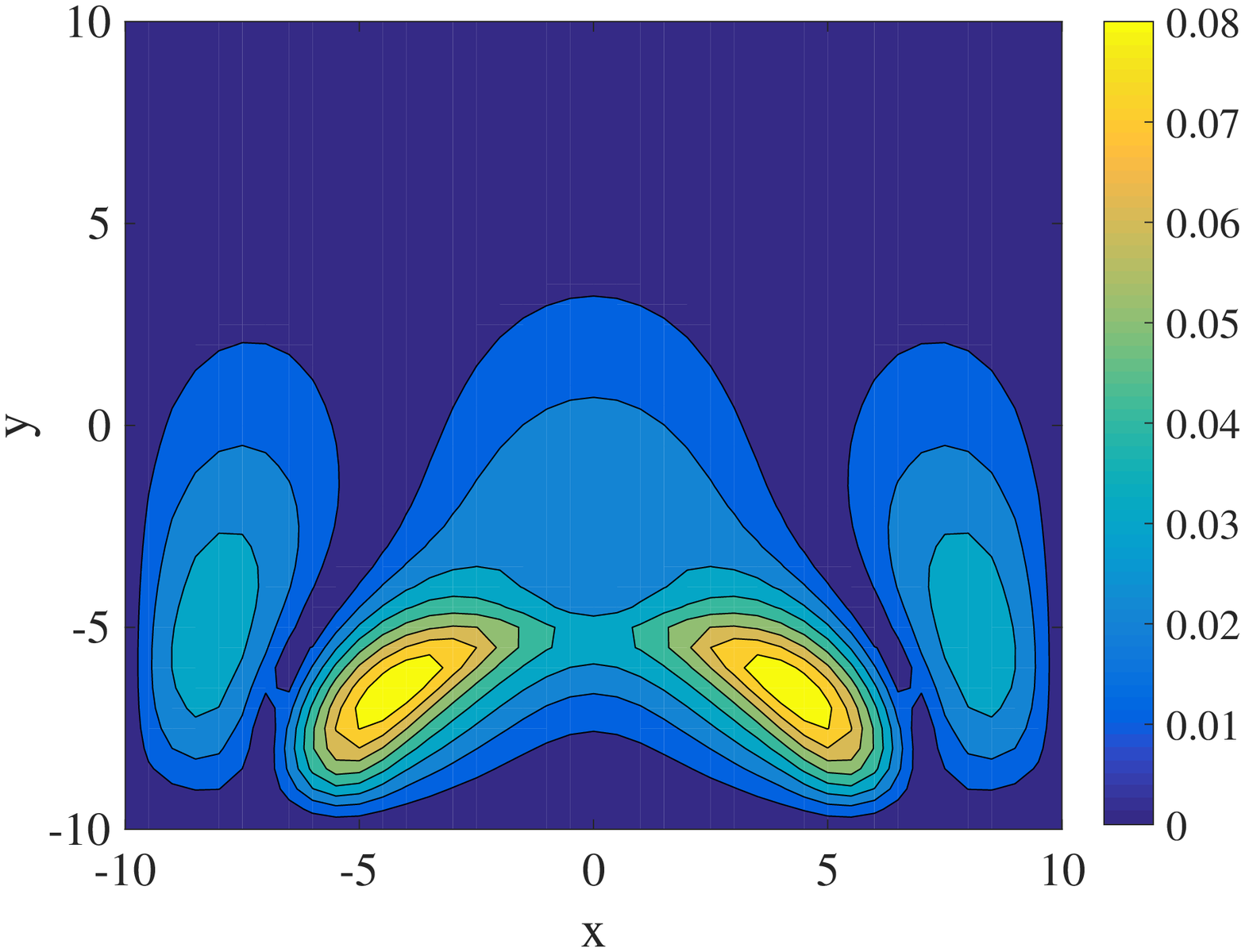}
            \end{minipage}
            }
           \caption{Bubble dropping problem:  the flow quivers (left column), magnitude contours of x-direction velocity component (center column), and magnitude contours of y-direction velocity component (right column)   at the 1000th(the first row), 20000th(the second row), 30000th(the third row), and 50000th(the bottom row) time step  respectively.}
           \label{BubbleDropingnC4Velocity}
 \end{figure}

%%%%%%%%%%%%%% Conclusions %%%%%%%%%
\section{Conclusions}

%%%%%%%%%%%%%%%%%%%%%%%%%%%%%%%%%%%%%%%%

We have studied modeling and numerical simulation of a diffuse-interface model of   gas-liquid two-phase   flow in an inhomogeneous temperature field. 
It is different from the existing models that we employ the Peng-Robinson equation of state instead of the van der Waals equation of state, and use  a realistic temperature-dependent influence parameter in the gradient contribution of Helmholtz free energy density.   As a result, this model is capable of  describing   physical behaviors of numerous realistic gas-liquid fluids accurately, such as  N$_2$, CO$_2$ and hydrocarbons. 

In order to resolve the difficulty resulting from  the complicate  form of thermodynamical pressure,  we prove a relation associating     the pressure gradient with the gradients of temperature and  chemical potential.  Using this  relation, we reformulate the model equations, which is beneficial to theoretical analysis and numerical simulation.  The new formulation of  momentum equation  shows  that  chemical potential and temperature gradients  become the primary driving force of the fluid motion.  By the new  formulations,  we prove that the model obey the first and second laws of thermodynamics.  

 To design efficient numerical time schemes,  we prove that the bulk contribution of Helmholtz free energy density   is a concave function with respect to the temperature and  its gradient contribution is also concave with respect to the temperature under certain conditions.  Based on the proposed modeling formulations,       we propose a novel thermodynamically consistent   numerical scheme by applying the convex-concave  splitting of Helmholtz free energy density. 
 The proposed scheme also utilizes  an auxiliary velocity, which depends on molar density and temperature, to alleviate the nonlinear coupling relation between molar density, velocity and temperature.  Furthermore,   a decoupled, linearized iterative method is developed for solving the discrete equations. 
It is also proved with a mathematical rigor that  the proposed time-marching scheme  satisfies the first and second laws of thermodynamics. 
Using the proposed numerical method, we have carried out a series of numerical tests and investigate the simulation results.

\small
%\bibliographystyle{plain}
%\bibliography{References}

\end{document}